\documentclass{amsart}

\usepackage{cases}
\usepackage{graphicx}
\usepackage{xargs}   
\usepackage{mathrsfs}
\usepackage{amssymb}
\usepackage{enumitem}
\usepackage{array}
\usepackage{caption} 
\renewcommand\star{\ast}

\usepackage{tikz}
	\usetikzlibrary{graphs,arrows,decorations.pathmorphing,backgrounds,positioning,fit,chains,matrix,scopes,decorations.markings}

\tikzset{-<-/.style={decoration={ 
  markings,
  mark=at position .5 with {\arrowreversed[line width = 0.5pt]{angle 90}}},postaction={decorate}}}

\tikzset{->-/.style={decoration={ 
  markings,
  mark=at position .5 with {\arrow[line width = 0.5pt]{angle 90}}},postaction={decorate}}}

\newtheorem{theorem}{Theorem}[section]
\newtheorem{conjecture}[theorem]{Conjecture}
\newtheorem{lemma}[theorem]{Lemma}
\newtheorem{proposition}[theorem]{Proposition}
\newtheorem{corollary}[theorem]{Corollary}
\newtheorem{case}{Case}

\theoremstyle{definition}
\newtheorem{definition}[theorem]{Definition}

\theoremstyle{remark}
\newtheorem{remark}[theorem]{Remark}

\numberwithin{equation}{section}

\usepackage{scalerel}
\usepackage{stackengine}
\stackMath

\newcommand\reallywidehat[1]{%
\savestack{\tmpbox}{\stretchto{%
  \scaleto{%
    \scalerel*[\widthof{\ensuremath{#1}}]{\kern-.6pt\bigwedge\kern-.6pt}%
    {\rule[-\textheight/2]{1ex}{\textheight}}
  }{\textheight}%
}{0.5ex}}%
\stackon[1pt]{#1}{\tmpbox}%
}

\newcommand{\Ad}{\text{Ad}} 

\newcommand{\Bcs}{B_{\mathbf{c},\mathbf{s}}} 
\newcommand{\barinv}{\overline{\phantom{m}}^U} 
\newcommand{\Bbarinv}{\overline{\phantom{m}}^B} 
\newcommand{\Bbarinvo}{\overline{\phantom{m}}^{B_{\bc,\bo}}} 
\newcommand{\B}[1]{\mathbf{#1}} 
\newcommand{\bc}{{\mathbf c}} 
\newcommand{\bs}{{\mathbf s}} 
\newcommand{\bo}{{\mathbf 0}} 
\newcommand{\Bco}{B_{\bc,\bo}} 

\newcommand{\cM}{\mathcal{M}} 
\newcommand{\cS}{\mathcal{S}} 
 
\newcommand{\field}{\mathbb{K}} 

\newcommand{\ir}[2]{{}_{#1}r(#2)} 

\newcommand{\lie}[1]{\mathfrak{#1}} 

\newcommand{\N}{\mathbb{N}} 

\newcommand{\os}{\widetilde{s}} 

\newcommand{\q}{(q-q^{-1})} 
\newcommand{\Q}[1]{(q^{#1} - q^{-#1})} 
\newcommand{\Qkm}{\mathfrak{X}} 

\newcommand{\ri}[2]{r_{#1}(#2)} 
\newcommand{\rank}{\mathrm{rank}} 
\newcommand{\rb}[1]{\raisebox{-\totalheight}{#1}} 

\newcommand{\satake}{(I,X,\tau)} 
\newcommand{\subsatake}{(J,X \cap J, \tau|_J)} 
\newcommand{\sA}{\mathscr{A}} 
\newcommand{\sU}{\mathscr{U}} 

\newcommand{\T}[1]{T_{#1}} 

\newcommand{\uqg}{U_q(\lie{g})} 

\newcommand{\widet}{\widetilde} 

\newcommand{\Z}{\mathbb{Z}} 

\title{Factorisation of quasi $K$-matrices for quantum symmetric pairs}

\author{Liam Dobson}
\address{Liam Dobson and Stefan Kolb, School of Mathematics, Statistics and Physics, Newcastle University, Herschel Building, Newcastle upon Tyne NE1 7RU, UK}
\email{l.dobson1@newcastle.ac.uk}
\email{stefan.kolb@newcastle.ac.uk}

\author{Stefan Kolb}



\begin{document}

\begin{abstract}
The theory of quantum symmetric pairs provides a universal $K$-matrix which is an analogue of the universal $R$-matrix for quantum groups. The main ingredient in the construction of the universal $K$-matrix is a quasi $K$-matrix which has so far only been constructed recursively. In this paper we restrict to the cases where the underlying Lie algebra is $\mathfrak{sl}_n$ or the Satake diagram has no black dots. In these cases we give an explicit formula for the quasi $K$-matrix as a product of quasi $K$-matrices for Satake diagrams of rank one. This factorisation depends on the restricted Weyl group of the underlying symmetric Lie algebra in the same way as the factorisation of the quasi $R$-matrix depends on the Weyl group of the Lie algebra. We conjecture that our formula holds in general.
\end{abstract}

\maketitle

\section{Introduction}
\subsection{Background}
Let $\lie{g}$ be a complex semisimple Lie algebra and $U_q(\lie{g})$ the corresponding Drinfeld-Jimbo quantised enveloping algebra with positive and negative parts $U^+$ and $U^-$, respectively. The quasi $R$-matrix  for $\uqg$ is a canonical element in a completion of $U^-\otimes U^+$ which plays a pivotal role in many applications of quantum groups. In the theory of canonical or crystal bases for $\uqg$ developed by G.~Lusztig and M.~Kashiwara, the quasi $R$-matrix appears as an intertwiner of two bar involutions on $\Delta(\uqg)$, where $\Delta$ denotes the coproduct of $\uqg$. The quasi $R$-matrix is used to define canonical bases of tensor products of $\uqg$-modules and of the modified quantised enveloping algebra $\dot{\mathbf{U}}$, see \cite[Part IV]{b-Lusztig94}.

Moreover, the quasi $R$-matrix for $\uqg$ is used in \cite[Chapter 32]{b-Lusztig94} to construct a family of commutativity isomorphisms. These maps turn suitable categories of $\uqg$-modules into braided monoidal categories and hence have applications in low-dimensional topology, in particular the construction of invariants of knots and links, see \cite{a-RT90}. Up to completion, the commutativity isomorphisms come from a universal $R$-matrix for $\uqg$.

As in \cite{a-BK15} we denote the quasi $R$-matrix of $\uqg$ by $R$. One of the key properties of $R$ is that it admits a factorisation as a product of quasi $R$-matrices for $U_q(\lie{sl}_2)$. Let $\{\alpha_i\,|\,i\in I\}$ be the set of simple roots for $\lie{g}$. The quasi $R$-matrix corresponding to $i\in I$ is given by 
\begin{equation}\label{eq:Ri}
R_{i} = \sum_{r \geq 0} (-1)^r q_{i}^{-r(r-1)/2} \dfrac{(q_{i} - q_{i}^{-1})^r}{ [r]_{q_{i}}!} F_{i}^r \otimes E_{i}^r
\end{equation}
where $E_i, F_i\in \uqg$ are generators of the copy of $U_q(\lie{sl}_2)$ labelled by $i$, and $q_i=q^{(\alpha_i,\alpha_i)/2}$. Let $s_i$ for $i\in I$ be the generators of the Weyl group $W$ of $\lie{g}$, and let $T_i: \uqg \rightarrow \uqg$ for $i\in I$ be the corresponding Lusztig automorphisms. For any reduced expression $w_0 = s_{i_1} \dotsm s_{i_t}$ of the longest word $w_0 \in W$ define
\begin{equation*}
  	R^{[j]}=\big(T_{i_1}\dots T_{i_{j-1}}\otimes T_{i_1}\dots T_{i_{j-1}}\big)(R_{i_j}) \qquad \mbox{for $j=1, \dots,t$.}
\end{equation*}
With this notation the quasi $R$-matrix of $\uqg$ can be written as
\begin{equation}\label{eq:R-factor}
	R = R^{[t]} \cdot R^{[t-1]} \cdots R^{[2]} \cdot R^{[1]},
\end{equation}        
see \cite{a-LS90}, \cite{a-KR90}, \cite[8.30]{b-Jantzen96}.

In the present paper we aim to find a similar factorisation in the theory of quantum symmetric pairs. Let $\theta:\lie{g}\rightarrow \lie{g}$ be an involutive Lie algebra automorphism and let $\lie{k}=\{x\in \lie{g}\,|\,\theta(x)=x\}$ be the corresponding fixed Lie subalgebra. We refer to $(\lie{g},\lie{k})$ as a symmetric pair. A comprehensive theory of quantum symmetric pairs was developed by G.~Letzter in \cite{a-Letzter99a}, \cite{MSRI-Letzter}. This theory provides  families of subalgebras $B_{\bc,\bs}\subset U_q(\lie{g})$ with parameters $\bc$ and $\bs$, which are quantum group analogues of $U(\lie{k})$. Crucially, $B_{\bc,\bs}$ is a right coideal subalgebra of $U_q(\lie{g})$, that is
\begin{align*}
  \Delta(B_{\bc,\bs})\subseteq B_{\bc,\bs}\otimes \uqg.
\end{align*}  
Initially, the theory of quantum symmetric pairs was used to perform harmonic analysis on quantum group analogues of symmetric spaces, see \cite{a-Noumi96}, \cite{a-Letzter04}. Recent pioneering work by H.~Bao and W.~Wang \cite{a-BaoWang13} and by M.~Ehrig and C.~Stroppel \cite{a-ES13} has placed quantum symmetric pairs in a much broader representation theoretic context. Both papers consider a bar involution for specific quantum symmetric pair coideal subalgebras of type $AIII$/$AIV$. Moreover, Bao and Wang construct an intertwiner $\Qkm$ (denoted by $\Upsilon$ in \cite{a-BaoWang13}) between the bar involutions on $B_{\bc,\bs}$ and on $\uqg$. The intertwiner $\Qkm$ is an analogue for quantum symmetric pairs of the quasi $R$-matrix for $\uqg$. Using the intertwiner, Bao and Wang show that large parts of Lusztig's theory of canonical bases \cite[Part IV]{b-Lusztig94} extend to the setting of quantum symmetric pairs. While \cite{a-BaoWang13} restricts to the specific quantum symmetric pairs of type $AIII$/$AIV$, the more recent work \cite{a-BW16} develops a theory of canonical bases for all quantum symmetric pairs of finite type.

Following the program outlined in \cite{a-BaoWang13}, the existence of the bar involution and the intertwiner $\Qkm$ was proved for general quantum symmetric pairs in \cite{a-BK14} and \cite{a-BK15}, respectively. The intertwiner  
$\Qkm$ was then used in \cite{a-BK15} to construct a universal $K$-matrix for $B_{\bc,\bs}$ which is an analogue of the universal $R$-matrix for $\uqg$.  For this reason we call the intertwiner $\Qkm$ the \textit{quasi $K$-matrix for} $B_{\bc,\bs}$. The universal $K$-matrix gives suitable categories of $B_{\bc,\bs}$-modules the structure of  a braided module category and allows similar applications in low-dimensional topology as braided monoidal categories, see \cite{a-Kolb17p}. 

In \cite{a-BaoWang13} and \cite{a-BK15} the quasi $K$-matrix is defined recursively by the intertwiner property for the two bar involutions. It is noted at the end of \cite[Section 4.4]{a-BW16} that this recursion can be solved in principle. However, to this date no closed formula for $\Qkm$ is known.

\subsection{Results}
In the present paper we provide a general closed formula for the quasi $K$-matrix $\Qkm$ in many cases, and we conjecture that our formula holds for all quantum symmetric pairs of finite type. In doing so, we take guidance from the known factorisation \eqref{eq:R-factor} of the quasi $R$-matrix of $\uqg$. 

Recall that the involutive automorphism $\theta$ is determined up to conjugation by a Satake diagram $(I,X,\tau)$. Here, $X$ is a subset of $I$ and $\tau:I\rightarrow I$ is a diagram automorphism. The $\tau$-orbits in $I\setminus X$ correspond to rank one subdiagrams of the Satake diagram $(I,X,\tau)$.
Associated to the symmetric Lie algebra $(\lie{g},\theta)$ is a restricted root system $\Sigma$ with Weyl group $\widet{W}=W(\Sigma)$ which can be considered as a subgroup of $W$. The Coxeter generators $\widet{s_i}$ of $\widet{W}$ are parametrised by the $\tau$-orbits in $I\setminus X$. We introduce the notion of a partial quasi $K$-matrix $\Qkm_{\widet{w}}$ for any $\widet{w}\in \widet{W}$ with a reduced expression $\widet{w}=\widet{s_{i_1}}\dots \widet{s_{i_t}}$. More precisely, for $j=1, \dots, t$ let $\Qkm_i$ denote the quasi $K$-matrix corresponding to the rank one Satake subdiagram $(\{i,\tau(i)\}\cup X, X, \tau|_{\{i,\tau(i)\}\cup X})$ of $(I,X,\tau)$. The Lusztig automorphisms $T_w: \uqg \rightarrow  \uqg$ for $w\in W$ allow us to define automorphisms $\widet{T_i}:=T_{\widet{s_i}}$ for all $i\in I\setminus X$. For $j=1,\dots,t$ we set
\begin{equation*}
\Qkm_{\widet{w}}^{[j]} = \Psi \circ \widet{T_{i_1}} \dotsm  \widet{T_{i_{j-1}}} \circ \Psi^{-1} (\Qkm_{i_j})
\end{equation*}
where $\Psi$ denotes an algebra automorphism of an extension of $\widet{U}^+=\bigoplus_{\mu\in Q^+(2\Sigma)} U^+_\mu$ defined in Equation \eqref{eq:Psi}. In analogy to \eqref{eq:R-factor} we now define the partial quasi K-matrix corresponding to $\widet{w}\in \widet{W}$ by
\begin{equation}\label{eq:Qkm-def}
\Qkm_{\widet{w}} = \Qkm_{\widet{w}}^{[t]} \cdot \Qkm_{\widet{w}}^{[t-1]}  \dotsm  \Qkm_{\widet{w}}^{[2]} \cdot \Qkm_{\widet{w}}^{[1]}.
\end{equation}
The following theorem is the main result of the present paper. It gives an explicit formula for $\Qkm$ in the case that $\bs=\bo=(0,0,\dots,0)$ for certain Satake diagrams. Let $\widet{w_0}\in \widet{W} $ denote the longest element.

\medskip

\noindent{\bf Theorem A.} (Corollary \ref{cor:Qkm}) \textit{Let $\lie{g}$ be of type $A_n$ or $X = \emptyset$. Then the quasi $K$-matrix $\Qkm$ for $B_{\bc,\bo}$ is given by $\Qkm = \Qkm_{\widet{w_0}}$ for any reduced expression of the longest element $\widet{w_0} \in \widet{W}$.}

\medskip

We conjecture that Theorem A holds for all Satake diagrams of finite type, see Conjecture \ref{conj:anyType}. The proof of Theorem A proceeds in three steps. First we construct the quasi $K$-matrices corresponding to all rank one Satake diagrams of type $A_n$ in the case where $\bs=\bo$. The difficulty here is that there are many rank one cases, see Table \ref{Table:RankOne}. Secondly, we prove Theorem A in rank two by direct calculation. In rank two the longest element $\widet{w_0}\in\widet{W}$ has two reduced expressions. We show for each reduced expression that the element $\Qkm_{\widet{w_0}}$ defined by  
\eqref{eq:Qkm-def} satisfies the defining recursive relations for the quasi $K$-matrix. This involves tedious calculations which we have banned to Appendix \ref{App:RankTwo}. The calculations require explicit knowledge of the rank one quasi $K$-matrices. The restriction to type $A_n$ or $X=\emptyset$ in Theorem A hence stems from the fact that the rank one quasi $K$-matrices are only known in type $A_n$.

In the cases considered in Appendix \ref{App:RankTwo}, the calculations imply in particular that $\Qkm_{\widet{w_0}}$ is independent of the chosen reduced expression for $\widet{w_0}$. We conjecture that this is true in general.

\medskip

\noindent{\bf Conjecture B.} (Conjecture \ref{conj}) \textit{Assume that $\satake$ is a Satake diagram of rank two. Then the element $\Qkm_{\widet{w}}$ defined by \eqref{eq:Qkm-def} depends only on $\widet{w} \in \widetilde{W}$ and not on the chosen reduced expression.}

\medskip

Conjecture B is all that is needed to prove Theorem A for all Satake diagrams of finite type. Indeed, assume that Conjecture B holds for all rank two Satake subdiagrams of the given Satake diagram. Then we can use braid relations for the operators $\widet{T_i}$ to show that $\Qkm_{\widet{w}}$ is independent of the chosen reduced expression for $\widet{w}\in \widet{W}$. In the case of the longest element $\widet{w_0}\in \widet{W}$ we choose different reduced expressions for $\widet{w_0}$ to show that $\Qkm_{\widet{w_0}}$ satisfies the defining recursive relations for the quasi $K$-matrix for $B_{\bc,\bo}$. In summary, we obtain the following result in the case $\bs=\bo$.

\medskip

\noindent{\bf Theorem C.} (Theorems \ref{Thm:HigherRankPartialQkm}, \ref{Thm:LongestWordPartialQkm}) \textit{Let $\satake$ be a Satake diagram such that all rank two Satake subdiagrams satisfy Conjecture B. Then the following hold:
\begin{enumerate}
  \item The partial quasi $K$-matrix $\Qkm_{\widet{w}}$ depends only on $\widet{w} \in \widet{W}$ and not on the chosen reduced expression.
  \item The quasi $K$-matrix $\Qkm$ for $B_{\bc,\bo}$ is given by $\Qkm=\Qkm_{\widet{w_0}}$ where $\widet{w_0}\in \widet{W}$ denotes the longest element.
\end{enumerate}
}
\medskip
In the case $\bs\neq \bo$ it is harder to give an explicit formula for the quasi $K$-matrix $\Qkm_{\bc,\bs}$ of $B_{\bc,\bs}$. However, we can make use of the fact that $\Bcs$ is obtained from $B_{\bc,\bo}$ via a twist by a character $\chi_\bs$ of $B_{\bc,\bo}$. We consider the element $R^\theta_{\bc,\bs}=\Delta(\Qkm_{\bc,\bs})R(\Qkm^{-1}_{\bc,\bs}\otimes 1)$ which was introduced in \cite{a-BaoWang13} under the name quasi $R$-matrix for $\Bcs$, and which lives in a completion of $\Bcs\otimes U^+$, see also \cite[Section 3.3]{a-Kolb17p}. We show that the quasi $K$-matrix $\Qkm_{\bc,\bs}$ for $\Bcs$ satisfies the relation
\begin{equation}\label{eq:Qkmsneq0}
  \Qkm_{\bc,\bs}=(\chi_{\bs}\otimes \mathrm{id})(R^\theta_{\bc,\bo}).
\end{equation}
Hence the explicit formulas \eqref{eq:R-factor} and  \eqref{eq:Qkm-def} for $R$ and $\Qkm_{\bc,\bo}$, respectively, provide a formula for the quasi $K$-matrix of $\Bcs$ also in the case $\bs\neq \bo$. However, in this case we do not obtain a factorisation as in Equation \eqref{eq:Qkm-def}.

\subsection{Organisation}
In Section \ref{sec:WSigma} we recall background material on the restricted root system $\Sigma$ and its Weyl group $W(\Sigma)$. We show in particular how $W(\Sigma)$ can be identified with a subgroup $\widet{W}$ of the Weyl group $W$ of $\lie{g}$, see Proposition \ref{proposition:ReflGrp}. Section \ref{sec:facK} forms the heart of the paper. In Sections \ref{sec:QEAs} and \ref{sec:QSPs} we fix notation for quantised enveloping algebras and quantum symmetric pairs, respectively. In Section \ref{sec:rank1Qkms} we determine explicit closed formulas for the quasi $K$-matrices $\Qkm$ for $B_{\bc,\bo}$ in all cases where the Satake diagram is of rank one and of type $A_n$. We plan to return to the remaining rank one cases in the future.  The theory of partial quasi $K$-matrices $\Qkm_{\widet{w}}$ for $\widet{w}\in \widet{W}$ is developed in Section \ref{sec:PartialQkms}. Here the rank two case is crucial. The explicit calculations which prove Theorem A in rank two are left for Appendix \ref{App:RankTwo}. Based on the rank two results we prove in Theorem \ref{Thm:HigherRankPartialQkm} that $\Qkm_{\widet{W}}$ is independent of the chosen reduced expression for $\widet{w}$. This gives the first part of Theorem C, and the second part as well as Theorem A follow, see Theorem \ref{Thm:LongestWordPartialQkm}. The case $\bs\neq \bo$ is treated in Section \ref{sec:sGeneral} where Formula \eqref{eq:Qkmsneq0} is proved in  Proposition \ref{prop:os}.


\section{The restricted Weyl group} \label{sec:WSigma}

In this section we recall the construction of involutive automorphisms $\theta: \lie{g} \rightarrow \lie{g}$ of a semisimple Lie algebra $\lie{g}$ following \cite{a-Kolb14}. This allows us to construct a subgroup $W^{\Theta}$ of the Weyl group $W$ consisting of elements fixed under the corresponding group automorphism of $W$. Of particular importance is a subgroup $\widet{W}$ of $W^{\Theta}$ which has an interpretation as the Weyl group of the corresponding restricted root system. The results in this section do not claim originality, but are all in some form contained in \cite{a-Lu76}, \cite{a-Lu02} and also \cite{a-GI14}. 

\subsection{Involutive automorphisms of semisimple Lie algebras} \label{sec:invaut}

Let $\lie{g}$ be a finite dimensional  complex semisimple Lie algebra.
Let $\lie{h} \subset \lie{g}$ be a Cartan subalgebra and $\Phi \subset \lie{h}^{\ast}$ the corresponding root system. Choose a set of simple roots $\Pi = \{ \alpha_i \mid i \in I\}$ where $I$ is an index set labelling the nodes of the Dynkin diagram of $\lie{g}$. Let $Q=\Z\Pi$ be the root lattice and set $Q^+=\N_0\Pi$. Let $\Phi^{+}$ be the set of positive roots corresponding to $\Pi$ and set $V = \mathbb{R}\Phi$. For $i \in I$, let $s_i : V \rightarrow V$ denote the reflection at the hyperplane orthogonal to $\alpha_i$. We write $W$ to denote the Weyl group generated by the simple reflections $s_i$. We fix the $W$-invariant non-degenerate bilinear form $(-,-)$ on $V$ such that $(\alpha,\alpha)= 2$ for all short roots $\alpha \in \Phi$ in each component.
Let $\{e_i, f_i, h_i \mid i \in I\}$ be the Chevalley generators for $\lie{g}$.

Involutive automorphisms of $\lie{g}$ are classified up to conjugation by pairs $(X, \tau)$ where $X \subset I$ and $\tau: I \rightarrow I$ is a diagram automorphism.
More precisely, for any subset $X \subset I$ let $\lie{g}_X$ denote the Lie subalgebra of $\lie{g}$ generated by $\{e_i, f_i, h_i \mid i \in X\}$.
Let $Q_X$ denote the sublattice of $Q$ generated by $\{ \alpha_i \mid i \in X\}$. This is the root lattice for $\lie{g}_X$.
Let $\rho_X \in V$ and $\rho_X^{\vee} \in V^{\ast}$ denote the half sum of positive roots and positive coroots for $\lie{g}_X$.
The Weyl group $W_X$ of $\lie{g}_X$ is the parabolic subgroup of $W$ generated by $\{s_i \mid i \in X\}$.
Let $w_X \in W_X$ denote the longest element of $W_X$.

\begin{definition}{(\cite[p.~109]{a-Satake60}, see also \cite[Definition 2.3]{a-Kolb14})} \label{def:Satake}
Let $X \subset I$ and $\tau:I \rightarrow I$ a diagram automorphism such that $\tau(X) = X$. The pair $(X, \tau)$ is called a Satake diagram if it satisfies the following properties:

\begin{enumerate}[label = \arabic*)]
\item $\tau^2 = \text{id}_{I}$. \label{property:Satake1}
\item The action of $\tau$ on $X$ coincides with the action of $-w_X$, that is \begin{equation*}
  -w_X(\alpha_i) = \alpha_{\tau(i)}\quad \mbox{for all $i\in X$}.
\end{equation*}

\label{property:Satake2}
\item If $j \in I \setminus X$ and $\tau(j) = j$, then $\alpha_j(\rho_X^{\vee}) \in \mathbb{Z}$. \label{property:Satake3}
\end{enumerate}
\end{definition}

\begin{remark} \label{rem:Satake}
When we need to identify the underlying Lie algebra, we write $\satake$ to indicate the Satake diagram. With this convention, if $\satake$ is a Satake diagram and $i \in I \setminus X$ then $(\{i, \tau(i)\} \cup X, X, \tau|_{\{i,\tau(i)\}\cup X})$ is also a Satake diagram.
\end{remark}

Graphically, the ingredients of a Satake diagram are recorded in the Dynkin diagram of $\lie{g}$.
The nodes labelled by $X$ are coloured black and a double sided arrow is used to indicate the diagram automorphism. See \cite[p.32/33]{a-Araki62} for a complete list of Satake diagrams for simple $\lie{g}$.

We can now construct an involution $\theta$ corresponding to the Satake diagram $(X, \tau)$ as in \cite[Section 2.4]{a-Kolb14}.
Let $\omega: \lie{g} \rightarrow \lie{g}$ be the Chevalley involution given by
\begin{equation} \label{eq:Chevalley}
	\omega(e_i) = -f_i, \qquad \omega(h_i) = -h_i, \qquad \omega(f_i) = -e_i.
\end{equation}
The diagram automorphism $\tau$ can be lifted to a Lie algebra automorphism $\tau: \lie{g} \rightarrow \lie{g}$ denoted by the same symbol.
Recall from \cite[Lemma 56]{b-Ste67} that the action of the Weyl group $W$ on $\lie{h}$ can be lifted to an action of the corresponding braid group $Br(W)$ on $\lie{g}$.
We denote the action of $w \in Br(W)$ on $\lie{g}$ by $\Ad(w)$.
Let $s: I \rightarrow \mathbb{C}^{\times}$ be a function such that
\begin{align}
	s(i) &= 1 &  &\mbox{if $i \in X$ or $\tau(i) = i$,} \label{eq:sCond1}\\
	\dfrac{s(i)}{s(\tau(i))} &= (-1)^{\alpha_i(2\rho_X^{\vee})} & &\mbox{if $i \not \in X$ and $\tau(i) \neq i$.} \label{eq:sCond2} 
\end{align}  
This map extends to a group homomorphism $s_Q: Q \rightarrow \mathbb{C}^{\times}$ such that $s_Q(\alpha_i) = s(i)$ for each simple root $\alpha_i$.
Let $\Ad(s): \lie{g} \rightarrow \lie{g}$ denote the Lie algebra automorphism such that the restriction to any root space $\lie{g}_{\alpha}$ is given by multiplication by $s_Q(\alpha)$.

Given a Satake diagram $(X, \tau)$ we define a Lie algebra automorphism $\theta=\theta(X,\tau)$ of $\lie{g}$ by
\begin{equation} \label{eq:Theta(X,t)}
	 \theta(X, \tau) = \Ad(s) \circ \Ad(w_X) \circ \tau \circ \omega : \lie{g} \rightarrow \lie{g}. 
\end{equation}
Here, the longest element $w_X \in W_X$ is considered as an element in the braid group $Br(W)$. The Lie algebra automorphism $\theta=\theta(X,\tau)$ is involutive, that is $\theta^2=\mathrm{id}$.  Any involutive Lie algebra automorphism of $\lie{g}$ is $\text{Aut}(\lie{g})$-conjugate to an automorphism of the form $\theta(X,\tau)$, see for example \cite[Theorems 2.5, 2.7]{a-Kolb14}.


\subsection{The subgroup $W^{\Theta}$} \label{sec:WTheta}

For any Satake diagram $(X, \tau)$, the automorphism $\theta = \theta(X, \tau)$ satisfies $\theta(\lie{h}) = \lie{h}$. More explicitly, the definition \eqref{eq:Theta(X,t)} implies that $\theta|_{\lie{h}} = -w_X \circ \tau$.
The dual map $\Theta: \lie{h}^{\ast} \rightarrow \lie{h}^{\ast}$ is given by the same expression
\begin{equation} \label{eq:Theta}
	\Theta = -w_X \circ \tau
\end{equation} 
where now $w_X$ and $\tau$ act on $\lie{h}^{\ast}$.
We obtain a group automorphism 
\begin{equation} \label{eq:ThetaGpHom}
\Theta_W \colon	W \rightarrow W, \qquad
				w \mapsto \Theta \circ w \circ \Theta.
\end{equation}
Let $W^{\Theta} = \{ w \in W \mid \Theta_W(w) = w\}$ denote the subgroup of elements fixed by $\Theta_W$. In this section we recall the structure of the subgroup $W^\Theta$ following \cite{a-Lu76} and \cite{a-GI14}.
For any $i \in I$ one has
\begin{equation}\label{eq:ThetaWsi}
	\Theta_W(s_i)=(-w_X\tau) \circ s_i \circ (-w_X\tau)
				=w_Xs_{\tau(i)}w_X.
\end{equation}
This formula and property $2)$ in Definition \ref{def:Satake} imply that $W_X$ is a subgroup of $W^{\Theta}$.
Recall that for any subset $J \subset I$ we write $w_J$ to denote the longest element in the parabolic subgroup $W_J$.
For all $i \in I \setminus X$ define
\begin{equation} \label{eq:sitilde}
 \widet{s_i} = w_{\{i, \tau(i)\} \cup X} w_X^{-1}.
\end{equation}
Recall that there exists a diagram automorphism $\tau_0:I \rightarrow I$ such that the longest element $w_0 \in W$ satisfies
\begin{equation} \label{eq:w0Action}
	w_0(\alpha_i)	=	-\alpha_{\tau_0(i)}
\end{equation}
for all $i \in I$. By inspection of the list of Satake diagrams in \cite[p. 32/33]{a-Araki62} one sees that the set $X$ is $\tau_0$-invariant.

By Remark \ref{rem:Satake} the triple $(\{i, \tau(i)\} \cup X, X, \tau|_{\{i,\tau(i)\}\cup X})$ is a Satake diagram for any $i \in I \setminus X$.
In analogy to \eqref{eq:w0Action} there exists hence a diagram automorphism $\tau_{0,i} : \{i, \tau(i)\} \cup X \rightarrow \{i, \tau(i)\} \cup X$ such that $X$ is $\tau_{0,i}$-invariant and
\begin{equation} \label{eq:w0iAction}
w_{\{i, \tau(i)\} \cup X} (\alpha_j) = -\alpha_{\tau_{0,i}(j)}
\end{equation}
for all $j \in \{i,\tau(i)\} \cup X$. With this notation we get
\begin{equation} \label{eq: LongestWordComm (i,t(i))}
	w_{\{i, \tau(i)\} \cup X} s_j = s_{\tau_{0,i}(j)}w_{\{i,\tau(i)\} \cup X}
\end{equation}
and hence 
\begin{equation*}
	w_{\{i,\tau(i)\} \cup X} w_X	= \tau_{0,i}(w_X) w_{\{i,\tau(i)\} \cup X} = w_X w_{\{i,\tau(i)\} \cup X}.
\end{equation*}
This proves that $w_X$ and $w_{\{i,\tau(i)\} \cup X}$ commute in $W$ for any $i \in I \setminus X$.
Hence Equation \eqref{eq:ThetaWsi} implies that
\begin{equation} \label{eq:sitildeinW0}
	\widet{s_i} \in W^{\Theta} \qquad \mbox{for all $i \in I \setminus X$.}
\end{equation}
Let $\widet{W} \subset W^{\Theta}$ denote the subgroup of $W$ generated by all $\widet{s_i}$ for $i \in I \setminus X$. Let $l : W \rightarrow \mathbb{N}_0$ denote the length function with respect to $W$.
Let 
\begin{equation}\label{eq:W^X}
W^X = \{w \in W \mid l(s_iw) > l(w) \: \mbox{for all $i \in X$}\}
\end{equation}
be the set of minimal length left coset representatives of $W_X$. By \cite[25.1]{a-Lu02} the subset
\begin{equation}\label{eq:curlyW}
\mathcal{W} = \{ w \in W^X \mid wW_X = W_Xw\}
\end{equation}
is a subgroup of $W$. By \eqref{eq:sitilde} and \eqref{eq: LongestWordComm (i,t(i))} we have $\widet{s_i} \in \mathcal{W}$ for all $i \in I \setminus X$. Let $\mathcal{W}^{\tau} = \{ w \in \mathcal{W} \mid \tau(w) = w\}$. As $\tau(X) = X$ the definition of $\widet{s_i}$ implies that $\widet{W} \subseteq \mathcal{W}^{\tau}$. The following lemma is a generalisation of \cite[A.1(a)]{a-Lu02} and \cite[Lemma 2]{a-GI14} in the spirit of \cite[Remark 8]{a-GI14}.

\begin{lemma} \label{lem:curlyW subset Wtil}
Any $w \in \mathcal{W}^{\tau}$ may be written as $w = \widet{s_{i_1}} \dotsm \widet{s_{i_k}}$ such that $\widet{s_{i_1}}, \dotsc , \widet{s_{i_k}} \in \widet{W}$ and $l(w) = l(\widet{s_{i_1}}) + \dotsm + l(\widet{s_{i_k}})$. 
\end{lemma}

Lemma \ref{lem:curlyW subset Wtil} implies that
 $\widet{W} = \mathcal{W}^{\tau}$.
Set $\widet{S} = \{\widet{s_i} \mid i \in I \setminus X\}$. The following theorem is stated in \cite[25.1]{a-Lu02} and proved in \cite[Theorem 5.9(i)]{a-Lu76}. A proof using only Weyl group combinatorics is indicated in \cite[Remark 8]{a-GI14}.

\begin{theorem}\label{thm: Wtilde Coxeter}
The pair $(\widet{W}, \widet{S})$ is a Coxeter system.
\end{theorem} 
Let $\lambda : \widet{W} \rightarrow \mathbb{N}_0$ denote the length function with respect to $(\widet{W}, \widet{S})$. The following proposition, proved in \cite[Theorem 5.9(iii)]{a-Lu76} and indicated in \cite[Corollary 6]{a-GI14} implies that reduced expressions in $\widet{W}$ are also reduced in $W$.

\begin{proposition} \label{prop: length comparison}
Let $w, w^{\prime} \in \widet{W}$. Then $l(ww^{\prime}) = l(w) + l(w^{\prime})$ if and only if $\lambda(ww^{\prime}) = \lambda(w) + \lambda(w^{\prime})$.
\end{proposition}

\begin{remark}
By \eqref{eq: LongestWordComm (i,t(i))} the group $\widet{W}$ acts on $W_X$ by conjugation. Moreover $W_X\cap \widet{W}=\{\mathrm{id}\}$. One can show that $W^{\Theta} = W_X \rtimes \widet{W}$.
\end{remark}

\subsection{The restricted root system} \label{sec:restricted}

The group $\widet{W}$ has an interpretation as the Weyl group of the restricted root system of the symmetric Lie algebra $(\lie{g},\theta)$. This fact is implicit in \cite{a-Lu76} but we feel that it is beneficial to explain this connection in some detail. As $\theta(\lie{h}) = \lie{h}$, we obtain a direct sum decomposition
\begin{equation} \label{eq:hdecomp}
	\lie{h} = \lie{h}_1 \oplus \lie{a}
\end{equation}
where $\lie{h}_1 = \{x\in \lie{h}\,|\,\theta(x)=x\}$ and $\lie{a} = \{x\in\lie{h}\,|\,\theta(x)=-x\}$.
The restricted root system $\Sigma \subset \lie{a}^{\star}$ is obtained by restriction of all roots in $\Phi$ to $\lie{a}$, that is
\begin{equation} \label{eq:Sigma}
	\Sigma	=	\Phi|_{\lie{a}} \setminus \{0\}.
\end{equation}
As $\Theta(\Phi)=\Phi$ we have $\Theta(V)=V$. Moreover, the inner product $(-,-)$ is $\Theta$-invariant. Hence we obtain an orthogonal direct sum decomposition
\begin{equation} \label{eq:Vdecomp}
	V	=	V_{+1} \oplus V_{-1}
\end{equation}
where $V_{\lambda} = \{ \alpha \in V \mid \Theta(\alpha) = \lambda\alpha \}$ for $\lambda = \pm 1$.
Indeed, any $\alpha \in V$ can be written as
\begin{equation} \label{eq:splittingroot}
	\alpha	=	\dfrac{\alpha + \Theta(\alpha)}{2} + \dfrac{\alpha - \Theta(\alpha)}{2}
\end{equation}
where $(\alpha + \lambda\Theta(\alpha))/2 \in V_\lambda$ for $\lambda = \pm 1$.
For any $\beta \in V_{-1}$ and $h \in \lie{h}_1$, we have $\beta(h) = 0$ as
	$\beta(h)	=	\beta(\Theta(h))	=	\Theta(\beta)(h)=	-\beta(h)$.
Hence we may consider $V_{-1}$ as a subspace of $\lie{a}^{\star}$ and $V_{-1}=\mathbb{R}\Sigma$.
For any $\beta \in V$ we define 
\begin{equation} \label{eq:roottilde}
	\widet{\beta} = \dfrac{\beta - \Theta(\beta)}{2}.
\end{equation}
Equation \eqref{eq:splittingroot} implies that $\Sigma = \{ \widet{\beta} \mid \beta \in \Phi, \widet{\beta} \neq 0\}$. We write $\widet{\Pi}=\{\widet{\alpha_i}\,|\, i\in I\setminus X\}$ and define $Q(\Sigma)=\Z \Sigma=\Z \widet{\Pi}$ and $Q^+(\Sigma)=\N_0\widet{\Pi}$.
For any $w \in W^{\Theta}$, we have
\begin{equation} \label{eq:WThetaaction}
\begin{split}
	w(\widet{\beta})	=	w\bigg( \dfrac{\beta - \Theta(\beta)}{2} \bigg)	
					&=	\dfrac{w(\beta) - w(\Theta(\beta))}{2}\\
					&=	\dfrac{w(\beta) - \Theta(w(\beta))}{2}
					= \widet{w(\beta)}. 
\end{split}
\end{equation}
Hence the group $W^{\Theta}$ acts on $\Sigma$.

We have an inner product on $V_{-1}$ by restriction of the inner product on $V$.
As the inner product on $V$ is $W$-invariant and $V_{-1}$ is a $W^{\Theta}$-invariant subspace, the restriction of the inner product on $V_{-1}$ is $W^{\Theta}$-invariant.
For any $i \in X$, we have $\alpha_i \in V_{+1}$. 
Hence the direct sum decomposition \eqref{eq:Vdecomp} implies that $s_i(\widet{\beta}) = \widet{\beta}$ for all $\widet{\beta} \in \Sigma, i \in X$.
The group $\widet{W}$ on the other hand can be interpreted in terms of the restricted root system $\Sigma$. We include a pedestrian proof, avoiding the more sophisticated setting in \cite{a-Lu76}.

\begin{proposition}\label{proposition:ReflGrp}
  \begin{enumerate}
\item  The reflections at the hyperplanes perpendicular to elements of $\Sigma$ generate a finite reflection group $W(\Sigma)$. \label{part:ReflGrp1}
\item There is an isomorphism of groups $\rho:\widet{W}\rightarrow W(\Sigma)$ which sends $\widet{s_i}$ to the reflection at the hyperplane orthogonal to $\widet{\alpha_i}$ for any $i\in I\setminus X$. \label{part:ReflGrp2}
\end{enumerate}
\end{proposition}

\begin{proof}
For any $i \in I \setminus X$ we have
\begin{align*}
	\widet{s_i}(\widet{\alpha_i})	&=	\big( w_X^{-1}w_{\{i,\tau(i)\} \cup X} \big)(\widet{\alpha_i})\\
									&=	w_{\{i, \tau(i)\} \cup X} (\widet{\alpha_i})\\
									&=	\big( w_{\{i,\tau(i)\} \cup X}(\alpha_i) \big)|_{\lie{a}}.			
\end{align*}
Since $w_{\{i,\tau(i)\} \cup X}(\alpha_j) \in \{-\alpha_j, -\alpha_{\tau(j)}\}$ for all $j \in \{i, \tau(i)\} \cup X$, it follows that
\begin{align} \label{eq:siai}
	\widet{s_i}(\widet{\alpha_i}) = -\widet{\alpha_i}.
\end{align}
Now suppose $\widet{\beta} \in V_{-1}$ such that $(\widet{\beta}, \widet{\alpha_i}) = 0$. 
Using the $W^{\Theta}$-invariance of the bilinear form on $V_{-1}$, we obtain
\begin{align*}
  (\widet{s_i}(\widet{\beta}), \widet{\alpha_i})&=(\widet{\beta}, \widet{s_i}(\widet{\alpha_i}) )\\
  &=-(\widet{\beta}, \widet{\alpha_i})\\
  &=0.
\end{align*}
On the other hand, by the definition of $\widet{s_i}$ we have
\begin{equation*}
	\widet{s_i}(\beta)	=	\beta + n_i\alpha_i + n_{\tau(i)}\alpha_{\tau(i)} + \sum_{j \in X}n_j\alpha_j
\end{equation*}
for some $n_j \in \mathbb{Q}$. 
From this we obtain
\begin{equation*}
	\widet{s_i}(\widet{\beta})	=	\widet{\beta} + m_i\widet{\alpha_i}
\end{equation*}
where $m_i = n_i + n_{\tau(i)}$. 
Hence,
\begin{align*}
	0	=	(\widet{s_i}(\widet{\beta}), \widet{\alpha_i})
		&=	(\widet{\beta} + m_i\widet{\alpha_i}, \widet{\alpha_i})\\
		&=	(\widet{\beta}, \widet{\alpha_i}) + m_i(\widet{\alpha_i}, \widet{\alpha_i})\\
		&=	m_i(\widet{\alpha_i}, \widet{\alpha_i}).
\end{align*}
The inner product is positive definite so it follows that $m_i = 0$.
Hence $\widet{s_i}(\widet{\beta}) = \widet{\beta}$.
This together with \eqref{eq:siai} implies that $\widet{s_i}$ is the reflection at the hyperplane orthogonal to $\widet{\alpha_i}$.
As seen above, the action of $\widet{W}$ on $V_{-1}$ gives a group homomorphism
\begin{equation*}
	\rho : \widet{W} \rightarrow W(\Sigma).
\end{equation*}
Adapting the proof of \cite[Theorem 1.5]{b-Hum90} one shows that $\rho$ is surjective which implies part \ref{part:ReflGrp1}.
To prove part \ref{part:ReflGrp2} it remains to show that $\rho$ is injective. This is a consequence of Lemma \ref{lem:faithful} below.
\end{proof}

\begin{lemma} \label{lem:faithful}
The action of $\widet{W}$ on $\Sigma$ is faithful.
\end{lemma}

\begin{proof}
Assume that there exists $w \in \widet{W}$ such that $w \neq 1_{\widet{W}}$ and
\begin{equation*}
	w (\widet{\alpha_i})	=	\widet{\alpha_i}	\quad	\mbox{for all $i \in I$.}
\end{equation*} 
We can rewrite this formula as 
\begin{equation} \label{prop:Faithful1}
 w(\alpha_i) - w(\Theta(\alpha_i))	=	\alpha_i - \Theta(\alpha_i).
\end{equation}
For all $i \in X$ we have $w(\alpha_i) > 0$ as $l(ws_i) = l(w)+1$.
Hence there exists $i \in I \setminus X$ such that $w(\alpha_i) < 0$. 
In this case also $w(\alpha_{\tau(i)}) < 0$ since elements of $\widet{W}$ are fixed under $\tau$.
Consider Equation \eqref{prop:Faithful1} for this $i$:
The right hand side lies in $Q^+$ and is of the form 
\begin{equation} \label{prop:Faithful2}
\alpha_i- \Theta(\alpha_i)=\alpha_i + \alpha_{\tau(i)} + \sum_{j \in X} n_j\alpha_j
\end{equation}
where $n_j \in \mathbb{N}_0$ for each $j \in X$. 
We can write the left hand side as
\begin{equation} \label{prop:Faithful3}
w(\alpha_i) - w(\Theta(\alpha_i)) = w(\alpha_i) + w(\alpha_{\tau(i)}) + \sum_{j \in X} m_j w(\alpha_j)
\end{equation}
where $m_j \in \mathbb{N}_0$ for each $j \in X$.
Hence inserting \eqref{prop:Faithful2} and \eqref{prop:Faithful3} into \eqref{prop:Faithful1}, we get
\begin{equation*}
w(\alpha_i) + w(\alpha_{\tau(i)}) + \sum_{j \in X}m_jw(\alpha_j) = \alpha_i + \alpha_{\tau(i)} + \sum_{j \in X}n_j\alpha_j.
\end{equation*}
Now we apply the tilde map to the above equation.
The terms involving $\alpha_j$ for $j \in X$ vanish, because the tilde map is zero on $Q_X$ and $w$ commutes with $\Theta$. 
We get
\begin{equation*}
\widet{w(\alpha_i)} + \widet{w(\alpha_{\tau(i)})} = \widet{\alpha_i} + \widet{\alpha_{\tau(i)}}.
\end{equation*}
The right hand side lies in $Q^+(\Sigma)$. 
The left hand side lies in $-Q^+(\Sigma)$ because $w(\alpha_i)$ and $w(\alpha_{\tau(i)})$ lie in $-Q^+$.
Hence both sides of the equation must vanish.
However, this is not possible, in particular for the right hand side which is $2\widet{\alpha_i}$.
We have a contradiction.
\end{proof}

Proposition \ref{proposition:ReflGrp} has the following consequence which we note for later reference.

\begin{corollary}\label{cor:ReflGrp}
  For any $i\in I\setminus X$ and $\mu\in Q(\Sigma)$ the relation
  \begin{align*}
    \widet{s_i}(\mu) = \mu - 2\frac{(\mu,\widet{\alpha_i})}{(\widet{\alpha_i},\widet{\alpha_i})}\widet{\alpha_i} 
  \end{align*}
  holds and $\displaystyle 2\frac{(\mu,\widet{\alpha_i})}{(\widet{\alpha_i},\widet{\alpha_i})}\in \Z$.
\end{corollary}

\section{Factorisation of quasi $K$-matrices} \label{sec:facK}

As explained in the introduction, the quasi $R$-matrix for $U_q(\lie{g})$ \cite[Chapter 4]{b-Lusztig94} has a deep connection to the Weyl group $W$. This was first observed by Levendorski\u{\i} and Soibelman \cite{a-LS90}, and independently by Kirillov and Reshetikhin \cite{a-KR90}. In this section, we establish a similar connection between the quasi $K$-matrix $\Qkm$ for a quantum symmetric pair and the restricted Weyl group $\widet{W}$. In particular, we will see in many cases that the quasi $K$-matrix $\Qkm$ factorises into a product of quasi $K$-matrices for Satake diagrams of rank one.

We first fix notation for quantised enveloping algebras and quantum symmetric pairs in Sections \ref{sec:QEAs} and \ref{sec:QSPs}. Recall that the construction of quantum symmetric pairs depends on an additional choice of parameters $\mathbf{c} \in \mathcal{C}$ and $\mathbf{s} \in \mathcal{S}$, see Definition \ref{Def:QSPCS}. In Sections \ref{sec:rank1Qkms} and \ref{sec:PartialQkms} we find explicit formulas for $\Qkm$ in the case $\mathbf{s} = (0, \dotsc , 0)$. Section \ref{sec:sGeneral} then deals with the case of general parameters $\mathbf{s}$. 

\subsection{Quantised enveloping algebras} \label{sec:QEAs}

In this section we fix notation for quantum groups mostly following the conventions in \cite{b-Lusztig94} and \cite{b-Jantzen96}.
Let $q$ be an indeterminate and $\mathbb{K}$ a field of characteristic zero.
Denote by $\mathbb{K}(q)$ the field of rational functions in $q$ with coefficients in $\mathbb{K}$. The quantised enveloping algebra $\uqg$ is the associative $\field(q)$-algebra generated by elements $E_i, F_i, K_i^{\pm1}$ for all $i \in I$ satisfying the relations given in \cite[3.1.1]{b-Lusztig94} or \cite[4.3]{b-Jantzen96}.
The algebra $\uqg$ has the structure of a Hopf algebra with coproduct $\Delta$, counit $\varepsilon$ and antipode $S$ given by
\begin{align}
\Delta(E_i) &= E_i \otimes 1 + K_i\otimes E_i,	&	\varepsilon(E_i) &= 0, 	&	S(E_i) &= -K_i^{-1}E_i, \\
\Delta(F_i) &= F_i \otimes K_i^{-1} + 1 \otimes F_i,	&	\varepsilon(F_i) &= 0, & S(F_i) &= -F_iK_i, \\
\Delta(K_i) &= K_i \otimes K_i,	&	\varepsilon(K_i) &= 1,	& S(K_i) &= K_i^{-1},	
\end{align}
for all $i \in I$.
Let $U^{+}, U^{-}$ and $U^0$ denote the subalgebras of $U_q(\lie{g})$ generated by $\{E_i \mid i \in I\}$, $\{F_i \mid i \in I\}$ and $\{K_i^{\pm1} \mid i \in I\}$, respectively. 
For $\lambda = \sum_{i \in I} n_i\alpha_i \in Q$ we also write
\begin{equation} \label{eq:Klambda}
	K_\lambda	=	\prod_{i \in I}K_i^{n_i}.
\end{equation}
The elements $K_\lambda$ for $\lambda\in Q$ form a vector space basis of $U^0$.
For any $U^0$-module $M$ and any $\mu \in Q$ let
\begin{equation}\label{eq:WeightSpace}
	M_{\mu} = \{m \in M \mid K_im = q^{(\mu, \alpha_i)}m \:\:	\mbox{for all $i \in I$}\}
\end{equation}
denote the corresponding weight space. In particular, both $U^+$ and $U^-$ are $U^0$-modules via the left adjoint action so we may apply the above notation.
This gives algebra gradings
\begin{align} \label{eq:WeightDecomp}
	U^+ &= \bigoplus_{\mu \in Q^+} U_{\mu}^+,	&	U^-	&=	\bigoplus_{\mu \in Q^+} U_{-\mu}^-.
\end{align}
By \cite[39.4.3]{b-Lusztig94} the braid group $Br(W)$ acts on $U_q(\lie{g})$ by algebra automorphisms analogously to the action of $Br(W)$ on $\lie{g}$.
For any $i \in I$ let $T_i$ be the algebra automorphism of $U_q(\lie{g})$ denoted by $T_{i,1}^{\prime\prime}$ in \cite[37.1]{b-Lusztig94}.
The automorphisms $T_i$ for $i \in I$ satisfy braid relations
\begin{equation}\label{eq:Tibraid}
	\underbrace{T_iT_jT_i \cdots}_{\mbox{\small{$m_{ij}$ factors}}} 	=	\underbrace{T_jT_iT_j \cdots}_{\mbox{\small{$m_{ij}$ factors}}}
\end{equation}
where $m_{ij}$ denotes the order of $s_is_j \in W$.
Hence for any $w \in W$, there is a well-defined algebra automorphism $T_w: U_q(\lie{g}) \rightarrow U_q(\lie{g})$ defined by
\begin{equation}
T_w = T_{i_1}T_{i_2} \cdots T_{i_k}
\end{equation}
where $w = s_{i_1}s_{i_2}\dots s_{i_k}$ is a reduced expression.

In Sections \ref{sec:QSPs} and \ref{sec:PartialQkms} it is necessary to consider a completion $\mathscr{U}$ of $U_q(\lie{g})$. 
We recall the construction of $\mathscr{U}$ following \cite[3.1]{a-BK15}.
Let $\mathcal{O}_{int}$ denote the category of finite dimensional $U_q(\lie{g})$-modules of type $1$, and let $\mathcal{V}ect$ denote the category of vector spaces over $\mathbb{K}(q)$. Both categories are tensor categories and the forgetful functor
$\mathcal{F}or: \mathcal{O}_{int} \rightarrow \mathcal{V}ect$
is a tensor functor.
We let $\mathscr{U} = \mathrm{End}(\mathcal{F}or)$ be the $\field(q)$-algebra of all natural transformations from the functor $\mathcal{F}or$ to itself. 
Observe that $\uqg$ and $\widehat{U^+} = \prod_{\mu \in Q^+} U_{\mu}^+$ may be considered as subalgebras of $\sU$, see \cite[Section 3.1]{a-BK15}. We usually write elements in $\widehat{U^+}$ additively as infinite sums $\sum_{\mu\in Q^+} u_\mu$ with $u_\mu\in U^+_\mu$.

By \cite[1.2.13]{b-Lusztig94} for any $i\in I$ there exist uniquely determined linear maps ${}_ir$, $r_i: U^{+} \rightarrow U^{+}$ satisfying
\begin{align}
	\ir{i}{E_j}	&=	\delta_{ij},	&	\ir{i}{xy}	&=	\ir{i}{x}y + q^{(\alpha_i,\mu)}x\ir{i}{y} \label{eq:ir}\\
	\ri{i}{E_j}	&=	\delta_{ij},	&	\ri{i}{xy}	&=	q^{(\alpha_i, \nu)}\ri{i}{x}y + x\ri{i}{y}, \label{eq:ri}
\end{align}
for all $j \in I$, $x \in U^{+}_{\mu}$ and $y \in U^{+}_{\nu}$ where $\mu,\nu\in Q^+$. We may extend the skew derivation $_{i}r: U^+ \rightarrow U^+$  to a linear map
\begin{equation}\label{eq:irExtension}
	_{i}r: \widehat{U^+} \rightarrow \widehat{U^+}, \quad \sum_{\mu\in Q^+} u_{\mu} \mapsto \sum_{\mu\in Q^+} \ir{i}{u_{\mu}}
\end{equation}
where ${}_ir(u_\mu)$ is the component in $U^+_{\mu-\alpha_i}$ for all $\mu \in Q^+$ with $\mu\ge \alpha_i$. Similarly we may extend the skew derivation $r_i: U^+ \rightarrow U^+$ to a linear map $r_i: \widehat{U^+} \rightarrow \widehat{U^+}$.

Finally, recall that the bar involution for $\uqg$ is the $\mathbb{K}$-algebra automorphism
\begin{align} \label{eq:BarInv}
  \barinv: U_q(\lie{g}) \rightarrow U_q(\lie{g}), \qquad u\mapsto \overline{u}^U
\end{align}
defined by $\overline{q}^U = q^{-1}$ and $\overline{E_i}^U = E_i$, $\overline{F_i}^U = F_i$, $\overline{K_i}^U = K_i^{-1}$ for all $i \in I$.

\subsection{Quantum symmetric pairs} \label{sec:QSPs}
We recall the definition of quantum symmetric pair coideal subalgebras $\Bcs$ as introduced by G.~Letzter in \cite{a-Letzter99a}. Here we follow the conventions in \cite{a-Kolb14}. Let $(X, \tau)$ be a Satake diagram and let $s:I \rightarrow \field^{\times}$ be a function satisfying equations \eqref{eq:sCond1} and \eqref{eq:sCond2}.
Let $\mathcal{M}_X~=~U_q(\lie{g}_X)$ be the subalgebra of $U_q(\lie{g})$ generated by $\{E_i, F_i, K_i^{\pm1} \mid i \in X \}$, and let $U_\Theta^0$ be the subalgebra of $U^0$ generated by $\{K_\lambda\,|\,\lambda\in Q, \Theta(\lambda)=\lambda\}$.
Quantum symmetric pair coideal subalgebras depend on a choice of parameters $\mathbf{c} = (c_i)_{i \in I \setminus X} \in (\mathbb{K}(q)^{\times})^{I \setminus X}$ and $\mathbf{s} = (s_i)_{i \in I \setminus X} \in (\mathbb{K}(q)^{\times})^{I \setminus X}$ with added constraints. Define
\begin{equation} \label{eq:Ins}
	I_{ns} = \{i \in I \setminus X \mid \mbox{$\tau(i) = i$ and $a_{ij} = 0$ for all $j \in X$} \} 
\end{equation}
where $a_{ij} = 2\frac{(\alpha_i, \alpha_j)}{(\alpha_j,\alpha_j)}$ for $i,j \in I$ are the entries of the Cartan matrix of $\lie{g}$. Define the parameter sets
\begin{align}
	\mathcal{C}	&=	\{ 	\mathbf{c} \in (\mathbb{K}(q)^{\times})^{I \setminus X} 
							\mid 
						\mbox{$c_i = c_{\tau(i)}$ if $\tau(i) \neq i$ and $(\alpha_i, \Theta(\alpha_i)) = 0$} \}, \label{eq:ParameterSetC} \\
	\mathcal{S}	&=	\{	\mathbf{s} \in (\mathbb{K}(q)^{\times})^{I \setminus X}
							\mid
						\mbox{$s_j \neq 0 \Rightarrow (j \in I_{ns}$ and $a_{ij} \in -2\mathbb{N}_0 \; \forall i \in I_{ns} \setminus \{j\}) $} \}. \label{eq:ParameterSetS}  
\end{align}

\begin{definition} \label{Def:QSPCS}
Let $(X, \tau)$ be a Satake diagram, $\mathbf{c} = (c_i)_{i \in I \setminus X} \in \mathcal{C}$ and $\mathbf{s} = (s_i)_{i \in I \setminus X} \in \mathcal{S}$. 
The quantum symmetric pair coideal subalgebra $\Bcs = \Bcs(X, \tau)$ is the subalgebra of $U_q(\lie{g})$ generated by $\mathcal{M}_X, U_{\Theta}^0$ and the elements 
\begin{equation} \label{eq:Bi}
	B_i	=	F_i	-	c_is(\tau(i))T_{w_X}(E_{\tau(i)})K_i^{-1}	+	s_iK_i^{-1} 
\end{equation}
for all $i\in I\setminus X$.
\end{definition}
By \cite[Proposition 5.2]{a-Kolb14} the algebra $\Bcs$ is a right coideal subalgebra of $U_q(\lie{g})$, that is
\begin{equation} \label{eq:RCSA}
	\Delta(\Bcs)	\subseteq	\Bcs \otimes U_q(\lie{g}).
\end{equation}
All through this paper we assume that the parameters $\mathbf{c} = (c_i)_{i \in I \setminus X} \in \mathcal{C}$ and $\mathbf{s} = (s_i)_{i \in I \setminus X} \in \mathcal{S}$ satisfy the additional relations
\begin{align}
	c_{\tau(i)}	&=	q^{(\alpha_i, \Theta(\alpha_i) - 2\rho_X)} \overline{c_i}^U, \label{eq:ciCond} \\
	\overline{s_i}^U	&=	s_i \label{eq:siCond} 
\end{align}
for all $i \in I \setminus X$.
By \cite[Theorem 3.11]{a-BK14}, if condition \eqref{eq:ciCond} holds, then there exists a $\mathbb{K}$-algebra automorphism
\begin{equation} \label{eq:BbarInv}
	\Bbarinv: \Bcs \rightarrow \Bcs, \quad b \mapsto \overline{b}^B
\end{equation}
such that 
\begin{align}
	\Bbarinv|_{\cM_X U_\Theta^0} &= \barinv|_{\cM_X U_\Theta^0}, & 	\overline{B_i}^B&=B_i \quad \mbox{for all $i\in I\setminus X$} . \label{eq:BbarInvProp}
\end{align}
The map $\Bbarinv$ is called the bar involution for $\Bcs$ and plays a similar role as the bar involution \eqref{eq:BarInv} for $U_q(\lie{g})$.
In particular there exists a quasi $K$-matrix $\Qkm \in \mathscr{U}$ which resembles the quasi $R$-matrix for $U_q(\lie{g})$.
More explicitly, following a program outlined by H.~Bao and W.~Wang in \cite{a-BaoWang13}, it was proved in \cite[Theorem 6.10]{a-BK15} that there exists a uniquely determined element $\Qkm = \sum_{\mu \in Q^+} \Qkm_{\mu} \in \prod_{\mu \in Q^+} U_{\mu}^+$ with $\Qkm_0 = 1$ and $\Qkm_{\mu} \in U_{\mu}^+$ such that the equation
\begin{equation} \label{eq:BbarIntertwiner}
	\overline{b}^B \Qkm	=	\Qkm \overline{b}^U
\end{equation}
holds in $\mathscr{U}$ for all $b \in \Bcs$. For symmetric pairs of type $AIII$/$IV$ with $X=\emptyset$ the existence of the quasi $K$-matrix satisfying \eqref{eq:BbarIntertwiner} was first observed in \cite{a-BaoWang13}.
The quasi $K$-matrix is an essential building block for the construction of the universal $K$-matrix in \cite{a-BK15}.
To unify notation we define $c_i = s_i = 0$ and $B_i = F_i$ for $i\in X$. Moreover, we write 
\begin{align} \label{eq:Xi}
	X_i	&=	-s(\tau(i))T_{w_X}(E_{\tau(i)}) &&\mbox{for all $i \in I \setminus X$,}\\
        X_j &= 0 &&\mbox{for all $j\in X$.}
\end{align}

By \cite[Proposition 6.1]{a-BK15} the quasi $K$-matrix $\Qkm = \sum_{\mu \in Q^+}\Qkm_{\mu}$ is the unique solution to the recursive system of equations
\begin{equation} \label{eq:irSyst}
\ir{i}{\Qkm_\mu} = -(q_i - q_i^{-1}) \big(q^{-(\Theta(\alpha_i),\alpha_i)}c_iX_i \Qkm_{\mu + \Theta(\alpha_i)-\alpha_i} + s_i\Qkm_{\mu-\alpha_i} \big) \quad \mbox{for all $i\in I$}
\end{equation}
with the normalisation $\Qkm_0 = 1$. Using the extension of the skew derivation ${}_{i}r$ to $\widehat{U^+}$ given by \eqref{eq:irExtension} we can rewrite \eqref{eq:irSyst} in the compact form
\begin{equation} \label{eq:irSyst2}
\ir{i}{\Qkm} = -(q_i-q_i^{-1}) \big(q^{-(\Theta(\alpha_i),\alpha_i)}c_iX_i \Qkm + s_i\Qkm \big) \qquad \mbox{for all $i\in I$.}
\end{equation} 
In Section \ref{sec:rank1Qkms} and in Appendix \ref{App:RankTwo} we will use the above formula to perform uniform calculations with $\Qkm$.
Equation \eqref{eq:irSyst} implies in particular that
\begin{align}\label{eq:jrX=0}
  \ir{j}{\Qkm_\mu}=0 \qquad \mbox{for all $j\in X$}
\end{align}
as $c_j = s_j = 0$ for all $j \in X$.
This property has the following consequence which was already observed in \cite[Proposition 4.15]{a-BW16}. Recall that $w_0\in W$ denotes the longest element. Moreover, for any $w\in W$ recall the definition of the subalgebra $U^+[w]$ of $U^+$ given in \cite[8.21, 8.24]{b-Jantzen96}.

\begin{lemma}\label{lem:xmuinUwxw0}
  For any $\mu\in Q^+$ the relation $\Qkm_\mu\in U^+[w_Xw_0]$ holds. 
\end{lemma}

\begin{proof}
  In view of \eqref{eq:jrX=0}, Equation $(4)$ of \cite[8.26]{b-Jantzen96} implies that $\Qkm_\mu\in  U^+[s_jw_0]$ for all $j \in X$.
By \cite[Theorem 7.3]{a-HS} we have
\begin{equation*}
\bigcap_{j \in X} U^+[s_jw_0] = U^+[w_Xw_0]
\end{equation*}
which completes the proof of the Lemma.
\end{proof}  

We write $\Qkm_{\B{c},\B{s}}$ for $\Qkm$ if we need to specify the dependence on the parameters. Any diagram automorphism $\eta : I \rightarrow I$ with $\eta(X)=X$ induces a map $\eta: \mathbb{K}(q)^{I \setminus X} \rightarrow \mathbb{K}(q)^{I \setminus X}$ by
\begin{equation} \label{eq:ParameterDiagAut}
	\eta((c_i)) = (d_i) \quad \mbox{with $d_i = c_{\eta^{-1}(i)}$.}
\end{equation}
This notation allows us to record the effect of diagram automorphisms on the quasi $K$-matrix $\Qkm$.

\begin{lemma} \label{lem:DiagAutofQkm}
Let $\eta: I \rightarrow I$ be a diagram automorphism with $\eta\circ \tau=\tau\circ \eta$ and $\eta(X)=X$. Then for any $\B{c} \in \mathcal{C}$, $\B{s} \in \mathcal{S}$ we have $\eta(\B{c}) \in \mathcal{C}$, $\eta(\B{s}) \in \mathcal{S}$ and
\begin{equation} \label{eq:DiagAutofQkm}
	\eta(\Qkm_{\B{c},\B{s}}) = \Qkm_{\eta(\B{c}),\eta(\B{s})}.
\end{equation}
\end{lemma}

\begin{proof}
The relations $\eta(\B{c}) \in \mathcal{C}$, $\eta(\B{s}) \in \mathcal{S}$ follow from the assumptions on $\eta$ and the definitions \eqref{eq:ParameterSetC} and \eqref{eq:ParameterSetS} of $\mathcal{C}$ and $\mathcal{S}$.
By \cite[Proposition 6.1]{a-BK15}, property \eqref{eq:BbarIntertwiner} is equivalent to the relation
\begin{equation} \label{eq:BbarIntEquiv}
	B_i^{\B{c},\B{s}} \Qkm_{\B{c},\B{s}} = \Qkm_{\B{c},\B{s}} \overline{B_i^{\B{c},\B{s}}}^U \quad \mbox{for all $i \in I$}
\end{equation}
where we write $B_i^{\B{c},\B{s}}$ instead of $B_i$ to denote the dependence on $\B{c}$ and $\B{s}$.

By construction, $\eta(B_i^{\B{c},\B{s}}) = B_{\eta(i)}^{\eta(\B{c}),\eta(\B{s})}$ and $\eta:U_q(\lie{g}) \rightarrow U_q(\lie{g})$ commutes with the bar involution $\barinv$ on $U_q(\lie{g})$.
Hence applying $\eta$ to relation \eqref{eq:BbarIntEquiv} we obtain
\begin{equation*}
	B_{\eta(i)}^{\eta(\B{c}),\eta(\B{s})} \eta(\Qkm_{\B{c},\B{s}}) = \eta(\Qkm_{\B{c},\B{s}}) \overline{B_{\eta(i)}^{\eta(\B{c}),\eta(\B{s})}}^U\quad \mbox{for all $i \in I$}
\end{equation*}
which proves \eqref{eq:DiagAutofQkm}.
\end{proof}
\subsection{Quasi $K$-matrices for Satake diagrams of rank one} \label{sec:rank1Qkms}
For the remainder of this paper, following Remark \ref{rem:Satake}, we denote Satake diagrams as triples $(I, X, \tau)$ to also indicate the underlying Lie algebra.

\begin{definition} \label{Def:SatakeSubdiagram}
A subdiagram of a Satake diagram $\satake$ is a triple $(J, X \cap J, \tau|_J)$ with $J \subseteq I$, such that $\tau(J)=J$ and $J\cap X$ consists of all connected components of $X$ which are connected to a white node in $J$. 
\end{definition}
\begin{remark}
  Any subdiagram $(J, X \cap J, \tau|_J)$ of a Satake diagram satisfies the properties of Definition \ref{def:Satake} and hence is itself a Satake diagram.
  It is possible to give a slightly more general definition of a Satake subdiagram which allows $J\cap X$ to contain connected components of $X$ which are not connected to any white node in $J$. Here we exclude such components for convenience. 
\end{remark}
  Let $\widetilde{I}$ be the set of $\tau$-orbits of $I \setminus X$.
There is a projection map
\begin{equation} \label{eq:NodeProjection}
	\pi \colon I \setminus X \longrightarrow \widetilde{I}
\end{equation}
that takes any white node to the $\tau$-orbit it belongs to.

\begin{definition} \label{Def:SubdiagramRank}
The rank of a Satake diagram $\satake$ is ~defined by ~$\rank\satake =|\pi(I \setminus X ) |$.
\end{definition}

In other words, a Satake diagram has rank $n$ if there are $n$ distinct orbits of white nodes. By Proposition \ref{proposition:ReflGrp} the rank of a Satake diagram coincides with the rank of the corresponding restricted root system $\Sigma$. 

For any $i\in I\setminus X$ let $X^i$ denote the union of all connected components of $X$ which are connected to $i$. It follows by inspection of the list of Satake diagrams in \cite[p.~32/33]{a-Araki62} that $X^i=X^{\tau(i)}$.
Given a Satake diagram $\satake$, any $i\in I\setminus X$ determines a subdiagram $(\{i,\tau(i)\}\cup X^i, X^i, \tau|_{\{i,\tau(i)\}\cup X^i})$ of rank one. Let $\Qkm_i$ be the 
quasi $K$-matrix corresponding to this rank one subdiagram. For any $w\in W$ we define $\widehat{U^+[w]} = \prod_{\mu \in Q^+} U^+[w]_{\mu}$.
As $U[w]^+$ is a subalgebra of $U^+$ we obtain that $\widehat{U^+[w]}$ is a subalgebra of $\widehat{U^+}$ and hence of $\mathscr{U}$. Formulating Lemma \ref{lem:xmuinUwxw0} in the present setting we obtain
\begin{equation}\label{eq:Qkmi in U[si]}
  \Qkm_i \in \widehat{U^+[\widet{s_i}]}.
\end{equation}  
In the following lemma we consider the case $\tau(i)=i$ and make the dependence of $\Qkm_i$ on the parameter $c_i$ more explicit. 

\begin{lemma} \label{lem:RankOneQkm}
Assume that $\B{s} = (0, \dotsc ,0)$ and $i \in I \setminus X$ satisfies $\tau(i) = i$. 
Then 
\begin{equation} \label{eq:RankOneQkm}
\Qkm_i	=	\sum_{n \in \mathbb{N}_0} c_i^n E_{n(\alpha_i - \Theta(\alpha_i))}
\end{equation}
where $E_{n(\alpha_i - \Theta(\alpha_i))} \in U^+_{n(\alpha_i - \Theta(\alpha_i))}$ is independent of $\B{c}$.
\end{lemma}

\begin{proof}
It follows from the recursion \eqref{eq:irSyst} and the assumption $s_i = 0$ that $\Qkm_i = \sum_{n \in \mathbb{N}_0} \Qkm_{n(\alpha_i - \Theta(\alpha_i))}$ with $\Qkm_{n(\alpha_i - \Theta(\alpha_i))} \in U^+_{n(\alpha_i - \Theta(\alpha_i))}$.
Again by \eqref{eq:irSyst}, the elements $E_{n(\alpha_i - \Theta(\alpha_i))} = c_i^{-n}\Qkm_{n(\alpha_i - \Theta(\alpha_i))}$ for $n \in \mathbb{N}$ satisfy the relations
\begin{equation} \label{eq:irRankOne1}
	\ir{i}{E_{n(\alpha_i - \Theta(\alpha_i))}} = -(q-q^{-1})q^{-(\Theta(\alpha_i),\alpha_i)}X_iE_{(n-1)(\alpha_i - \Theta(\alpha_i))}
\end{equation}
and 
\begin{equation} \label{eq:irRankOne2}
\ir{j}{E_{n(\alpha_i - \Theta(\alpha_i))}} = 0 \quad \mbox{for $j \in X$}.
\end{equation}
The relations \eqref{eq:irRankOne1} and \eqref{eq:irRankOne2} are independent of $\B{c}$ and determine $E_{n(\alpha_i - \Theta(\alpha_i))}$ uniquely if we additionally impose $E_0 = 1$.
\end{proof}

The quasi $K$-matrices of rank one are the building blocks for quasi $K$-matrices of higher rank. In the following we give explicit formulas for rank one quasi $K$-matrices of type $A$ shown on the left hand side of Table \ref{Table:RankOne} in the case $\B{s} = (0, \dotsc ,0)$. The additional rank one Satake diagrams on the right hand side of Table \ref{Table:RankOne} are not tackled in this paper.

 Recall the definition of the $q$-number
\begin{equation} \label{eq:qnumber}
	[n]_{q_i} =	[n]_i	= \dfrac{q_i^n - q_i^{-n}}{q_i - q_i^{-1}} \qquad \mbox{for any $n \in \mathbb{Z}$ and $i \in I$}
\end{equation} 
and of the $q$-factorial $[n]_i!=[1]_i [2]_i \dots [n]_i$ for $n\ge 0$, see for example \cite[Chapter 0]{b-Jantzen96}. If all roots $\alpha \in \Phi$ are of the same length, then we simply write $[n]$ and $[n]!$.

Throughout the following calculations in the present section and in Appendix \ref{App:RankTwo}, we use a modified form of the $q$-number $[n]_i$. For $n\ge 1$ define
\begin{equation}
	\{ n \}_i = q_i^{n-1}[n]_i = 1 + q_i^2 + \dots + q_i^{2(n-1)}.
\end{equation}
Define $\{n\}_i! = \prod_{k=1}^{n}\{k\}_i$, and let $\{n\}_i!!$ be the double factorial of $\{n\}_i$ defined by
\begin{equation}
	\{n\}_i!!	=	\prod_{k=0}^{\lceil\frac{n}{2}\rceil-1} \{n - 2k\}_i. 
\end{equation}

For $n=0$ we set $\{0\}_i! = 1$ and $\{0\}_i!!=1$.
Again, we omit the index $i$ if all roots are of the same length.
Further we use the following conventions.
For any $x,y \in U_q(\lie{g})$, $a \in \mathbb{K}(q)$ we denote by $[x,y]_a$ the element $xy - ayx$.
For any $i,j \in I$ we write $T_{ij} = T_i \circ T_j:\uqg \rightarrow \uqg$ and we extend this definition recursively.
 
\begin{table}
 \centering
 \caption{Satake diagrams of symmetric pairs of rank one} \label{Table:RankOne}
 
 \newcolumntype{C}{ >{\centering\arraybackslash} m{2cm} }
\newcolumntype{D}{ >{\centering\arraybackslash} m{4.5cm} }
\resizebox{\columnwidth}{!}{
\begin{tabular}[t]{| C |D || C | D |}

	\hline

	$AI_1$ & 
	\rb{
		\begin{tikzpicture} 
			[scale = 0.7, white/.style={circle,draw=black,inner sep = 0mm, minimum size = 3mm},
			black/.style={circle,draw=black,fill=black, inner sep= 0mm, minimum size = 3mm}, every node/.style={transform shape}]
		
			\node[white] (first)  [label = below:{ \scriptsize $1$}] {};		
	
		\end{tikzpicture}	
	} & 
	
	$BII$, $n\geq 2$ & 
	\rb{
		\begin{tikzpicture} 
			[scale=0.7, white/.style={circle,draw=black,inner sep = 0mm, minimum size = 3mm},
			black/.style={circle,draw=black,fill=black, inner sep= 0mm, minimum size = 3mm}, every node/.style={transform shape}] 
		
			\node[white] (first)  [label = below:{\scriptsize $1$}] {};		
			\node[black] (second) [right=of first] [label = below:{\scriptsize $2$}]  {}			
				edge (first);
			\node[black] (third) [right = 1.5cm of second]  {}
				edge [dashed] (second);
			\node[black] (fourth) [right = of third] [label = below:{\scriptsize $n$}] {}
				edge [double equal sign distance, -<-] (third);
		\end{tikzpicture}	
	} \\ \hline
	
	
	
	$AII_3$ & 
	\rb{
		\begin{tikzpicture}  
			[scale=0.7, white/.style={circle,draw=black,inner sep = 0mm, minimum size = 3mm},
			black/.style={circle,draw=black,fill=black, inner sep= 0mm, minimum size = 3mm}, every node/.style={transform shape}]
	
			\node[black] (first) [label = below:{ \scriptsize $1$}] {};
			\node[white] (second) [right=of first] [label = below:{ \scriptsize $2$}]{}
				edge (first);
			\node[black] (last) [right= of second] [label = below:{ \scriptsize $3$}]{}
				edge (second);		
		\end{tikzpicture}	
	} & 
	$CII$, $n \geq 3$ & 
	\rb{
		\begin{tikzpicture} 
			[scale=0.7, white/.style={circle,draw=black,inner sep = 0mm, minimum size = 3mm},
			black/.style={circle,draw=black,fill=black, inner sep= 0mm, minimum size = 3mm}, every node/.style={transform shape}]
		
			\node[black] (first)  [label = below:{\scriptsize $1$}] {};		
			\node[white] (second) [right= of first] [label = below:{\scriptsize $2$}] {}			
				edge (first);
			\node[black] (third) [right = of second] [label = below:{\scriptsize $3$}] {}
				edge (second);
			\node[black] (fourth) [right = 1.5cm of third] {}
				edge [dashed] (third);
			\node[black] (last) [right = of fourth] [label = below:{\scriptsize $n$}] {}
				edge [double equal sign distance, ->-] (fourth);
		\end{tikzpicture}	
	} \\ \hline

		

	$AIII_{11}$ & 
	\rb{
		\begin{tikzpicture} 
			[scale=0.7, white/.style={circle,draw=black,inner sep = 0mm, minimum size = 3mm},
			black/.style={circle,draw=black,fill=black, inner sep= 0mm, minimum size = 3mm}, every node/.style={transform shape}]
		
			\node[white] (first) [label = below:{ \scriptsize $1$}] {};		
			\node[white] (third) [right= 0.8cm of first] [label = below:{ \scriptsize $2$}] {}
				edge	 [latex'-latex' , shorten <=3pt, shorten >=3pt, bend right=60, densely dotted] node[auto,swap] {} (first); 
		\end{tikzpicture}
	} & 
	$DII$, $n \geq 4$ & 
	\rb{
		\begin{tikzpicture}
			[scale=0.7, white/.style={circle,draw=black,inner sep = 0mm, minimum size = 3mm},
			black/.style={circle,draw=black,fill=black, inner sep= 0mm, minimum size = 3mm}, every node/.style={transform shape}]
		
			\node[white] (first) [label = below:{\scriptsize $1$}] {};		
			\node[black] (second) [right= of first] [label = below:{\scriptsize $2$}]  {}
				edge (first);
			\node[black] (third) [right = 1.5cm of second] {}
				edge [dashed] (second);
			\node[black] (fourth) [above right = 0.5cm of third] [label = above:{\scriptsize $n-1$}] {}
				edge (third);
			\node[black] (fifth) [below right = 0.5cm of third] [label = below:{\scriptsize $n$}] {}
				edge (third);		
		\end{tikzpicture}	
	} \\ \hline

		

	$AIV$, $n \geq 2$ & 
	\rb{
		\begin{tikzpicture}
			[scale=0.7, white/.style={circle,draw=black,inner sep = 0mm, minimum size = 3mm},
			black/.style={circle,draw=black,fill=black, inner sep= 0mm, minimum size = 3mm}, every node/.style={transform shape}]
	
			\node[white] (first) [label = below:{ \scriptsize $1$}] {};
			\node[black] (second) [right=of first] [label = below:{ \scriptsize $2$}]{}
				edge (first);
			\node[black] (last) [right=1.5cm of second] {}
				edge [dashed] (second);		
			\node[white] (fourth) [right=of last] [label = below:{ \scriptsize $n$}] {}
				edge (last)
				edge	 [latex'-latex' , shorten <=3pt, shorten >=3pt, bend right=30, densely dotted] node[auto,swap] {} (first);
		\end{tikzpicture}	
	} & 
	$FII$ & 
	\rb{
		\begin{tikzpicture} 
			[scale=0.7, white/.style={circle,draw=black,inner sep = 0mm, minimum size = 3mm},
			black/.style={circle,draw=black,fill=black, inner sep= 0mm, minimum size = 3mm}, every node/.style={transform shape}]
		
			\node[black] (first) [label = below:{ \scriptsize $1$}] {};
			\node[black] (second) [right = of first] [label = below:{ \scriptsize $2$}] {}
				edge (first);		
			\node[black] (third) [right=of second] [label = below:{ \scriptsize $3$}] {}			
				edge [double equal sign distance, -<-]  (second);
			\node[white] (fourth) [right = of third] [label = below:{ \scriptsize $4$}]{}
				edge (third);
		\end{tikzpicture}	
	} \\ \hline
	
		
\end{tabular}
}
\end{table} 


\subsubsection{Type $AI_1$}

Consider the Satake diagram of type $AI_1$.

\begin{center}
	\begin{tikzpicture}  
		[white/.style={circle,draw=black,inner sep = 0mm, minimum size = 3mm},
		black/.style={circle,draw=black,fill=black, inner sep= 0mm, minimum size = 3mm}]
		
		\node[white] (first)  [label = below:{ \scriptsize $1$}] {};		
	
	\end{tikzpicture}
\end{center}
\begin{lemma} \label{lem:QkmRankOneAI}
The quasi $K$-matrix $\Qkm$ of type $AI_1$ is given by
\begin{equation} \label{eq:QkmRankOneAI}
	\Qkm	=	\sum_{n \geq 0} \dfrac{\q^n}{ \{2n\}!! } (q^2c_1)^n E_1^{2n}.
\end{equation}
\end{lemma}

\begin{proof}
By Equation \ref{eq:irSyst2}, we need to show that
\begin{equation*}
	\ir{1}{\Qkm} = \q(q^2c_1)E_1\Qkm.
\end{equation*}
Using the recursive formula \eqref{eq:ir} for ${}_1r$, we see that 
\begin{equation} \label{eq:irEi}
\ir{1}{E_1^n}	=	\{n\}E_1^{n-1} \qquad \mbox{for all $n\in \N$.}
\end{equation}
Hence 
\begin{align*}
	\ir{1}{\Qkm}	&=	\sum_{n \geq 0} \dfrac{\q^n}{\{2n\}!!} (q^2c_1)^n \ir{1}{E_1^{2n}}\\
				&=	\sum_{n \geq 1} \dfrac{\q^n}{\{2n\}!!} (q^2c_1)^n \{2n\} E_1^{2n-1}\\
				&=	\sum_{n \geq 0} \dfrac{\q^{n+1}}{\{2n\}!!} (q^2c_1)^{n+1} E_1^{2n+1}\\
				&=	\q(q^2c_1)E_1\Qkm
\end{align*} 
as required.
\end{proof}


\subsubsection{Type $AII_3$}

Consider the Satake diagram of type $AII_3$.

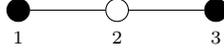
\begin{center}
	\begin{tikzpicture} 
		[white/.style={circle,draw=black,inner sep = 0mm, minimum size = 3mm},
		black/.style={circle,draw=black,fill=black, inner sep= 0mm, minimum size = 3mm}]
	
		\node[black] (first) [label = below:{ \scriptsize $1$}] {};
		\node[white] (second) [right=of first] [label = below:{ \scriptsize $2$}]{}
			edge (first);
		\node[black] (last) [right= of second] [label = below:{ \scriptsize $3$}]{}
			edge (second);		
	\end{tikzpicture}
\end{center}

\begin{lemma} \label{lem:QkmRankOneAII}
The quasi $K$-matrix $\Qkm$ of type $AII_3$ is given by 
\begin{equation} \label{eq:QkmRankOneAII}
	\Qkm	=	\sum_{n \geq 0} \dfrac{(qc_2)^n}{\{n\}!} [E_2, T_{13}(E_2)]_{q^{-2}}^n.
\end{equation}
\end{lemma}

\begin{proof}
Since $T_{13}(E_2) = [E_1, T_3(E_2)]_{q^{-1}}$, Property \eqref{eq:ir} of the skew derivative ${}_1r$ implies that $\ir{1}{T_{13}(E_2)} = (1-q^{-2})T_3(E_2)$. Again by Property \eqref{eq:ir}, it follows that $\ir{1}{[E_2,T_{13}(E_2)]_{q^{-2}}} = 0$. Hence $\ir{1}{\Qkm} = 0$. By symmetry, we also have $\ir{3}{\Qkm} = 0$. 

We want to show that
\begin{equation}
\ir{2}{\Qkm} = (q-q^{-1})c_2T_{13}(E_2)\Qkm. \nonumber
\end{equation}
Since $\ir{2}{T_{13}(E_2)} = 0$ by \cite[8.26, (4)]{b-Jantzen96}, the relation
\begin{equation}
\ir{2}{[E_2,T_{13}(E_2)]_{q^{-1}}} = (1-q^{-2})T_{13}(E_2) \nonumber
\end{equation}
holds in $U_q(\lie{sl}_4)$. 

Moreover, the element $T_{13}(E_2)$ commutes with the element $[E_2, T_{13}(E_2)]_{q^{-2}}$. 
Indeed, this follows from the fact that $E_2$ commutes with $[T_{213}(E_2), E_2]_{q^{-2}}$ by applying the automorphism $T_{13}(E_2)$.
This implies that the relation
\begin{equation}
\ir{2}{[E_2, T_{13}(E_2)]_{q^{-2}}^n} = (1-q^{-2})\{n\}T_{13}(E_2)[E_2, T_{13}(E_2)]_{q^{-2}}^{n-1} \nonumber
\end{equation}
holds in $U_q(\lie{sl}_4)$. Using this, we obtain
\begin{align*}
\ir{2}{\Qkm}	&= \sum_{n \geq 0} \dfrac{(qc_2)^n}{\{n\}!} \ir{2}{[E_2,T_{13}(E_2)]_{q^{-2}}^n} \\
				&= (1-q^{-2})T_{13}(E_2) \sum_{n \geq 1} \dfrac{(qc_2)^n}{\{n-1\}!} [E_2, T_{13}(E_2)]_{q^{-2}}^{n-1}\\
				&= (q-q^{-1})c_2T_{13}(E_2)\Qkm
\end{align*}
as required.
\end{proof}


\subsubsection{Type $AIII_{11}$}

Consider the Satake diagram of type $AIII_{11}$.

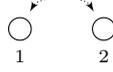
\begin{center}
\begin{tikzpicture} 
		[white/.style={circle,draw=black,inner sep = 0mm, minimum size = 3mm},
		black/.style={circle,draw=black,fill=black, inner sep= 0mm, minimum size = 3mm}]
		
		\node[white] (first) [label = below:{ \scriptsize $1$}] {};		
		\node[white] (third) [right= 0.8cm of first] [label = below:{ \scriptsize $2$}] {}
			edge	 [latex'-latex' , shorten <=3pt, shorten >=3pt, bend right=60, densely dotted] node[auto,swap] {} (first); 
		
\end{tikzpicture}

\end{center}
Note that $s(1) = s(2) = 1$ by \eqref{eq:sCond2} and $c_1 = c_2$ by \eqref{eq:ParameterSetC} and \eqref{eq:ciCond}
.
\begin{lemma} \label{lem:QkmRankOneA1xA1}
The quasi $K$-matrix $\Qkm$ of type $AIII_{11}$ is given by
\begin{equation} \label{eq:QkmRankOneA1xA1}
\Qkm = \sum_{n \geq 0} \dfrac{\q^n}{\{n\}!}c_1^n (E_1E_2)^n.
\end{equation}
\end{lemma}

\begin{proof}
By symmetry, we only need to show that
\begin{equation*}
	\ir{1}{\Qkm} = \q c_1E_2\Qkm.
\end{equation*}
By \eqref{eq:irEi}, we have
\begin{align*}
	\ir{1}{\Qkm}		&= \sum_{n \geq 0} \dfrac{\q^n}{\{n\}!}c_1^n \ir{1}{(E_1E_2)^n}\\
					&= \sum_{n \geq 1} \dfrac{\q^n}{\{n\}!} \{n\} c_1^nE_1^{n-1}E_2^n\\
					&= \q c_1 E_2\Qkm
\end{align*}
as required.
\end{proof}


\subsubsection{Type $AIV$ for $n \geq 2$}
Consider the Satake diagram of type $AIV$ for $n \geq 2$.

\begin{center}
	\begin{tikzpicture}
		[white/.style={circle,draw=black,inner sep = 0mm, minimum size = 3mm},
		black/.style={circle,draw=black,fill=black, inner sep= 0mm, minimum size = 3mm}]
	
		\node[white] (first) [label = below:{ \scriptsize $1$}] {};
		\node[black] (second) [right=of first] [label = below:{ \scriptsize $2$}]{}
			edge (first);
		\node[black] (last) [right=1.5cm of second] {}
			edge [dashed] (second);		
		\node[white] (fourth) [right=of last] [label = below:{ \scriptsize $n$}] {}
			edge (last)
			edge	 [latex'-latex' , shorten <=3pt, shorten >=3pt, bend right=30, densely dotted] node[auto,swap] {} (first); 
	\end{tikzpicture}
\end{center}
By \eqref{eq:sCond2}, we have $s(1) = -s(n)$ and by \eqref{eq:ciCond}, we have $c_1 = q^{-2}\overline{c_n}$.

\begin{lemma} \label{lem:QkmRankOneAIV}
The quasi $K$-matrix $\Qkm$ of type $AIV$ is given by
\begin{equation} \label{eq:QkmRankOneAIV}
	\Qkm	=	\bigg( \sum_{k \geq 0} \dfrac{ (c_1s(n) )^k}{\{k\}!} T_1T_{w_X}(E_n)^k \bigg)
			\bigg( \sum_{k \geq 0} \dfrac{ (c_ns(1) )^k}{\{k\}!} T_nT_{w_X}(E_1)^k \bigg).
\end{equation}
\end{lemma}

\begin{proof}
We have $\ir{i}{\Qkm} = 0$ for $i \in X$. 
Hence by symmetry we only need to show that
\begin{equation*}
	\ir{1}{\Qkm}	=	\q q^{-1}c_1s(n)T_{w_X}(E_n)\Qkm
\end{equation*}
since $T_1T_{w_X}(E_n)$ and $T_nT_{w_X}(E_1)$ commute.
We have
\begin{align*}
	\ir{1}{T_nT_{w_X}(E_1)^k} &= 0,\\
	\ir{1}{T_1T_{w_X}(E_n)^k} &= q^{-1}\q \{k\} T_{w_X}(E_n)T_1T_{w_X}(E_n)^{k-1}.
\end{align*} 
Using this, we obtain
\begin{align*}
	\ir{1}{\Qkm}	&=	\bigg( \sum_{k \geq 0} \dfrac{ (c_1s(n) )^k}{\{k\}!} \ir{1}{T_1T_{w_X}(E_n)^k} \bigg)
					\bigg( \sum_{k \geq 0} \dfrac{ (c_ns(1) )^k}{\{k\}!} T_nT_{w_X}(E_1)^k \bigg)\\
				&=	q^{-1}\q c_1s(n) T_{w_X}(E_n)\Qkm
\end{align*}
as required.
\end{proof}

\begin{remark}\label{rem:Qkm-int-rank1}
  Let $\mathscr{A}=\Z[q,q^{-1}]$ and let ${}_{\sA}U^+$ be the $\sA$-subalgebra of $U^+$ generated by $E_i^{(n)}=\frac{E_i^n}{[n]!}$ for all $n\in \N_0$, $i\in I$. Set ${}_{\sA}\widehat{U^+}=\prod_{\mu\in Q^+}{}_{\sA}U^+_\mu$ where ${}_{\sA}U^+_\mu= {}_{\sA}U^+\cap U^+_\mu$ for all $\mu\in Q^+$. By \cite[Theorem 5.3]{a-BW16} we have $\Qkm\in {}_{\sA}\widehat{U^+}$ if $c_is(\tau(i))\in \pm q^\Z$ for all $i\in I\setminus X$. This integrality property is crucial for the theory of canonical bases of quantum symmetric pairs developed in \cite{a-BW16}.

  We observe that the integrality of the quasi $K$-matrix in rank one can in some cases be read off the explicit formulas given in this section. Indeed, Lemma \ref{lem:QkmRankOneAI}, \ref{lem:QkmRankOneA1xA1} and \ref{lem:QkmRankOneAIV} imply that $\Qkm\in {}_{\sA}\widehat{U^+}$ in the rank one cases of type $AI$, $AIII$ and $AIV$. The rank one case $AII_3$ is more complicated, and Lemma \ref{lem:QkmRankOneAII} does not give an obvious way to see that $\Qkm\in {}_{\sA}\widehat{U^+}$. Nevertheless, $\Qkm$ is also integral in this case, as shown in \cite[A.5]{a-BW16}. Based on the present remark, the integrality of $\Qkm$ in higher rank is discussed in Remark \ref{rem:Qkm-int-higherrank}.
\end{remark}


\subsection{Partial quasi $K$-matrices} \label{sec:PartialQkms}
 All through this section we make the assumption that $\B{s} = \bo = (0, 0, \dotsc ,0) \in \mathcal{S}$. In Section \ref{sec:sGeneral} we discuss the case of general parameters $\B{s} \in \mathcal{S}$.
Recall that the Lusztig automorphisms $T_i$ of $U_q(\lie{g})$ for all $i \in I$ give rise to a representation of $Br(W)$ on $U_q(\lie{g})$.
Since $\widetilde{W}$ is a subgroup of $W$, we obtain algebra automorphisms of $U_q(\lie{g})$ defined by
\begin{align*}
  \widet{T_i} := T_{\widetilde{s_i}} \qquad \mbox{for each $i\in I\setminus X$.}
\end{align*}  
By Theorem \ref{thm: Wtilde Coxeter} and Proposition \ref{prop: length comparison} the algebra automorphisms $\widetilde{T_i}$ give rise to a representation of $Br(\widetilde{W})$ on $U_q(\lie{g})$.

Recall from Section \ref{sec:restricted} that $\widet{\Pi}=\{\widet{\alpha_i}\,|\,i\in I\setminus X\}$ and define $Q(2\Sigma)=2\Z\widet{\Pi}$ and $Q^+(2\Sigma)=2\N_0 \widet{\Pi}$.
By \eqref{eq:irSyst} and the assumption $\B{s}= \bo$ we have
\begin{equation*}
	\ir{i}{\Qkm_{\mu}} = -\q q^{-(\Theta(\alpha_i),\alpha_i)} c_i X_i \Qkm_{\mu - 2\widet{\alpha_i}} \qquad \mbox{ for any $\mu\in Q^+$}.
\end{equation*}
Hence we may consider the quasi $K$-matrix $\Qkm$ as an element in $\prod_{\mu \in Q^+(2\Sigma)}U_{\mu}^+ \subset \mathscr{U}$. For any $w\in W$ define
\begin{align*}
  \widet{U}^+[w]= \bigoplus\limits_{\mu \in Q^+(2\Sigma)} U^+[w]_{\mu}
\end{align*}  
and set $\widet{U}^+ = \bigoplus_{\mu \in Q^+(2\Sigma)} U_{\mu}^+$. Then $\widet{U}^+$ and $\widet{U}^+[w]$ are $\mathbb{K}(q)$-subalgebras of $U^+$ and $U^+[w]$, respectively.
In particular by Equation \eqref{eq:Qkmi in U[si]} we have
\begin{align*}
  \Qkm_i \in \reallywidehat{\widet{U}^+[\widet{s_i}]}=\prod_{\mu\in Q^+(2\Sigma)} \widet{U}^+[\widet{s_i}]_\mu \qquad \mbox{for any $i\in I \setminus X$.} 
\end{align*}  
Let $\field'$ be a field extension of $\field(q)$ which contains $q^{1/2}$ and elements $\widet{c_i}$ such that
\begin{align}\label{eq:citil}
  \widet{c_i}^2=c_i c_{\tau(i)}s(i)s(\tau(i)) \qquad \mbox{for all $i\in I\setminus X$.}
\end{align}  
We extend $\widet{U}^+$ and $\widet{U}^+[w]$ for $w\in W$ to $\mathbb{K}'$-algebras $\widet{U}^+_{1/2} = \mathbb{K}' \otimes_{\mathbb{K}(q)} \widet{U}^+$ and $\widet{U}^+_{1/2}[w] = \mathbb{K}' \otimes_{\mathbb{K}(q)} \widet{U}^+[w]$.
Define an algebra automorphism $\Psi : \widet{U}^+_{1/2} \rightarrow \widet{U}^+_{1/2}$ by
\begin{equation} \label{eq:Psi}
	\Psi(E_{2\widet{\alpha_i}}) = q^{(\widet{\alpha_i}, \widet{\alpha_i})} \widet{c_i}E_{2\widet{\alpha_i}} \qquad \mbox{for all  $E_{2\widet{\alpha_i}} \in U^+_{2\widet{\alpha_i}}$.}
\end{equation}
For each $i \in I\setminus X$ define an algebra homomorphism
\begin{equation*}
  \Omega_i = \Psi \circ \widet{T_i} \circ \Psi^{-1}: \widet{U}^+_{1/2}[\widet{s_i}w_0] \rightarrow \widet{U}^{+}_{1/2}.
\end{equation*}
We consider the restriction of the algebra homomorphism $\Omega_i$ to the subalgebra $\widet{U}^+[\widet{s_i}w_0]$, and we denote this restriction also by $\Omega_i$.
Crucially, by the following proposition, the image of the restriction $\Omega_i$ belongs to $\widet{U}^+$ and does not involve any of the adjoined square roots. 
\begin{proposition} \label{prop:Omega1}
For every $i \in I \setminus X$ the map $\Omega_i: \widet{U}^+[\widet{s_i}w_0]  \rightarrow \widet{U}^+$ is a well defined algebra homomorphism.
\end{proposition} 
\begin{proof}
It remains to show that the image of $\Omega_i$ is contained in $\widet{U}^+$.
Observe that $\widet{T_i}(\widet{U}^+_{\mu}) \subseteq \widet{U}^+_{\widet{s_i}(\mu)}$ for all $\mu \in Q^+(2\Sigma)$.
By Corollary \ref{cor:ReflGrp} we have
\begin{equation*}
	\widet{s_i}(\mu) = \mu - \dfrac{2(\mu, \widet{\alpha_i})}{(\widet{\alpha_i}, \widet{\alpha_i})} \widet{\alpha_i} \qquad \mbox{for all $\mu \in Q^+(2\Sigma)$}.
\end{equation*}
Hence Equation \eqref{eq:Psi} implies that
\begin{equation*}
\Omega_i|_{\widet{U}^+_{\mu}} = q^{-(\mu, \widet{\alpha_i})} \widet{c_i}^{-(\mu,\widet{\alpha_i})/(\widet{\alpha_i},\widet{\alpha_i})} \widet{T_i}|_{\widet{U}^+_{\mu}}.
\end{equation*}
Since $\mu \in Q^+(2\Sigma)$ it follows that the exponent $-(\mu,\widet{\alpha_i})$ is an integer. Moreover, Corollary \ref{cor:ReflGrp} implies that the exponent $-(\mu, \widet{\alpha_i})/(\widet{\alpha_i},\widet{\alpha_i})$ is an integer.

If $i = \tau(i)$ then Equation \eqref{eq:citil} and condition \eqref{eq:sCond1} imply that $\widet{c_i} = \pm c_i$. This implies that the image of $\Omega_i$ is contained in $\widet{U}^+$ in this case.

Suppose instead that $i\in I\setminus X$ satisfies $i \neq \tau(i)$. If additionally $(\alpha_i,\Theta(\alpha_i))=0$, then \eqref{eq:ParameterSetC} implies that $c_i=c_{\tau(i)}$. Moreover in this case $\Theta(\alpha_i)=-\alpha_{\tau(i)}$ and hence $s(i)=s(\tau(i))$ by \eqref{eq:sCond2}. Hence we get $\widet{c_i}=\pm c_i s(i)$ in the case $i\neq \tau(i)$, $(\alpha_i,\Theta(\alpha_i))=0$ which implies that the image of $\Omega_i$ is contained in $\widet{U}^+$ in this case.

Finally, we consider the case that $i\neq \tau(i)$ and $(\alpha_i,\Theta(\alpha_i))\neq 0$. We are then in Case 3 in \cite[p.~17]{a-Letzter08} and
hence the restricted root system $\Sigma$ is of type $(BC)_n$ for $n \geq 1$ and $(\widet{\alpha_i}, \widet{\alpha_i}) = \frac{1}{4}(\alpha_i,\alpha_i)$. Since $\mu \in Q^+(2\Sigma)\subset Q$ we have
\begin{equation*}
  \dfrac{(\mu,\widet{\alpha_i})}{(\widet{\alpha_i}, \widet{\alpha_i})}
  =4 \dfrac{(\mu,\alpha_i)}{(\alpha_i, \alpha_i)}\in 2\mathbb{Z}.
\end{equation*}
Hence the image of $\Omega_i$ is contained in $\widet{U}^+$ in all cases as required.
\end{proof}



Consider $\widet{w} \in \widetilde{W}$ and let $\widet{w} = \widetilde{s_{i_1}}\widetilde{s_{i_2}} \dots \widetilde{s_{i_t}}$ be a reduced expression.
For $k = 1, \dotsc, t$ let
\begin{equation} \label{eq:PartialQkmPart}
	\Qkm_{\widet{w}}^{[k]}		=	\Omega_{i_1}\Omega_{i_2} \dotsm \Omega_{i_{k-1}}(\Qkm_{i_k})
					=	\Psi \circ \widetilde{T_{i_1}} \dotsm \widetilde{T_{i_{k-1}}} \circ \Psi^{-1}(\Qkm_{i_k}).
\end{equation}
By Proposition \ref{prop: length comparison} we have  $U^+[\widet{s_{i_k}}] \subset U^+[\widet{s_{i_{k-1}}}w_0]$ for $k = 2, \dotsc, t$, and 
\begin{equation*}
	\widet{T_{i_l}} \dotsm \widet{T_{i_{k-1}}}(U^+[\widet{s_{i_k}}]) \subset U^+[\widet{s_{i_{l-1}}}w_0] \qquad \mbox{for $l=2, \dotsc , k-1$}
\end{equation*}
and hence the elements $\Qkm_{\widet{w}}^{[k]}$ are well-defined.
Moreover, by Proposition \ref{prop:Omega1} we have
\begin{equation*}
  \Qkm_{\widet{w}}^{[k]} \in \reallywidehat{\widet{U}^+[\widet{w}]}=\prod_{\mu\in Q^+(2\Sigma)} U^+[\widet{w}]_\mu \qquad \mbox{for $k = 1, \dotsc , t.$}
\end{equation*}
When clear, we omit the subscript $\widet{w}$ and write $\Qkm^{[k]}$ instead of  $\Qkm_{\widet{w}}^{[k]}$.
\begin{definition}
   Let $\widet{w} \in \widetilde{W}$ and let $\widet{w} = \widetilde{s_{i_1}}\widetilde{s_{i_2}} \dots \widetilde{s_{i_t}}$ be a reduced expression. The partial quasi $K$-matrix $\Qkm_{\widet{w}}$ associated to $\widet{w}$ and the given reduced expression is defined by 
\begin{equation} \label{eq:PartialQkm}
	\Qkm_{\widet{w}}	=	\Qkm^{[k]}\Qkm^{[k-1]} \dotsm \Qkm^{[2]}\Qkm^{[1]}.
\end{equation}
\end{definition}
We expect that the partial quasi $K$-matrix $\Qkm_{\widet{w}}$ only depends on $\widet{w} \in \widet{W}$ and not on the chosen reduced expression. As we will see in Theorem
 \ref{Thm:HigherRankPartialQkm} it suffices to check the independence of the reduced expression in rank two. If the Satake diagram is of rank two then the restricted Weyl group $\widetilde{W}$ is of one of the types $A_1 \times A_1$, $A_2$, $B_2$ or $G_2$. In each case, only the longest word for $\widetilde{W}$ has distinct reduced expressions.

\begin{conjecture} \label{conj}
	Assume that $\satake$ is a Satake diagram of rank two.
	Then the element $\Qkm_{\widet{w}} \in \mathscr{U}$ defined by \eqref{eq:PartialQkm} depends only on $\widet{w} \in \widetilde{W}$ and not on the chosen reduced expression.
\end{conjecture}
In Appendix \ref{App:RankTwo}, we prove the following Theorem which confirms Conjecture \ref{conj} in many cases. The proof is performed by showing that for both reduced expressions of the longest element $\widet{w_0}\in\widet{W}$ the resulting elements $\Qkm_{\widet{w_0}}$ satisfy the relations \eqref{eq:irSyst2}.

\begin{theorem} \label{Thm:RankTwoPartialQkm}
	Assume that $\lie{g} = \mathfrak{sl}_n(\mathbb{C})$ or $X = \emptyset$.
	Then Conjecture \ref{conj} holds.
\end{theorem}
\begin{remark}
The Hopf algebra automorphism $\Psi$ in the definition of $\Omega_i$ turns out to be necessary for the rank two calculations in Appendix \ref{App:RankTwo} which prove Theorem \ref{Thm:RankTwoPartialQkm}. The conjugation by $\Psi$ affects the coefficients in the partial quasi $K$-matrix associated to a reduced expression of an element $\widet{w} \in \widet{W}$. In rank two the two partial quasi $K$-matrices associated to the longest word $\widet{w_0}\in \widet{W}$ coincide only after this change of coefficients. The effect of the conjugation by $\Psi$ can be seen in particular in Sections \ref{app:AIII3} and \ref{app:AIIIge4} of the appendix which treat type $AIII_n$ for $n\ge 3$.
\end{remark}
Once the rank two case is established, we can generalise to higher rank cases.

\begin{theorem} \label{Thm:HigherRankPartialQkm}
	Suppose that $\satake$ is a Satake diagram such that all subdiagrams $\subsatake$ of rank two satisfy Conjecture \ref{conj}.
	Then the element $\Qkm_{\widet{w}} \in \mathscr{U}$ depends on $\widet{w} \in \widetilde{W}$ and not on the chosen reduced expression. 
\end{theorem}

\begin{proof}
Let $\widet{w}$ and $\widet{w}^{\prime}$ be reduced expressions which represent the same element in $\widet{W}$. 
Assume that $\widet{w}$ and $\widet{w}^{\prime}$ differ by a single braid relation.	
The following are the possible braid relations:
\begin{align}
	\widet{s_p}\widet{s_r}	&=	\widet{s_r}\widet{s_p}, \nonumber \\
	\widet{s_p}\widet{s_r}\widet{s_p}	&=	\widet{s_r}\widet{s_p}\widet{s_r},  \label{Prf:Star}\\
	(\widet{s_p}\widet{s_r})^2	&=	(\widet{s_r}\widet{s_p})^2, \nonumber\\
	(\widet{s_p}\widet{s_r})^3	&=	(\widet{s_r}\widet{s_p})^3. \nonumber
\end{align} 
The argument for each relation is the same, so we only consider the second case. Assume that $\widet{w}$ and $\widet{w}^{\prime}$ differ by relation \eqref{Prf:Star}, that is
\begin{align*}
	\widet{w}	&=	\widet{s_{i_1}} \dotsm \widet{s_{i_{k-1}}} \big(\widet{s_p}\widet{s_r}\widet{s_p} \big) \widet{s_{i_{k+3}}} \dotsm \widet{s_{i_t}},\\
	\widet{w}^{\prime}	&=	\widet{s_{i_1}} \dotsm \widet{s_{i_{k-1}}} \big(\widet{s_r}\widet{s_p}\widet{s_r} \big) \widet{s_{i_{k+3}}} \dotsm \widet{s_{i_t}}
\end{align*}
for some $k = 1, \dotsc, t-2$.
For $l = 1, \dotsc ,k-1$, we have
\begin{align*}
	\Qkm_{\widet{w}}^{[l]}	&=	\Psi \circ \widet{T_{i_1}} \dotsm \widet{T_{i_{l-1}}} \circ \Psi^{-1}(\Qkm_{i_l})	=	\Qkm_{\widet{w}^{\prime}}^{[l]}.
\end{align*}
Since the algebra automorphisms $\widet{T_i}$ satisfy braid relations, we have
\begin{align*}
\Qkm_{\widet{w}}^{[l]}	&=	\Psi \circ \widet{T_{i_1}} \dotsm \widet{T_{i_{k-1}}} \big( \widet{T_p}\widet{T_r}\widet{T_p} \big) \widet{T_{i_{k+3}}} \dotsm \widet{T_{i_{l-1}}} \circ \Psi^{-1}(\Qkm_{i_l})\\
	&=\Psi \circ \widet{T_{i_1}} \dotsm \widet{T_{i_{k-1}}} \big( \widet{T_r}\widet{T_p}\widet{T_r} \big) \widet{T_{i_{k+3}}} \dotsm \widet{T_{i_{l-1}}} \circ \Psi^{-1}(\Qkm_{i_l}) = \Qkm_{\widet{w}^{\prime}}^{[l]}
\end{align*}
for $l = k+3, \dotsc ,t$.
Finally, consider the rank two subdiagram $\subsatake$ obtained by taking $J = J_1 \cup J_2$, where $J_1 = \{r, p, \tau(r), \tau(p)\}$ and $J_2 \subset X$ is the union of connected components of $X$ which are connected to a node of $J_1$.
By assumption,
\begin{align*}
\Qkm_{\widet{s_p}\widet{s_r}\widet{s_p}}	&=	\Qkm_p \cdot \Psi\widet{T_p}\Psi^{-1}(\Qkm_r) \cdot \Psi\widet{T_p}\widet{T_r}\Psi^{-1}(\Qkm_p)\\
										&=	\Qkm_r \cdot \Psi\widet{T_r}\Psi^{-1}(\Qkm_p) \cdot \Psi\widet{T_r}\widet{T_p}\Psi^{-1}(\Qkm_r)
										=	\Qkm_{\widet{s_r}\widet{s_p}\widet{s_r}}.
\end{align*}
It follows from this that
\begin{align*}
	\Qkm_{\widet{w}}^{[k]}\Qkm_{\widet{w}}^{[k+1]}\Qkm_{\widet{w}}^{[k+2]}
			&=	\Psi\widet{T_{i_1}} \dotsm \widet{T_{i_{k-1}}}\Psi^{-1} \big(\Qkm_{\widet{s_p}\widet{s_r}\widet{s_p}} \big)\\
			&=	\Psi\widet{T_{i_1}} \dotsm \widet{T_{i_{k-1}}}\Psi^{-1} \big(\Qkm_{\widet{s_r}\widet{s_p}\widet{s_r}} \big)\\
			&=	\Qkm_{\widet{w}^{\prime}}^{[k]}\Qkm_{\widet{w}^{\prime}}^{[k+1]}\Qkm_{\widet{w}^{\prime}}^{[k+2]}
\end{align*}
Hence we have $\Qkm_{\widet{w}} = \Qkm_{\widet{w}^{\prime}}$ as required.

If $\widet{w}$ and $\widet{w}^{\prime}$ differ by more than a single relation, then we can find a sequence of reduced expressions
\begin{equation*}
	\widet{w} = \widet{w_1}, \widet{w_2}, \dotsc, \widet{w_n} = \widet{w}^{\prime}
\end{equation*}
such that for each $i = 1, \dotsc ,n-1$, the expressions $\widet{w_i}$ and $\widet{w_{i+1}}$ differ by a single relation.
We repeat the above argument at each step and obtain $\Qkm_{\widet{w}} = \Qkm_{\widet{w}^{\prime}}$.
\end{proof}

Recall from \eqref{eq:w0Action} that there exists a diagram automorphism $\tau_0: I \rightarrow I$ such that the longest element $w_0 \in W$ satisfies
\begin{equation}
	w_0(\alpha_i)	=	-\alpha_{\tau_0(i)}
\end{equation}
for all $i \in I$.
Denote the longest element of $\widet{W}$ by $\widet{w_0}$.

\begin{proposition} \label{prop:LongestWordPartialQkmPart}
Let $\widet{w_0} \in \widet{W}$ be the longest word with reduced expression $\widet{w_0} = \widet{s_{i_1}} \dotsm \widet{s_{i_t}}$. 
Then
\begin{equation} \label{eq:LongestWordPartialQkmPart}
	\Qkm_{\widet{w_0}}^{[t]}	=	\Qkm_{\tau_0(i_t)}.
\end{equation}
\end{proposition}

\begin{proof}
To simplify notation we write $i_t = i$.
By construction we have
\begin{align*}
	\widet{w_0}w_X &= w_0, \\
	w_X\widet{s_i} &= w_{\{i,\tau(i)\} \cup X} \quad \mbox{for all $i \in I \setminus X$.}
\end{align*}
As $w_X$ and $w_{ \{i, \tau(i)\} \cup X}$ commute, we get
\begin{align*}
\Qkm_{\widet{w_0}}^{[t]}
	&=	\Psi \circ T_{\widet{w_0}} T_{\widet{s_{i}}}^{-1} \circ \Psi^{-1} (\Qkm_{i})\\
	&=	\Psi \circ T_{w_0}T_{w_X}^{-1} T_{w_X} T_{w_{\{i,\tau(i)\} \cup X}}^{-1} \circ \Psi^{-1}(\Qkm_{i})\\
	&=	\Psi \circ T_{w_0}T_{w_{\{i,\tau(i)\} \cup X}}^{-1} \circ \Psi^{-1}(\Qkm_{i}).
\end{align*}
Recall that
\begin{equation}\label{eq:Tw0AltDef}
	T_{w_0} = \mathrm{tw}^{-1} \circ \tau_0 
\end{equation}
where $\mathrm{tw} : U_q(\lie{g}) \rightarrow U_q(\lie{g})$ is the algebra automorphism defined by
\begin{equation*}
  \mathrm{tw}(E_i) = -K_i^{-1}F_i, \quad \mathrm{tw}(F_i) = -E_iK_i, \quad \mathrm{tw}(K_i) = K_i^{-1}
\end{equation*}
for $i \in I$, see \cite[Section 7.1]{a-BK15}.
Analogously we have on $U_q(\lie{g}_{\{i,\tau(i)\} \cup X})$ the relation
\begin{equation}
	T_{w_{\{i, \tau(i)\}\cup X}} = \mathrm{tw}^{-1} \circ \tau_{0,i} = \tau_{0,i} \circ \mathrm{tw}^{-1}
\end{equation}
where $\tau_{0,i}: \{i,\tau(i)\} \cup X \rightarrow \{i, \tau(i)\} \cup X$ is the diagram automorphism satisfying \eqref{eq:w0iAction}.
We obtain
\begin{align}
	\Qkm_{\widet{w_0}}^{[t]}
		&=	\Psi \circ \mathrm{tw}^{-1} \circ  \tau_0 \circ \mathrm{tw} \circ \tau_{0,i} \circ \Psi^{-1}(\Qkm_{i})\nonumber \\
		&=	\Psi \circ \tau_0\tau_{0,i} \circ \Psi^{-1} (\Qkm_{i}). \label{eq: tauRankOneQkm}
\end{align}
\begin{case}
$\tau(i) = i$.
\end{case}
In this case Lemma \ref{lem:DiagAutofQkm} implies that
\begin{equation} \label{eq:tau0i_Qkm}
	\tau_{0,i}(\Qkm_i)	=	\Qkm_i.
\end{equation}
Moreover $s(\tau(i)) = s(i) = 1$ and hence by Lemma \ref{lem:RankOneQkm} and by definition of $\Psi$ we have
\begin{align*}
\Qkm_{\widet{w_0}}^{[t]}	
	&\overset{\phantom{\eqref{eq:tau0i_Qkm}}}{=}	\sum_{n \in \mathbb{N}_0} \Psi \circ \tau_0\tau_{0,i} \circ \Psi^{-1} (c_i^nE_{n(\alpha_i - \Theta(\alpha_i))})\\
	&\overset{\phantom{\eqref{eq:tau0i_Qkm}}}{=}	\sum_{n \in \mathbb{N}_0} q^{-n/2(\alpha_i - \Theta(\alpha_i), \alpha_i)} \Psi \circ \tau_0\tau_{0,i}(E_{n(\alpha_i - \Theta(\alpha_i))})\\
	&\overset{\eqref{eq:tau0i_Qkm}}{=} \sum_{n \in \mathbb{N}_0} q^{-n/2(\alpha_i - \Theta(\alpha_i),\alpha_i)} \Psi \circ \tau_0 (E_{n(\alpha_i - \Theta(\alpha_i))})\\
	&\overset{\phantom{\eqref{eq:tau0i_Qkm}}}{=} \sum_{n \in \mathbb{N}_0} q^{-n/2(\alpha_i - \Theta(\alpha_i),\alpha_i)} \Psi (E_{n(\alpha_{\tau_0(i)} - \Theta(\alpha_{\tau_0(i)})})				
\end{align*}
where we use the notation from Lemma \ref{lem:RankOneQkm} also for $\Qkm_{\tau_0(i)}$.

As $(\alpha_{\tau_0(i)} - \Theta(\alpha_{\tau_0(i)}), \alpha_{\tau_0(i)})~= (\alpha_i - \Theta(\alpha_i), \alpha_i)$ and $s(\tau_0(i))=1$, formula \eqref{eq:Psi} gives us
\begin{equation}
	\Qkm_{\widet{w_0}}^{[t]} = \sum_{n \in \mathbb{N}_0} c_{\tau_0(i)}^n E_{n(\alpha_{\tau_0(i)} - \Theta(\alpha_{\tau_0(i)}))}	=	\Qkm_{\tau_0(i)}
\end{equation}
which proves the Lemma in this case.

\begin{case}
$\tau(i) \neq i$.
\end{case}
In this case the rank one Satake subdiagram is either of type $AIII_{11}$ or of type $AIV$ for $n \geq 2$ as in Table \ref{Table:RankOne}.

If the rank one Satake subdiagram is of type $AIV$ for $n \geq 2$ then $\tau = \tau_0$ and $\tau_{0,i}$ coincide on $\{i, \tau(i)\} \cup X$ and hence \eqref{eq: tauRankOneQkm} implies that
\begin{equation}
	\Qkm_{\widet{w_0}}^{[t]}	=	\Qkm_i	=	\Qkm_{\tau\tau_{0}(i)} = \Qkm_{\tau_0(i)}.
\end{equation} 
If the rank one subdiagram is of type $AIII_{11}$ then $\tau_{0,i}(i) = i$.
If additionally $\tau_0(i) = i$ then $\tau_0(\tau(i)) = \tau(i)$ and hence \eqref{eq: tauRankOneQkm} implies that $\Qkm_{\widet{w_0}}^{[t]} = \Qkm_i = \Qkm_{\tau_0(i)}$ in this case.

If on the other hand $\tau_0(i) \neq i$ then $\tau_0 = \tau$ and we invoke the fact that
\begin{equation} \label{eq:sscc}
s(i) = s(\tau(i)), \quad c_i = c_{\tau(i)}
\end{equation}
which hold by \eqref{eq:sCond2} and \eqref{eq:ParameterSetC}. Relation \eqref{eq:sscc} and $\tau = \tau_0$ imply that $\tau_0 \circ \Psi^{-1}(\Qkm_i) = \Psi^{-1} \circ \tau_0(\Qkm_i)$. Hence Equation \eqref{eq: tauRankOneQkm} implies that $\Qkm_{\widet{w_0}}^{[t]} = \Qkm_{\tau_0(i)}$ also in this case. 
\end{proof}

\begin{lemma}
Let $\widet{w_0} = \widet{s_{i_1}} \dotsm \widet{s_{i_t}}$ be a reduced expression for the longest word in $\widet{W}$. Then $\Qkm_{\widet{w_0}}^{[i]} \in \widehat{U^+[\widet{s_{k}}\widet{w_0}]}$ for $i = 1, \dotsc , t-1$ and $k = \tau_0(i_t)$.
\end{lemma}

\begin{proof}
We have 
\begin{equation*}
\widet{s_k}\widet{w_0} = \widet{s_k}w_0w_X = w_0\widet{s_{\tau_0(k)}}w_X = w_0w_X\widet{s_{\tau_0(k)}} = \widet{w_0}\widet{s_{i_t}} = \widet{s_{i_1}} \dotsm \widet{s_{i_{t-1}}}.
\end{equation*}
By definition of $U^+[w]$ for each $w \in W$ and Proposition \ref{prop: length comparison} we have
\begin{equation*}
\widet{T_{i_1}} \dotsm \widet{T_{i_{j-1}}}( U^+[\widet{s_{i_j}}] ) \subseteq U^+[\widet{s_k}\widet{w_0}]
\end{equation*}
for $j = 1, \dotsc, t-1$. 
Now the claim of the lemma follows from Equation \eqref{eq:Qkmi in U[si]},  Proposition \ref{prop:Omega1} and the fact that
\begin{equation*}
\Qkm_{\widet{w_0}}^{[j]} = \Psi \circ \widet{T_{i_1}} \dotsm \widet{T_{i_{j-1}}} \circ \Psi^{-1} ( \Qkm_{i_j})
\end{equation*}
for $j = 1, \dotsc , t-1$.
\end{proof}

With the above preparations we are ready to prove the main result of the paper.

\begin{theorem} \label{Thm:LongestWordPartialQkm}
Suppose that $\satake$ is a Satake diagram such that all subdiagrams $\subsatake$ of rank two satisfy Conjecture \ref{conj}.
Then $\Qkm_{\widet{w_0}}$ coincides with the quasi $K$-matrix $\Qkm$.
\end{theorem}

\begin{proof}
It suffices to show that 
\begin{equation} \label{eq:irLongPartial}
	\ir{i}{\Qkm_{\widet{w_0}}}	=	(q-q^{-1})q^{-(\Theta(\alpha_i), \alpha_i)}c_is(\tau(i))T_{w_X}(E_{\tau(i)})\Qkm_{\widet{w_0}} 
\end{equation}
for all $i \in I \setminus X$.
By Theorem \ref{Thm:HigherRankPartialQkm}, we can choose any reduced expression $\widet{w_0} = \widet{s_{i_1}} \dotsm \widet{s_{i_t}}$ of the longest element of $\widet{W}$.
Proposition \ref{prop:LongestWordPartialQkmPart} implies that
\begin{equation*}
	\Qkm_{\widet{w_0}}	=	\Qkm_{\tau_0(i_t)} \Qkm^{[t-1]} \dotsm \Qkm^{[2]}\Qkm^{[1]}.
\end{equation*} 
Suppose $\tau_0(i_t) \in I$ is a representative of the $\tau$-orbit $\{k, \tau(k)\}$ for some $k \in I \setminus X$.

By the previous Lemma we have $\Qkm^{[i]} \in \widehat{U^{+}[\widet{s_k}\widet{w_0}]}$ for $i = 1, \dotsc, t-1$.
By \cite[8.26, (4)]{b-Jantzen96} this implies that $\ir{k}{\Qkm^{[i]}} = 0$ for $i = 1, \dotsc , t-1$.
By Equation \ref{eq:irSyst2}, we have
\begin{equation*}
	\ir{k}{\Qkm_{\tau_0(i_t)}} = (q-q^{-1})q^{-(\Theta(\alpha_k),\alpha_k)}c_ks(\tau(k))T_{w_X}(E_{\tau(k)})\Qkm_{\tau_0(i_t)}
\end{equation*}
and similarly an expression for $\ir{\tau(k)}{\Qkm_{\tau_0(i_t)}}$.
Equation \eqref{eq:irLongPartial} for $k, \tau(k)$ follows from the above and the skew derivation property \eqref{eq:ir}.
Since we can arbitrarily choose the reduced expression for $\widet{w_0}$, the result follows.
\end{proof}

Combining Theorems \ref{Thm:RankTwoPartialQkm} and \ref{Thm:LongestWordPartialQkm} we obtain the following result. 

\begin{corollary} \label{cor:Qkm}
Let $\lie{g}$ be of type $A$ or $X = \emptyset$. Then the quasi $K$-matrix $\Qkm$ is given by $\Qkm = \Qkm_{\widet{w_0}}$ for any reduced expression of the longest word $\widet{w_0} \in \widet{W}$. 
\end{corollary}

\begin{conjecture}\label{conj:anyType}
 The statement of Corollary \ref{cor:Qkm} holds for any Satake diagram of finite type.
\end{conjecture}

\begin{remark} \label{rem:Qkm-int-higherrank}
  We continue the discussion of the integrality of the quasi $K$-matrix $\Qkm$ from Remark \ref{rem:Qkm-int-rank1} under the assumption that $c_is(\tau(i))\in \pm q^\Z$ for all $i\in I\setminus X$. In this case $\Qkm^{[k]}_{\widet{w}} \in {}_{\sA}\widehat{U^+}$ for $k=1, \dots, t$ if $\widet{w} \in \widet{W}$ has a reduced expression $\widet{w}=\widet{s_{i_1}} \widet{s_{i_2}} \dots \widet{s_{i_t}}$. Indeed, the discussion in the proof of Proposition \ref{prop:Omega1} shows that  $\Qkm^{[k]}_{\widet{w}}$ differs from $\widet{T_{i_1}} \dots \widet{T_{i_1}}(\Qkm_{i_k})$ by a factor in $\pm q^\Z$. Hence we obtain $\Qkm_{\widet{w}}\in {}_{\sA}\widehat{U^+}$ for all $\widet{w} \in \widet{W}$. By Corollary \ref{cor:Qkm}, choosing $\widet{w}=\widet{w_0}$, we obtain $\Qkm\in {}_{\sA}\widehat{U^+}$ whenever $\lie{g}$ is of type $A$ or $X=\emptyset$. In these cases we have hence reproduced \cite[Theorem 5.3]{a-BW16} for $\bs=\bo$ without the use of canonical bases. The case of general Satake diagrams hinges on Conjecture \ref{conj:anyType} and the integrality in rank one from \cite[Appendix A]{a-BW16}.
\end{remark}  

\subsection{Quasi $K$-matrices for general parameters}\label{sec:sGeneral}
We now give a description of the quasi $K$-matrix $\Qkm$ for general parameters $\B{s}\in \mathcal{S}$. By \cite[Theorem 7.1]{a-Letzter03}, \cite[Theorem 7.1]{a-Kolb14} there exists an algebra isomorphism $\varphi_{\B{s}}:\Bco\rightarrow \Bcs$ given by
\begin{equation*}
  \varphi_{\B{s}}(B_i^{\bc,\bo})=B_i^{\bc,\bs}, \quad \varphi_{\B{s}}(b)=b \qquad \mbox{for all $i\in I \setminus X$, $b\in \cM_X U^0_\Theta$.}
\end{equation*}
This algebra isomorphism allows us to define a one dimensional representation
$\chi_\bs:\Bco\rightarrow \field(q)$ by $\chi_\bs=\varepsilon\circ \varphi_\bs$.
By definition we have
\begin{align*}
  \chi_{\bs}(B_i^{\bc,\bo})=s_i \quad \mbox{for all $i\in I \setminus X$,} \qquad  \chi_{\bs}|_{\cM_X U^0_\Theta}=\varepsilon|_{\cM_X U^0_\Theta}.
\end{align*}
By \cite[(5.5)]{a-Kolb14} we have
\begin{equation}\label{eq:kowBi}
  \Delta(B_i) - B_i\otimes K_i^{-1} \in \cM_X U^0_\Theta\otimes \uqg
\end{equation}
which implies that
\begin{equation}\label{eq:varphis}
  \varphi_{\bs}=(\chi_\bs\otimes \mathrm{id})\circ \Delta
\end{equation}
on $\Bco$. For later use we observe the following compatibility with the bar involution.

\begin{lemma}\label{lem:bar-compatible}
  For all $b\in \Bco$ we have
  \begin{equation}\label{eq:bar-compatible}
     (\chi_\bs \otimes \mathrm{id})\circ (\Bbarinvo\otimes \barinv)\circ \Delta(b) = \overline{\varphi_\bs(b)}^{U}.
  \end{equation}  
\end{lemma}

\begin{proof}
  As $(\chi_\bs \otimes \mathrm{id})\circ (\Bbarinvo\otimes \barinv)\circ \Delta$ and $\barinv\circ \varphi_\bs$ are $\field$-algebra homomorphisms, it suffices to check Equation \eqref{eq:bar-compatible} on the generators $B_i^{\bc,\bo}$ for $i\in I\setminus X$ and on $\cM_X U^0_\Theta$. If $b\in \cM_X U^0_\Theta$ then both sides of \eqref{eq:bar-compatible} coincide with $\overline{b}^U$. If $b=B_i^{\bc, \bo}$ for some $i\notin \{j\in I_{ns}\,|\, a_{jk}\in -2\N_0 \mbox{ for all }k \in I_{ns}\setminus\{j\}\}$ then $s_i=0$ by the definition of $\cS$ in \eqref{eq:ParameterSetS} and hence $B_i^{\bc,\bo}=B_i^{\bc,\bs}$. Using the membership property \eqref{eq:kowBi} we get
  \begin{equation*}
((\chi_\bs\circ\Bbarinvo)\otimes \barinv)\circ \Delta(B_i^{\bc,\bo})=((\barinv\circ \varepsilon)\otimes \barinv)\circ\Delta(B_i^{\bc,\bo})=\overline{B_i^{\bc,\bo}}^{U}=\overline{\varphi_\bs(b)}^{U}
  \end{equation*}
  which proves \eqref{eq:bar-compatible} in this case. Finally, if $i\in \{j\in I_{ns}\,|\, a_{jk}\in -2\N_0 \mbox{ for all }k \in I_{ns}\setminus\{j\}\}$, then the definition \eqref{eq:Ins} of $I_{ns}$ implies that
  \begin{align*}
    B_i^{\bc,\bs}=F_i-c_iE_i K_i^{-1} + s_i K_i^{-1}.
  \end{align*}
Hence, using $s_i=\overline{s_i}^U$ from \eqref{eq:siCond}, we get
  \begin{align*}
    ((\chi_\bs\circ\Bbarinvo)\otimes \barinv)\circ \Delta(B_i^{\bc,\bo})&=  ((\chi_\bs\circ\Bbarinvo)\otimes \barinv)(B_i^{\bc,\bo} \otimes K_i^{-1} + 1 \otimes B_i^{\bc,\bo})\\
    &=s_i K_i + \overline{B_i^{\bc,\bo}}^U\\
    &=\overline{B_i^{\bc,\bs}}^U\\
    &=\overline{\varphi_\bs(B_i^{\bc,\bo})}^U
  \end{align*}  
  which completes the proof of the lemma.
\end{proof}
As in \cite[3.2]{a-BK15} we consider the algebra
\begin{align*}
  \mathscr{U}^{(2)}_0=\mathrm{End}(\mathcal{F}or\circ \otimes:\mathcal{O}_{int}\times \mathcal{O}_{int} \rightarrow \mathcal{V}ect)
\end{align*}  
and observe that $\prod_{\mu\in Q^+} U^-_\mu\otimes U^+_\mu$ is a subalgebra of $\mathscr{U}^{(2)}_0$. Let $R\in \prod_{\mu\in Q^+} U^-_\mu\otimes U^+_\mu$ be the quasi $R$-matrix for $\uqg$, see \cite[Theorem 4.1.2]{b-Lusztig94}. Following \cite[3.1]{a-BaoWang13} we define an element
\begin{align}\label{eq:Rtheta}
  R^\theta = \Delta(\Qkm)\cdot R \cdot (\Qkm^{-1} \otimes 1) \in \mathscr{U}_{0}^{(2)},
\end{align}  
see also \cite[Section 3.3]{a-Kolb17p}. In \cite{a-BaoWang13} the element $R^\theta$ is called the quasi $R$-matrix for $\Bcs$. By \cite[Proposition 3.2]{a-BaoWang13} it satisfies the following intertwiner property
  \begin{align}\label{eq:Rtheta-intertwine}
    \Delta(\overline{b}^B)\cdot R^\theta = R^\theta \cdot (\Bbarinv\otimes \barinv)\circ \Delta(b) \qquad \mbox{for all $b\in B = \Bcs$}
  \end{align}
  in $\sU^{(2)}_0$. Moreover, by \cite[Proposition 3.5]{a-BaoWang13}, \cite[Proposition 3.6]{a-Kolb17p} we can write $R^\theta$ as an infinite sum
  \begin{align}\label{eq:Rtheta-sum}
    R^\theta=\sum_{\mu\in Q^+}R^\theta_\mu \qquad \mbox{ with $R^\theta_\mu\in \Bcs\otimes U^+_\mu$.}
  \end{align}
  Similarly to the notation $\Qkm_{\bc,\bs}$ introduced at the end of Section \ref{sec:QSPs}, we write $R^\theta_{\bc,\bs}$ if we need to specify the dependence on the parameters. Observe that once we have an explicit formula for $\Qkm_{\bc,\bo}$, Equation \eqref{eq:Rtheta} provides us with an explicit formula for $R^\theta_{\bc,\bo}$. This in turn provides a formula for the quasi $K$-matrix $\Qkm_{\bc,\bs}$ for general parameters $\bs\in \mathcal{S}$. Indeed, by Equation \eqref{eq:Rtheta-sum} we can apply the character $\chi_\bs$ to the first tensor factor of $R^\theta_{\bc,\bo}$ to obtain an element
  $\Qkm'=(\chi_\bs\otimes \mathrm{id})(R^{\theta}_{\bc,\bo})$ which can be written as
  \begin{align*}
    \Qkm'=\sum_{\mu\in Q^+}\Qkm'_\mu \qquad \mbox{ with $\Qkm'_\mu\in U_{\mu}^+$.}
  \end{align*}
Moreover, Equation \eqref{eq:Rtheta-sum} implies that $\Qkm'_0=1$. By the following proposition the element $\Qkm'\in \sU$ is the quasi $K$-matrix for $\Bcs$. 
\begin{proposition}\label{prop:os}
  For any $\bc\in \mathcal{C}$, $\bs\in \cS$ we have $\Qkm_{\bc,\bs}=(\chi_{\bs} \otimes \mathrm{id})(R^{\theta}_{\bc,\bo})$.
\end{proposition}
\begin{proof}
  We keep the notation  $\Qkm'=(\chi_\bs\otimes \mathrm{id})(R^{\theta}_{\bc,\bo})$ from above. By Equation \eqref{eq:Rtheta-intertwine} we have
  \begin{align*}
    \Delta(\overline{b}^{B_{\bc,\bo}})\cdot R^\theta_{\bc,\bo}=R^\theta_{\bc,\bo}\cdot  (\Bbarinvo \otimes \barinv)\circ \Delta(b) \qquad \mbox{for all $b\in B_{\bc,\bo}$}.
  \end{align*}
  Applying $\chi_\bs\otimes \mathrm{id}$ to both sides of this relation, we obtain in view of Equation \eqref{eq:varphis} the relation
  \begin{align*}
     \varphi_\bs(\overline{b}^{B_{\bc,\bo}}) \cdot \Qkm' = \Qkm'  (\chi_\bs \otimes \mathrm{id})\circ (\Bbarinvo\otimes \barinv)\circ \Delta(b) \qquad \mbox{ for all $b\in B_{\bc,\bo}$.}
  \end{align*}
  By Lemma \ref{lem:bar-compatible} the above relation implies that
  \begin{align*}
    \varphi_\bs(\overline{b}^{B_{\bc,\bo}}) \Qkm'= \Qkm'\overline{\varphi_\bs(b)}^U \qquad \mbox{for all $b\in B_{\bc,\bo}$.}  
  \end{align*}
  This gives in particular $B_i^{\bc,\bs} \Qkm' = \Qkm' \overline{B_i^{\bc,\bs}}^U$ for all $i\in I$ and $b \Qkm'=\Qkm' b$ for all $b\in \cM_XU^\Theta_0$. This means that $\Qkm'$ satisfies the defining relation \eqref{eq:BbarIntertwiner} of $\Qkm_{\bc,\bs}$ and hence, in view of the normalisation $\Qkm'_0=1\otimes 1$ observed above, we get $\Qkm'=\Qkm_{\bc,\bs}$. 
\end{proof}  
\begin{remark}
  The existence of the quasi $K$-matrix $\Qkm_{\bc,\bs}$ was established in \cite[Theorem 6.10]{a-BK15} by fairly involved calculations. It was noted in \cite[Remark 6.9]{a-BK15} that these calculations simplify significantly if one restricts to the case $\bs=\bo$. Proposition \ref{prop:os} now shows that in the presence of \eqref{eq:Rtheta-sum} the existence of $\Qkm_{\bc,\bo}$ implies the existence of $\Qkm_{\bc,\bs}$ for any $\bs\in \mathcal{S}$ satisfying \eqref{eq:siCond}. Relation \eqref{eq:Rtheta-sum} was established in \cite{a-Kolb17p} for $\lie{g}$ of finite type.
\end{remark}


\appendix

\section{Rank two calculations} \label{App:RankTwo}

In rank two, there are two distinct reduced expressions for the longest word $\widet{w_0} \in \widet{W}$. All irreducible rank two Satake diagrams for simple $\lie{g}$ are shown in Table \ref{Table:RankTwo}. Using the explicit formulas from Section \ref{sec:rank1Qkms}, we confirm that the partial quasi $K$-matrices for the two reduced expressions of $\widet{w_0}$ coincide with the quasi $K$-matrix if the rank two Satake diagram is of type $A_n$ or $B_2$. 
We have also performed the calculation in type $G_2$. As this case involves significantly longer calculations, we do not include it here, but the details will appear in \cite{Dob18}. The calculations in this appendix and \cite{Dob18} prove Theorem \ref{Thm:RankTwoPartialQkm}.

\begin{table}[h!] \label{TableRankTwo}
\centering
\caption{Irreducible Satake diagrams of rank two for simple $\lie{g}$} \label{Table:RankTwo}

\newcolumntype{C}{ >{\centering\arraybackslash} m{1.5cm} }
\newcolumntype{D}{ >{\centering\arraybackslash} m{6.4cm} }
\newcolumntype{E}{ >{\centering\arraybackslash} m{4.3cm} }

\resizebox{\columnwidth}{!}{
\begin{tabular}[t]{| C |D || C | E |}
	\hline
	
	$AI_2$ & 
	\rb{
		\begin{tikzpicture} 
			[scale=0.7, white/.style={circle,draw=black,inner sep = 0mm, minimum size = 3mm},
			black/.style={circle,draw=black,fill=black, inner sep= 0mm, minimum size = 3mm}, every node/.style={transform shape}]
		
			\node[white] (first) [label = below:{\scriptsize $1$}] {};		
			\node[white] (third) [right= 0.8cm of first] [label = below:{\scriptsize $2$}] {}
				edge (first);
		\end{tikzpicture}	
	} & 
	
	$CII_4$ & 
	\rb{
		\begin{tikzpicture}
			[scale=0.7, white/.style={circle,draw=black,inner sep = 0mm, minimum size = 3mm},
			black/.style={circle,draw=black,fill=black, inner sep= 0mm, minimum size = 3mm}, every node/.style={transform shape}]
		
			\node[black] (first)  [label = below:{\scriptsize $1$}] {};		
			\node[white] (second) [right= of first] [label = below:{\scriptsize $2$}] {}			
				edge (first);
			\node[black] (third) [right = of second] [label = below:{\scriptsize $3$}] {}
				edge (second);
			\node[white] (fourth) [right = of third] [label = below:{\scriptsize $4$}] {}
				edge [double equal sign distance, ->-] (third);
		\end{tikzpicture}	
	} \\ \hline

	$AII_5$ & 
	\rb{
		\begin{tikzpicture}  
			[scale=0.7, white/.style={circle,draw=black,inner sep = 0mm, minimum size = 3mm},
			black/.style={circle,draw=black,fill=black, inner sep= 0mm, minimum size = 3mm}, every node/.style={transform shape}]
	
			\node[black] (1)	 [label = below:{\scriptsize $1$}]	{};
			\node[white] (first) [right = of 1] [label = below:{\scriptsize $2$}]  {}
				edge (1);
			\node[black] (second) [right=of first] [label = below:{\scriptsize $3$}] {}
				edge (first);
			\node[white] (last) [right= of second] [label = below:{\scriptsize $4$}]  {}
				edge (second);		
			\node[black] (fourth) [right=of last] [label = below:{\scriptsize $5$}] {}
				edge (last);
		\end{tikzpicture}	
	} & 
	
	$DI_n$, $n \geq 5$ & 
	\rb{
		\begin{tikzpicture}
			[scale=0.7, white/.style={circle,draw=black,inner sep = 0mm, minimum size = 3mm},
			black/.style={circle,draw=black,fill=black, inner sep= 0mm, minimum size = 3mm}, every node/.style={transform shape}]
		
			\node[white] (first) [label = below:{\scriptsize $1$}] {};		
			\node[white] (second) [right= of first] [label = below:{\scriptsize $2$}]  {}
				edge  (first);
			\node[black] (third) [right = of second] {}
				edge (second);
			\node[black] (fourth) [right = 1.5cm of third] {}
				edge [dashed] (third);
			\node[black] (fifth) [above right = 0.5cm of fourth] [label = above:{\scriptsize $n-1$}] {}
				edge (fourth);
			\node[black] (sixth) [below right = 0.5cm of fourth] [label = below:{\scriptsize $n$}] {}
				edge (fourth);	
		\end{tikzpicture}	
	} \\ \hline
	
	$AIII_3$ & 
	\rb{
		\begin{tikzpicture} 
			[scale=0.7, white/.style={circle,draw=black,inner sep = 0mm, minimum size = 3mm},
			black/.style={circle,draw=black,fill=black, inner sep= 0mm, minimum size = 3mm}, every node/.style={transform shape}]
		
			\node[white] (first) [label = below:{\scriptsize $1$}] {};		
			\node[white] (second) [right=of first] [label = below:{\scriptsize $2$}] {}
				edge (first);
			\node[white] (third) [right=of second] [label = below:{\scriptsize $3$}] {}
				edge (second)
				edge	 [latex'-latex' , shorten <=3pt, shorten >=3pt, bend right=30, densely dotted] node[auto,swap] {} (first); 
		\end{tikzpicture}	
	} & 
	
	$DIII_4$ & 
	\rb{
		\begin{tikzpicture}
			[scale=0.7, white/.style={circle,draw=black,inner sep = 0mm, minimum size = 3mm},
			black/.style={circle,draw=black,fill=black, inner sep= 0mm, minimum size = 3mm}, every node/.style={transform shape}]
		
			\node[black] (first) [label = below:{\scriptsize $1$}] {};		
			\node[white] (second) [right= of first] [label = below:{\scriptsize $2$}]  {}
				edge  (first);
			\node[black] (third) [above right = 0.5cm of second] [label = above:{\scriptsize $3$}] {}
				edge (second);
			\node[white] (fourth) [below right = 0.5cm of second] [label = below:{\scriptsize $4$}] {}
				edge (second);
		\end{tikzpicture}	
	} \\ \hline

	$AIII_{n}$, $n \geq 4$ & 
	\rb{
		\begin{tikzpicture} 
			[scale=0.7, white/.style={circle,draw=black,inner sep = 0mm, minimum size = 3mm},
			black/.style={circle,draw=black,fill=black, inner sep= 0mm, minimum size = 3mm}, every node/.style={transform shape}]
	
			\node[white] (1)	 [label = below:{\scriptsize $1$}]	{};
			\node[white] (first) [right = of 1] [label = below:{\scriptsize $2$}]  {}
				edge (1);
			\node[black] (second) [right=of first] {}
				edge (first);
			\node[black] (last) [right=1.5cm of second] {}
				edge [dashed] (second);		
			\node[white] (fourth) [right=of last]  {}
				edge (last);
			\node[white] (2) [right = of fourth] [label = below:{\scriptsize $\phantom{1} n \phantom{1}$}]{}
				edge(fourth) 
				edge	 [latex'-latex' , shorten <=3pt, shorten >=3pt, bend right=30, densely dotted] node[auto,swap] {} (1);
		\end{tikzpicture}	
	} & 
	
	$DIII_5$ & 
	\rb{
		\begin{tikzpicture}
			[scale=0.7, white/.style={circle,draw=black,inner sep = 0mm, minimum size = 3mm},
			black/.style={circle,draw=black,fill=black, inner sep= 0mm, minimum size = 3mm}, every node/.style={transform shape}]
		
			\node[black] (first) [label = below:{\scriptsize $1$}] {};		
			\node[white] (second) [right= of first] [label = below:{\scriptsize $2$}] {}
				edge (first);
			\node[black] (third) [right = of second][label = below:{\scriptsize $3$}] {}
				edge (second);
			\node[white] (fourth) [above right = 0.5cm of third] [label = above:{\scriptsize $4$}] {}
				edge (third);
			\node[white] (fifth) [below right = 0.5cm of third] [label = below:{\scriptsize $5$}] {}
				edge (third)
				edge	 [latex'-latex' , shorten <=3pt, shorten >=3pt, bend right=60, densely dotted] node[auto,swap] {} (fourth);	
		\end{tikzpicture}	
	} \\ \hline
	
	$BI_n$, $n \geq 3$ & 
	\rb{
		\begin{tikzpicture}
			[scale=0.7, white/.style={circle,draw=black,inner sep = 0mm, minimum size = 3mm},
			black/.style={circle,draw=black,fill=black, inner sep= 0mm, minimum size = 3mm}, every node/.style={transform shape}]
		
			\node[white] (first) [label = below:{\scriptsize $1$}] {};
			\node[white] (second) [right = of first] [label = below:{\scriptsize $2$}] {}
				edge  (first);		
			\node[black] (third) [right=of second] {}			
				edge (second);
			\node[black] (fourth) [right = 1.5cm of third] {}
				edge [dashed] (third);
			\node[black] (fifth) [right = of fourth] [label = below:{\scriptsize $\phantom{1} n \phantom{1}$}] {}
				edge [double equal sign distance, -<-] (fourth);
		\end{tikzpicture}	
	} & 
	
	$EIII$ & 
	\rb{
		\begin{tikzpicture}
			[scale=0.7, white/.style={circle,draw=black,inner sep = 0mm, minimum size = 3mm},
			black/.style={circle,draw=black,fill=black, inner sep= 0mm, minimum size = 3mm}, every node/.style={transform shape}]
		
			\node[white] (first) [label = below:{\scriptsize $1$}] {};		
			\node[black] (second) [right= of first] [label = below:{\scriptsize $3$}] {}
				edge (first);
			\node[black] (third) [right = of second] [label = below:{\scriptsize $4$}] {}
				edge (second);
			\node[black] (fourth) [right = of third][label = below:{\scriptsize $5$}]  {}
				edge (third);
			\node[white] (fifth) [above = of third] [label = above:{\scriptsize $2$}] {}
				edge (third);
			\node[white] (sixth) [right = of fourth] [label = below:{\scriptsize $6$}]  {}
				edge (fourth)
				edge	 [latex'-latex' , shorten <=3pt, shorten >=3pt, bend right=330, densely dotted] node[auto,swap] {} (first);
		\end{tikzpicture}	
	} \\ \hline

		$CI_2$ & 
	\rb{
		\begin{tikzpicture}
			[scale=0.7, white/.style={circle,draw=black,inner sep = 0mm, minimum size = 3mm},
			black/.style={circle,draw=black,fill=black, inner sep= 0mm, minimum size = 3mm}, every node/.style={transform shape}]
		
			\node[white] (first) [label = below:{\scriptsize $1$}] {};		
			\node[white] (second) [right=of first] [label = below:{\scriptsize $2$}] {}			
				edge [double equal sign distance, -<-] (first);
		\end{tikzpicture}	
	} &

	$EIV$ & 
	\rb{
		\begin{tikzpicture}
			[scale=0.7, white/.style={circle,draw=black,inner sep = 0mm, minimum size = 3mm},
			black/.style={circle,draw=black,fill=black, inner sep= 0mm, minimum size = 3mm}, every node/.style={transform shape}]
		
			\node[white] (first) [label = below:{\scriptsize $1$}]{};		
			\node[black] (second) [right= of first] [label = below:{\scriptsize $3$}] {}
				edge (first);
			\node[black] (third) [right = of second][label = below:{\scriptsize $4$}] {}
				edge (second);
			\node[black] (fourth) [right = of third][label = below:{\scriptsize $5$}] {}
				edge (third);
			\node[black] (fifth) [above = of third][label = above:{\scriptsize $2$}] {}
				edge (third);
			\node[white] (sixth) [right = of fourth] [label = below:{\scriptsize $6$}]{}
				edge (fourth);
		\end{tikzpicture}	
	}\\ 
	\hline
	
	$CII_n$, $n \geq 5$ & 
	\rb{
		\begin{tikzpicture}
			[scale=0.7, white/.style={circle,draw=black,inner sep = 0mm, minimum size = 3mm},
			black/.style={circle,draw=black,fill=black, inner sep= 0mm, minimum size = 3mm}, every node/.style={transform shape}]
		
			\node[black] (first)  [label = below:{\scriptsize $1$}] {};		
			\node[white] (second) [right= of first] [label = below:{\scriptsize $2$}] {}			
				edge (first);
			\node[black] (third) [right = of second] {}
				edge (second);
			\node[white] (fourth) [right = of third] {}
				edge (third);
			\node[black] (fifth) [right = of fourth] {}
				edge (fourth);
			\node[black] (sixth) [right = 1.5cm of fifth]  {}
				edge [dashed] (fifth);
			\node[black] (seventh) [right = of sixth] [label = below:{\scriptsize $n$}] {}
				edge [double equal sign distance, ->-] (sixth);
		\end{tikzpicture}	
	} & 
	
	$G$ & 
	\rb{
		\begin{tikzpicture}
			[scale=0.7, white/.style={circle,draw=black,inner sep = 0mm, minimum size = 3mm},
			black/.style={circle,draw=black,fill=black, inner sep= 0mm, minimum size = 3mm}, every node/.style={transform shape}]
		
			\node[white] (first) [label = below:{\scriptsize $1$}] {};		
			\node[white] (second) [right=of first] [label = below:{\scriptsize $2$}] {}			
				edge [double distance = 2pt, -<-] (first)
				edge (first);
		\end{tikzpicture}	
	} \\ \hline
\end{tabular} 
}
\end{table}


\begingroup
\allowdisplaybreaks
\subsection{Type $AI_2$}

Consider the Satake diagram of type $AI_2$.

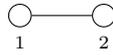
\begin{center}
	\begin{tikzpicture}  
		[white/.style={circle,draw=black,inner sep = 0mm, minimum size = 3mm},
		black/.style={circle,draw=black,fill=black, inner sep= 0mm, minimum size = 3mm}]
		
		\node[white] (first) [label = below:{\scriptsize $1$}] {};		
		\node[white] (third) [right= 0.8cm of first] [label = below:{\scriptsize $2$}] {}
			edge (first);
	\end{tikzpicture}
\end{center}

Since $\Theta = -\text{id}$ the restricted Weyl group $\widet{W}$ coincides with the Weyl group $W$. 
The longest word of the Weyl group has two reduced expressions given by
\begin{align*}
	w_0 &= s_1s_2s_1\\
	w_0^{\prime} &= s_2s_1s_2.
\end{align*} 

\begin{proposition} \label{Prop:QkmRankTwoAI}
	In this case, the partial quasi $K$-matrices $\Qkm_{w_0}$ and $\Qkm_{w_0^{\prime}}$ coincide with the quasi $K$-matrix $\Qkm$. Hence $\Qkm_{w_0} = \Qkm_{w_0^{\prime}}$.
\end{proposition}

Before we prove this, we need to know how the Lusztig skew derivations ${}_1r$ and ${}_2r$ act on certain elements and their powers, and also some commutation relations. These are given in the following two lemmas, whose proofs are obtained by straightforward computation.

\begin{lemma} \label{lem:RankTwoAIComm}
For any $n \in \mathbb{N}$, the relations
\begin{align*}
	E_2^n E_1 &= q^n E_1 E_2^n - q\{ n \} E_2^{n-1} \T{1}(E_2),\\
	E_1^n E_2 &= q^n E_2 E_1^n - q\{ n \} E_1^{n-1} \T{2}(E_1),\\
	\T{1}(E_2)^n E_1 &= q^{-n} E_1 \T{1}(E_2)^n,\\
	\T{2}(E_1)^n E_2 &= q^{-n} E_2 \T{2}(E_1)^n
\end{align*}
hold in $U_q(\lie{sl}_3)$.
\end{lemma}

\begin{lemma} \label{lem:RankTwoAIir}
For any $n \in \mathbb{N}$, the relations
\begin{align*}
	&\ir{ 1 }{ E_2^n } = \: \ir{ 1 }{ \T{2}(E_1)^n } = \: \ir{ 2 }{ E_1^n } = \: \ir{2}{ \T{1}(E_2)^n } = 0,\\
	&\ir{ 1 }{ E_1^n } = \{ n \} E_1^{n-1},\\
	&\ir{ 2 }{ E_2^n } = \{ n \} E_2^{n-1},\\
	&\ir{ 1 }{ \T{1}(E_2)^n } = (1 - q^{-2}) \{ n \} E_2 \T{1}(E_2)^{n-1},\\
	&\ir{ 2 }{ \T{2}(E_1)^n } = (1 - q^{-2}) \{ n \} E_1 \T{2}(E_1)^{n-1}
\end{align*}
hold in $U_q(\lie{sl}_3)$.
\end{lemma}

\begin{proof}[Proof of Proposition \ref{Prop:QkmRankTwoAI}]
Consider first the element $\Qkm_{w_0}$. Using \eqref{eq:PartialQkm} and Lemma \ref{lem:QkmRankOneAI}, we write
\begin{equation}
	\Qkm_{w_0}	=	\Qkm^{[3]}\Qkm^{[2]}\Qkm^{[1]}
\end{equation}
where 
\begin{align*}
\Qkm^{[3]}	&\overset{\eqref{eq:LongestWordPartialQkmPart}}{=} 	\Qkm_2	=	\sum_{n \geq 0} \dfrac{\q^n}{\{2n\}!!}(q^2c_2)^nE_2^{2n},\\
\Qkm^{[2]}	&\overset{\phantom{\eqref{eq:LongestWordPartialQkmPart}}}{=}	\Psi \circ T_1 \circ \Psi^{-1}(\Qkm_2)	=	\sum_{n \geq 0} \dfrac{\q^n}{\{2n\}!!}(q^2c_1)^n(q^2c_2)^nT_1(E_2)^{2n},\\
\Qkm^{[1]}	&\overset{\phantom{\eqref{eq:LongestWordPartialQkmPart}}}{=}	\Qkm_1	=	\sum_{n \geq 0}\dfrac{\q^n}{\{2n\}!!}(q^2c_1)^nE_1^{2n}.
\end{align*}
For $i = 1,2,3$, let $\Qkm_{[i]} = K_1\Qkm^{[i]}K_1^{-1}$. 
The difference between $\Qkm_{[i]}$ and $\Qkm^{[i]}$ is the occurrence of a $q$-power in each summand of the infinite series.
By Equation \eqref{eq:irSyst2}, to show that $\Qkm_{w_0}$ coincides with the quasi $K$-matrix $\Qkm$ we show that
\begin{align*}
	\ir{1}{\Qkm_{w_0}}	&=	\q (q^2c_1)E_1\Qkm_{w_0},\\
	\ir{2}{\Qkm_{w_0}}	&=	\q (q^2c_2)E_2\Qkm_{w_0}.
\end{align*}
By Lemma \ref{lem:RankTwoAIir} and \ref{lem:QkmRankOneAI}, we see that
\begin{align*}
	\ir{2}{\Qkm_{w_0}}	&=	\ir{2}{\Qkm^{[3]}}\Qkm^{[2]}\Qkm^{[1]}\\
					&=	\q (q^2c_2)E_2\Qkm^{[3]}\Qkm^{[2]}\Qkm^{[1]}\\
					&=	\q (q^2c_2)E_2\Qkm_{w_0}.
\end{align*}
By the property \eqref{eq:ir} of the skew derivative ${}_1r$, we have
\begin{equation} \label{proofeq:irQkmwAI}
\ir{1}{\Qkm_{w_0}}	=	\Qkm_{[3]}\ir{1}{\Qkm^{[2]}}\Qkm^{[1]} + \Qkm_{[3]}\Qkm_{[2]}\ir{1}{\Qkm^{[1]}} 
\end{equation} 
Using Lemma \ref{lem:RankTwoAIir}, we have
\begin{align*}
\ir{1}{\Qkm^{[2]}}
	&=	\sum_{n \geq } \dfrac{\q^n}{\{2n\}!!} (q^2c_1)^n(q^2c_2)^n \ir{1}{T_1(E_2)^{2n}}\\
	&=	q^{-1}\q E_2 \sum_{n \geq 1} \dfrac{\q^n}{\{2n-2\}!!} (q^2c_1)^n(q^2c_2)^n T_1(E_2)^{2n-1}\\
	&=	\q^2 q^{-1} (q^2c_1)(q^2c_2) E_2T_1(E_2) \Qkm^{[2]},\\
\ir{1}{\Qkm^{[1]}}
	&=	\q (q^2c_1) E_1 \Qkm^{[1]}.
\end{align*}
The second summand of Equation \eqref{proofeq:irQkmwAI} is of the form
\begin{equation*}
\q (q^2c_1) \Qkm_{[3]}\Qkm_{[2]}E_1\Qkm^{[1]}.
\end{equation*}
We use Lemma \ref{lem:RankTwoAIComm} to bring the $E_1$ in the above summand to the front.
We have
\begin{align*}
\Qkm_{[2]}E_1
	&=	\sum_{n \geq 0} \dfrac{\q^n}{\{2n\}!!} (q^2c_1)^n(q^2c_2)^n q^{2n} T_1(E_2)^{2n}E_1\\
	&=	E_1 \sum_{n \geq 0} \dfrac{\q^n}{\{2n\}!!} (q^2c_1)^n(q^2c_2)^n T_1(E_2)^{2n}\\
	&=	E_1\Qkm^{[2]},\\
\Qkm_{[3]}E_1
	&=	\sum_{n \geq 0} \dfrac{\q^n}{\{2n\}!!} (q^2c_2)^n q^{-2n} E_2^{2n}E_1\\
	&=	\sum_{n \geq 0} \dfrac{\q^n}{\{2n\}!!} (q^2c_2)^n q^{-2n} (q^{2n}E_1E_2^{2n} - q\{2n\}E_2^{2n-1}T_1(E_2))\\
	&=	E_1\Qkm^{[3]} - q\sum_{n \geq 1} \dfrac{\q^n}{\{2n-2\}!!} (q^2c_2)^nq^{-2n}E_2^{2n-1}T_1(E_2)\\
	&=	E_1\Qkm^{[3]} - \q q^{-1}(q^2c_2)\Qkm_{[3]}E_2T_1(E_2).
\end{align*}
Hence,
\begin{align*}
\q (q^2c_1) \Qkm_{[3]}\Qkm_{[2]}E_1\Qkm^{[1]}
	&=	\q (q^2c_1) \Qkm_{[3]}E_1\Qkm^{[2]}\Qkm^{[1]}\\
	&=	\q (q^2c_1)E_1\Qkm^{[3]}\Qkm^{[2]}\Qkm^{[1]}\\ 
		&\quad{}- \q^2 q^{-1}(q^2c_1)(q^2c_2)\Qkm_{[3]}E_2T_1(E_2) \Qkm^{[2]}\Qkm^{[1]}\\
	&=	\q (q^2c_1) E_1\Qkm_{w_0} - \Qkm_{[3]}\ir{1}{\Qkm^{[2]}}\Qkm^{[1]}.
\end{align*}
It follows from \eqref{proofeq:irQkmwAI} that $\ir{1}{\Qkm_{w_0}} = \q (q^2c_1) E_1\Qkm_{w_0}$, so $\Qkm_{w_0}$ coincides with the quasi $K$-matrix $\Qkm$.
Instead of repeating the same calculation for $\Qkm_{w_0^{\prime}}$, we use the underlying symmetry in type $AI_2$, which implies that $\Qkm_{w_0^{\prime}}$ also coincides with the quasi $K$-matrix $\Qkm$.
\end{proof}


\subsection{Type $AII_5$}

Consider the Satake diagram of type $AII_{5}$.

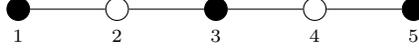
\begin{center}
	\begin{tikzpicture} 
		[white/.style={circle,draw=black,inner sep = 0mm, minimum size = 3mm},
		black/.style={circle,draw=black,fill=black, inner sep= 0mm, minimum size = 3mm}]
	
		\node[black] (1)	 [label = below:{\scriptsize $1$}]	{};
		\node[white] (first) [right = of 1] [label = below:{\scriptsize $2$}]  {}
			edge (1);
		\node[black] (second) [right=of first] [label = below:{\scriptsize $3$}] {}
			edge (first);
		\node[white] (last) [right= of second] [label = below:{\scriptsize $4$}]  {}
			edge (second);		
		\node[black] (fourth) [right=of last] [label = below:{\scriptsize $5$}] {}
			edge (last);
	\end{tikzpicture}
\end{center}
In this case the involutive automorphism $\Theta: \lie{h}^{\ast} \rightarrow \lie{h}^{\ast}$ is given by
\begin{equation*}
	\Theta = -s_1s_3s_5.
\end{equation*}
There are two $\tau$-orbits of white nodes given by the sets $\{2\}$ and $\{4\}$.
The restricted root system is of type $AI_2$ since the restricted roots
\begin{equation*}
	\widet{\alpha_2} = \dfrac{\alpha_1 + 2\alpha_2 + \alpha_3}{2}, \qquad{}	\widet{\alpha_4} = \dfrac{\alpha_3 + 2\alpha_4 + \alpha_5}{2}
\end{equation*}
have the same length.
The subgroup $\widet{W} \subset W^{\Theta}$ is generated by the elements
\begin{equation*}
	\widet{s_2}	=	 s_2s_1s_3s_2, \qquad{}
	\widet{s_4}	=	 s_4s_3s_5s_4.
\end{equation*}
The longest word of the restricted Weyl group has two reduced expressions given by
\begin{equation*}
	\widet{w_0}	=	\widet{s_2}\widet{s_4}\widet{s_2}, \qquad{}
	\widet{w_0}^{\prime}	=	\widet{s_4}\widet{s_2}\widet{s_4}.
\end{equation*}
By Lemma \ref{lem:QkmRankOneAII} we have
\begin{align*}
\Qkm_2	&=	\sum_{n \geq 0} \dfrac{(qc_2)^n}{\{n\}!} [E_2, T_{13}(E_2)]_{q^{-2}}^n,\\
\Qkm_4	&=	\sum_{n \geq 0} \dfrac{(qc_4)^n}{\{n\}!} [E_4, T_{35}(E_4)]_{q^{-2}}^n.
\end{align*}

\begin{proposition} \label{Prop:QkmRankTwoAII}
The partial quasi $K$-matrices $\Qkm_{\widet{w_0}}$ and $\Qkm_{\widet{w_0}^{\prime}}$ coincide with the quasi $K$-matrix $\Qkm$.
\end{proposition}

We have the following relations needed for the proof of Proposition \ref{Prop:QkmRankTwoAII}. These are proved by induction.

\begin{lemma} \label{lemma: AIIcomm}
For any $n \in \mathbb{N}$ the relations
\begin{samepage}
\begin{align}
 &[E_4, T_{35}(E_4)]_{q^{-2}}^n T_{13}(E_2) = q^nT_{13}(E_2)[E_4, T_{35}(E_4)]_{q^{-2}}^n \label{Appeq: AIIcomm1}\\ &\phantom{==} - q\{n\}[E_4, T_{35}(E_4)]^{n-1}_{q^{-2}}[T_3(E_4), T_{1235}(E_4)]_{q^{-2}}, \nonumber \\ 
 &[T_{23}(E_4), T_{1235}(E_4)]_{q^{-2}}^n T_{13}(E_2) = q^{-n}T_{13}(E_2)[T_{23}(E_4), T_{1235}(E_4)]_{q^{-2}}^n \label{Appeq: AIIcomm2}
\end{align}
\end{samepage}
hold in $U_q(\lie{sl}_6)$.
\end{lemma}

\begin{lemma} \label{lemma: AIIir}
For any $n \in \mathbb{N}$ the relation
\begin{equation} \label{Appeq: AIIir}
\begin{split}
&\ir{2}{[T_{23}(E_4), T_{1235}(E_4)]_{q^{-2}}^n}\\ &\qquad{}= q^{-1}(q-q^{-1})\{n\} [T_3(E_4), T_{1235}(E_4)]_{q^{-2}} [T_{23}(E_4), T_{1235}(E_4)]_{q^{-2}}^{n-1}
\end{split}
\end{equation}
holds in $U_q(\lie{sl}_6)$.
\end{lemma}

\begin{proof}[Proof of Proposition \ref{Prop:QkmRankTwoAII}]
We only confirm that $\Qkm_{\widet{w_0}}$ coincides with the quasi $K$-matrix $\Qkm$. By the underlying symmetry in type $AII_5$, the calculation for $\Qkm_{\widet{w_0}^{\prime}}$ is the same up to a change of indices. By definition, we have
\begin{equation} 
\Qkm_{\widet{w_0}} = \Qkm^{[3]} \Qkm^{[2]} \Qkm^{[1]} \nonumber
\end{equation}
where
\begin{align*}
\Qkm^{[3]}	&\overset{\eqref{eq:LongestWordPartialQkmPart}}{=} \Qkm_4, \\
\Qkm^{[2]}	&=	\Psi \circ T_{2132} \circ \Psi^{-1}(\Qkm_4) = \sum_{n \geq 0} \dfrac{(q^2c_2c_4)^n}{\{n\}!} [T_{23}(E_4), T_{1235}(E_4)]_{q^{-2}}^n,\\
\Qkm^{[1]}	&= \Qkm_2.
\end{align*}
Using Lemma \ref{lem:QkmRankOneAII} and \cite[8.26, (4)]{b-Jantzen96} we see that
\begin{align*}
\ir{4}{\Qkm_{\widet{w_0}}} &= \ir{4}{\Qkm_4} \Qkm^{[2]}\Qkm^{[1]}\\
							&= (q-q^{-1})c_4T_{35}(E_4)\Qkm_{4}\Qkm^{[2]}\Qkm^{[1]}\\
							&= (q-q^{-1})c_4T_{35}(E_4)\Qkm_{\widet{w_0}}.
\end{align*}
We want to show that
\begin{equation}
\ir{2}{\Qkm_{\widet{w_0}}} = (q-q^{-1})c_2T_{13}(E_2)\Qkm_{\widet{w_0}}. \nonumber
\end{equation}
For $i = 1, 2, 3$, let $\Qkm_{[i]} = K_2\Qkm^{[i]}K_2^{-1}$. By property \eqref{eq:ir} of the skew derivative ${}_2r$ and \cite[8.26, (4)]{b-Jantzen96} we have
\begin{equation} \label{Appeq: RankTwoAIIproof1}
\ir{2}{\Qkm_{\widet{w_0}}} = \Qkm_{[3]} \ir{2}{\Qkm^{[2]}}\Qkm^{[1]} + \Qkm_{[3]}\Qkm_{[2]} \ir{2}{\Qkm^{[1]}}.
\end{equation}
By Lemma \ref{lem:QkmRankOneAII} we have
\begin{equation}
\ir{2}{\Qkm^{[1]}} = (q-q^{-1})c_2T_{13}(E_2)\Qkm^{[1]}. \nonumber
\end{equation}
By Lemma \ref{lemma: AIIir} we have
\begin{align*}
\ir{2}{\Qkm^{[2]}}	&\overset{\phantom{\eqref{Appeq: AIIir}}}{=}	\sum_{n \geq 1} \dfrac{(q^2c_2c_4)^n}{\{n\}!} \ir{2}{[T_{23}(E_4), T_{1235}(E_4)]_{q^{-2}}^n} \\
					&\overset{\eqref{Appeq: AIIir}}{=}	q(q-q^{-1})c_2c_4 [T_3(E_4), T_{1235}(E_4)]_{q^{-2}} \Qkm^{[2]}.
\end{align*}
The second summand of Equation \eqref{Appeq: RankTwoAIIproof1} is of the form $(q-q^{-1})c_2\Qkm_{[3]}\Qkm_{[2]}T_{13}(E_2)\Qkm^{[1]}$. Using Lemma \ref{lemma: AIIcomm}, we bring the $T_{13}(E_2)$ term in this expression to the front. We have
\begin{align}
\Qkm_{[2]}T_{13}(E_2)	&\overset{\eqref{Appeq: AIIcomm2}}{=}	T_{13}(E_2)\Qkm^{[2]}, \label{Appeq: RankTwoAIIproof2} \\
\Qkm_{[3]}T_{13}(E_2)	&\overset{\eqref{Appeq: AIIcomm1}}{=}	T_{13}(E_2)\Qkm^{[3]} - qc_4\Qkm_{[3]}[T_3(E_4), T_{1235}(E_4)]_{q^{-2}}. \label{Appeq: RankTwoAIIproof3}
\end{align}
Substituting \eqref{Appeq: RankTwoAIIproof2} and \eqref{Appeq: RankTwoAIIproof3} into $(q-q^{-1})c_2\Qkm_{[3]}\Qkm_{[2]}T_{13}(E_2)\Qkm^{[1]}$, we obtain
\begin{align*}
(q-q^{-1})c_2\Qkm_{[3]}\Qkm_{[2]}T_{13}(E_2)\Qkm^{[1]}	&\overset{\eqref{Appeq: RankTwoAIIproof2}}{=} (q-q^{-1})c_2\Qkm_{[3]}T_{13}(E_2)\Qkm^{[2]}\Qkm^{[1]}\\
														&\overset{\eqref{Appeq: RankTwoAIIproof3}}{=} (q-q^{-1})c_2T_{13}(E_2)\Qkm^{[3]}\Qkm^{[2]}\Qkm^{[1]}\\ &\;\;\;{}- \q qc_2c_4\Qkm^{[3]}[T_3(E_4), T_{1235}(E_4)]_{q^{-2}} \Qkm^{[2]}\Qkm^{[1]}\\
														&\overset{\phantom{\eqref{Appeq: RankTwoAIIproof3}}}{=}	(q-q^{-1})c_2T_{13}(E_2)\Qkm_{\widet{w_0}} - \Qkm_{[3]}\ir{2}{\Qkm^{[2]}}\Qkm^{[1]}.
\end{align*}
It follows from \eqref{Appeq: RankTwoAIIproof1} that $\ir{2}{\Qkm_{\widet{w_0}}} = (q-q^{-1}c_2T_{13}(E_2)\Qkm_{\widet{w_0}}$ as required.
\end{proof}

\subsection{Type $AIII_3$}\label{app:AIII3}

We consider the diagram of type $A3$ with non-trivial diagram automorphism $\tau$ and no black dots. 

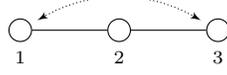
\begin{center}
	\begin{tikzpicture}  
		[white/.style={circle,draw=black,inner sep = 0mm, minimum size = 3mm},
		black/.style={circle,draw=black,fill=black, inner sep= 0mm, minimum size = 3mm}]
		
		\node[white] (first) [label = below:{\scriptsize $1$}] {};		
		\node[white] (second) [right=of first] [label = below:{\scriptsize $2$}] {}
			edge (first);
		\node[white] (third) [right=of second] [label = below:{\scriptsize $3$}] {}
			edge (second)
			edge	 [latex'-latex' , shorten <=3pt, shorten >=3pt, bend right=30, densely dotted] node[auto,swap] {} (first); 
	\end{tikzpicture}
\end{center}
Here, we see that there are $2$ nodes in the restricted Dynkin diagram, corresponding to the restricted roots
\begin{equation*}
\widet{\alpha_1} = \dfrac{\alpha_1 +\alpha_3}{2}, \qquad{}	\widet{\alpha_2} = \alpha_2.
\end{equation*}
A quick check confirms that $1/2(\alpha_1 + \alpha_3)$ is the short root, and hence the restricted root system is of type $B_2$. 
The subgroup $\widet{W}$ is generated by the elements
\begin{equation*}
\widet{s_1} = s_1s_3, \qquad{}
\widet{s_2}	= s_2.
\end{equation*}
The longest word of the restricted Weyl group has two reduced expressions given by
\begin{equation*}
\widet{w_0} = \widet{s_1}\widet{s_2}\widet{s_1}\widet{s_2}, \qquad{}
\widet{w_0^{\prime}} = \widet{s_2}\widet{s_1}\widet{s_2}\widet{s_1}. 
\end{equation*}
The definition \eqref{eq:ParameterSetC} and condition \eqref{eq:ciCond} imply that $c_1=c_3=\overline{c_1}$.
By Lemmas \ref{lem:QkmRankOneAI} and \ref{lem:QkmRankOneA1xA1} we have
\begin{align}
		\Qkm_1 &= \sum_{n \geq 0} \dfrac{(q-q^{-1})^n}{ \{n \}! } c_1^n(E_1E_3)^n,\\
		\Qkm_2 &= \sum_{n \geq 0} \dfrac{(q-q^{-1})^n}{ \{2n \}!! } (q^2c_2)^n E_2^{2n}.
\end{align}

\begin{proposition} \label{AppProp: A3 Qkmw0}
The partial quasi $K$-matrix $\Qkm_{\widet{w_0}}$ coincides with the quasi $K$-matrix $\Qkm$.
\end{proposition}

The following relations are needed for the proof of Proposition \ref{AppProp: A3 Qkmw0}. They are checked by induction.

\begin{lemma}\label{A3commE3}
For any $n \in \mathbb{N}$, the relations
\begin{align}
T_{13}(E_2)^{n} E_3 &= q^{-n}E_3T_{13}(E_2)^n, \label{eq:A3commE3 T13E2}\\
\big( T_1(E_2)T_3(E_2) \big)^nE_3 &= E_3 \big( T_1(E_2)T_3(E_2) \big)^n \label{eq:A3commE3 T1E2T3E2}\\ &\qquad{}- q\{ n \} \big( T_1(E_2)T_3(E_2) \big)^{n-1} T_3(E_2)T_{13}(E_2), \nonumber \\
E_2^nE_3 &= q^nE_3E_2^n - q\{ n \} E_2^{n-1}T_3(E_2) \label{eq:A3commE3 E2} 
\end{align}
hold in $U_q(\mathfrak{sl}_4)$.
\end{lemma}

\begin{lemma}\label{A3r1}
For any $n \in \mathbb{N}$, the relations
\begin{align}
\ir{1}{T_{13}(E_2)^n } &= q^{-1}(q-q^{-1}) \{n \} T_3(E_2)T_{13}(E_2)^{n-1}, \label{eq:A3r1 T13E2}\\
\ir{1}{ \big( T_1(E_2)T_3(E_2) \big)^n} &= q^{-1}(q-q^{-1}) \{ n \} E_2T_3(E_2) \big( T_1(E_2)T_3(E_2) \big)^{n-1} \label{eq:A3r1 T1E2T3E2}
\end{align}
hold in $U_q(\mathfrak{sl}_4)$.
\end{lemma}

\begin{proof}[Proof of Proposition \ref{AppProp: A3 Qkmw0}]

Take $\widet{w_0} = \os_1 \os_2 \os_1 \os_2$. Then we have
\begin{equation*}
\Qkm_{\widet{w_0}} = \Qkm^{[4]}\Qkm^{[3]}\Qkm^{[2]}\Qkm^{[1]}
\end{equation*}
where
\begin{align*}
\Qkm^{[4]} &\overset{\eqref{eq:LongestWordPartialQkmPart}}{=} \Qkm_{\tau_0(2)} = \Qkm_2, \\
\Qkm^{[3]} &\overset{\phantom{\eqref{eq:LongestWordPartialQkmPart}}}{=} \Psi \circ \widet{T_1}\widet{T_2} \circ \Psi^{-1}(\Qkm_1) = \sum_{n \geq 0} \dfrac{(q-q^{-1})^n}{ \{n \}! } (q^2c_1c_2)^n (T_3(E_2)T_1(E_2))^n,\\
\Qkm^{[2]} &\overset{\phantom{\eqref{eq:LongestWordPartialQkmPart}}}{=} \Psi \circ \widet{T_1} \circ \Psi^{-1}(\Qkm_2) = \sum_{n \geq 0} \dfrac{(q-q^{-1})^n}{ \{2n \}!!} (q^4c_1^2c_2)^n T_{13}(E_2)^{2n},\\
\Qkm^{[1]} &\overset{\phantom{\eqref{eq:LongestWordPartialQkmPart}}}{=} \Qkm_1.  
\end{align*}
By Lemma \ref{lem:QkmRankOneAI}, property \eqref{eq:ir} of the skew derivative ${}_2r$ and \cite[8.26, (4)]{b-Jantzen96}, we see that $\ir{2}{\Qkm_{\widet{w_0}}} = (q-q^{-1})q^2c_2E_2\Qkm_{\widet{w_0}}$. Due to the underlying symmetry in this case, we only need to show that
\begin{equation*}
\ir{1}{\Qkm_{\widet{w_0}}} = (q-q^{-1})c_1E_3\Qkm_{\widet{w_0}}.
\end{equation*}
For each $i = 1, 2, 3, 4$, let $\Qkm_{[i]} = K_1\Qkm^{[i]}K_1^{-1}$. Then by the property \eqref{eq:ir} of the skew derivation ${}_1r$, we have
\begin{equation*}
\ir{1}{\Qkm_w} = \Qkm_{[4]} \ir{1}{\Qkm^{[3]}} \Qkm^{[2]}\Qkm^{[1]} + \Qkm_{[4]}\Qkm_{[3]} \ir{1}{\Qkm^{[2]}} \Qkm^{[1]} + \Qkm_{[4]}\Qkm_{[3]}\Qkm_{[2]} \ir{1}{\Qkm^{[1]}}.
\end{equation*}
Using Lemmas \ref{lem:QkmRankOneA1xA1} and \ref{A3r1}, it follows that
\begin{align}
\ir{1}{\Qkm^{[3]}}	&\overset{\eqref{eq:A3r1 T1E2T3E2}}{=}	q^{-1}(q-q^{-1})^2(q^2c_1c_2)E_2T_3(E_2)\Qkm^{[3]}, \label{eq:A3 ir1 Qkm3}\\
\ir{1}{\Qkm^{[2]}}	&\overset{\eqref{eq:A3r1 T13E2}}{=}	q^{-1}(q-q^{-1})^2(q^4c_1^2c_2)T_3(E_2)T_{13}(E_2)\Qkm^{[2]}, \label{eq:A3 ir1 Qkm2}\\
\ir{1}{\Qkm^{[1]}}	&\overset{{\eqref{eq:QkmRankOneA1xA1}}}{=}	(q-q^{-1})E_3\Qkm^{[1]}. \label{eq:A3 ir1 Qkm1}	
\end{align}
Using Lemma \ref{A3commE3}, we look at the term $\Qkm_{[4]}\Qkm_{[3]}\Qkm_{[2]} \ir{1}{\Qkm^{[1]}}$ in more detail. We have
\begin{align}
\Qkm_{[2]}E_3	&\overset{\eqref{eq:A3commE3 T13E2}}{=}	E_3\Qkm^{[2]}, \label{eq:A3 Qkm2 E3} \\
\Qkm_{[3]}E_3	&\overset{\eqref{eq:A3commE3 T1E2T3E2}}{=}	E_3\Qkm^{[3]} - q(q-q^{-1})(q^2c_1c_2)\Qkm_{[3]}T_3(E_2)T_{13}(E_2), \label{eq:A3 Qkm3 E3}\\
\Qkm_{[4]}E_3	&\overset{\eqref{eq:A3commE3 E2}}{=}	E_3\Qkm^{[4]} - q^{-1}(q-q^{-1})(q^2c_2)\Qkm_{[4]}E_2T_3(E_2). \label{eq:A3 Qkm4 E3}
\end{align}
It follows that 
\begin{align*}
\Qkm_{[4]}&\Qkm_{[3]}\Qkm_{[2]} \ir{1}{\Qkm^{[1]}}	\\
	&\overset{\eqref{eq:A3 Qkm2 E3}}{=}	(q-q^{-1})c_1\Qkm_{[4]}\Qkm_{[3]}E_3\Qkm^{[2]}\Qkm^{[1]}\\
	&\overset{\eqref{eq:A3 Qkm3 E3}}{=}	(q-q^{-1})c_1\Qkm_{[4]} \big( E_3\Qkm^{[3]} - q(q-q^{-1})(q^2c_1c_2)\Qkm_{[3]}T_3(E_2)T_{13}(E_2) \big)\Qkm^{[2]}\Qkm^{[1]}\\
	&\overset{\eqref{eq:A3 Qkm4 E3}}{=} (q-q^{-1})c_1\big( E_3\Qkm^{[4]} - q^{-1}(q-q^{-1})(q^2c_2)\Qkm_{[4]}E_2T_3(E_2) \big)\Qkm^{[3]}\Qkm^{[2]}\Qkm^{[1]}\\
	&\qquad{} -  \Qkm_{[4]}\Qkm_{[3]} \ir{1}{\Qkm^{[2]}}\Qkm^{[1]} \\
	&\overset{\phantom{\eqref{eq:A3 Qkm2 E3}}}{=} (q-q^{-1})c_1E_3\Qkm_{\widet{w_0}} - \Qkm_{[4]}\ir{1}{\Qkm^{[3]}}\Qkm^{[2]}\Qkm^{[1]} - \Qkm_{[4]}\Qkm_{[3]} \ir{1}{\Qkm^{[2]}}\Qkm^{[1]} 
\end{align*}
by equations \eqref{eq:A3 ir1 Qkm3}, \eqref{eq:A3 ir1 Qkm2} and \eqref{eq:A3 ir1 Qkm1}.
Hence, it follows that $\ir{1}{\Qkm_{\widet{w_0}}} = (q-q^{-1})c_1E_3\Qkm_{\widet{w_0}}$, as required.

\end{proof}

\begin{proposition} \label{AppProp: A3 Qkmw0'}
The partial quasi $K$-matrix $\Qkm_{\widet{w_0}^\prime}$ coincides with the quasi $K$-matrix $\Qkm$.
\end{proposition}

The following relations are needed for the proof of Proposition \ref{AppProp: A3 Qkmw0'}. They are checked by induction. 

\begin{lemma}\label{A3commE2}
For any $n \in \mathbb{N}$, the relations
\begin{align}
 &\big(T_{2}(E_3)T_{2}(E_1) \big)^n E_2 = q^{-2n}E_2 \big(T_{2}(E_3)T_{2}(E_1) \big)^n, \label{eq:A3commE2 T2E1E3}\\
 &T_{213}(E_2)^nE_2 = E_2T_{213}^n(E_2) - (q-q^{-1}) \{n \} T_{213}(E_2)^{n-1}T_{2}(E_3)T_{2}(E_1), \label{eq:A3commE2 T213E2}\\
 &\big(E_1E_3\big)^n E_2 = q^{2n}E_2 \big(E_1E_3 \big)^n - q \{ n \}\big(E_1E_3\big)^{n-1}(E_3T_{2}(E_1) + e_1T_2(E_3) \label{eq:A3commE2 E1E3}\\
&\qquad{}- q^2 \{n \}^2 \big(E_1E_3 \big)^{n-1} T_{213}(E_2) \nonumber  
\end{align}
hold in $U_q(\mathfrak{sl}_4)$.
\end{lemma}

\begin{lemma}\label{A3r2}
For any $n \in \mathbb{N}$, the relations
\begin{align}
\ir{2}{T_{213}(E_2)^n} &= q^{-2}(q-q^{-1})^2\{n\}E_1E_3T_{213}(E_2)^{n-1}, \label{eq:A3r2 T213E2}\\
\ir{2}{T_{2}(E_3)^n T_{2}(E_1)^n } &= q^{-1}(q-q^{-1})\{n\} E_3T_2(E_3)^{n-1}T_2(E_1)^n \label{eq:A3r2 T2E3}\\ 
&\quad{} + q^{-1}(q-q^{-1})\{n\}E_1T_2(E_3)^nT_2(E_1)^{n-1} \nonumber \\ 
&\quad\quad{} + (q-q^{-1}) \{n \}^2 T_{213}(E_2) \big( T_2(E_1)T_2(E_3) \big)^{n-1} \nonumber
\end{align}
hold in $U_q(\mathfrak{sl}_4)$.
\end{lemma}

\begin{proof}[Proof of Proposition \ref{AppProp: A3 Qkmw0'}]

For $\widet{w_0}^{\prime} = \os_{2}\os_{1}\os_{2}\os_{1}$ we have
\begin{equation*}
 \Qkm_{\widet{w_0}^{\prime}} = \Qkm^{[4]}\Qkm^{[3]}\Qkm^{[2]}\Qkm^{[1]},
\end{equation*}
where
\begin{align*}
\Qkm^{[4]} &\overset{\eqref{eq:LongestWordPartialQkmPart}}{=} \Qkm_{\tau_0(1)} = \Qkm_1,\\
\Qkm^{[3]} &\overset{\phantom{\eqref{eq:LongestWordPartialQkmPart}}}{=} \Psi \circ \widet{T_2}\widet{T_1} \circ \Psi^{-1}(\Qkm_2) = \sum_{n \geq 0} \dfrac{(q-q^{-1})^n}{ \{ 2n \}!!}(q^4c_1^2c_2)^n T_{213}(E_2)^{2n},\\
\Qkm^{[2]} &\overset{\phantom{\eqref{eq:LongestWordPartialQkmPart}}}{=}\Psi \circ \widet{T_2} \circ \Psi^{-1}(\Qkm_1) = \sum_{n \geq 0} \dfrac{(q-q^{-1})^n}{ \{n \}!}(q^2c_1c_2)^n T_2(E_1)^nT_2(E_3)^n,\\
\Qkm^{[1]} &\overset{\phantom{\eqref{eq:LongestWordPartialQkmPart}}}{=} \Qkm_2.
\end{align*}
By Lemma \ref{lem:QkmRankOneA1xA1}, property \eqref{eq:ir} of the skew derivative ${}_1r$ and \cite[8.26, (4)]{b-Jantzen96} we have 
$\ir{1}{\Qkm_{\widet{w_0}^{\prime}}}=	(q-q^{-1})c_1E_3\Qkm_{\widet{w_0}^{\prime}}$.
We want to show that
\begin{equation*}
\ir{2}{\Qkm_{\widet{w_0}^{\prime}}} = (q-q^{-1})(q^2c_2)E_2\Qkm_{\widet{w_0}^{\prime}}.
\end{equation*}
For $i = 1, 2, 3, 4$, we let $\Qkm_{[i]} = K_2\Qkm^{[i]}K_2^{-1}$.
Note that we have $\Qkm_{[3]} = \Qkm^{[3]}$. 
By the property \eqref{eq:ir} of the skew derivative ${}_2r$, we have
\begin{align*}
\ir{2}{\Qkm_{\widet{w_0}^{\prime}}} = \Qkm_{[4]} \ir{2}{\Qkm^{[3]}} \Qkm^{[2]}\Qkm^{[1]} + \Qkm_{[4]}\Qkm_{[3]} \ir{2}{\Qkm^{[2]}} \Qkm^{[1]} + \Qkm_{[4]} \Qkm_{[3]}\Qkm_{[2]} \ir{2}{\Qkm^{[1]}}.
\end{align*}
Using Lemma \ref{A3r2}, we have
\begin{align}
	\ir{2}{\Qkm^{[3]}}	&\overset{\phantom{\eqref{eq:A3r2 T213E2}}}{=}	\sum_{n \geq 0} \dfrac{(q-q^{-1})^n}{ \{ 2n \}!!}(q^4c_1^2c_2)^n \ir{2}{T_{213}(E_2)^{2n}} \nonumber \\
						&\overset{\eqref{eq:A3r2 T213E2}}{=}	q^{-2}(q-q^{-1})^2 \sum_{n \geq 1} \dfrac{(q-q^{-1})^n}{ \{ 2n \}!!}(q^4c_1^2c_2)^n \{ 2n \} E_1E_3T_{213}(E_2)^{2n-1} \nonumber \\
						&\overset{\phantom{\eqref{eq:A3r2 T213E2}}}{=}	q^{-2}(q-q^{-1})^2 E_1E_3 \sum_{n \geq 0} \dfrac{(q-q^{-1})^{n+1}}{ \{2n \}!!}(q^4c_1^2c_2)^{n+1} T_{213}(E_2)^{2n+1} \nonumber \\
						&\overset{\phantom{\eqref{eq:A3r2 T213E2}}}{=}	(q-q^{-1})^3(q^2c_1^2c_2) E_1E_3T_{213}(E_2) \Qkm^{[3]}. \label{eq:A3 ir2 Qkm3} 
\end{align}
Similarly, we have
\begin{align}
	&\ir{2}{\Qkm^{[2]}} \overset{\phantom{\eqref{eq:A3r2 T213E2}}}{=}	\sum_{n \geq 0} \dfrac{(q-q^{-1})^n}{ \{ n \}! } (q^2c_1c_2)^n \ir{2}{ T_2(E_1)^nT_2(E_3)^n} \nonumber\\
						&\overset{\eqref{eq:A3r2 T2E3}}{=}	q^{-1}(q-q^{-1})^2(q^2c_1c_2) \Big( E_3T_2(E_1) + E_1T_2(E_3) \Big) \Qkm^{[2]} \nonumber\\
							&\quad{} + (q-q^{-1})T_{213}(E_2) \sum_{n \geq 1} \dfrac{(q-q^{-1})^n}{ \{ n \}!}(q^2c_1c_2)^n \{n \}^2 T_2(E_1)^{n-1}T_2(E_3)^{n-1} \nonumber\\
						&\overset{\phantom{\eqref{eq:A3r2 T213E2}}}{=}	(q-q^{-1})^2(qc_1c_2) \Big( E_3T_2(E_1) + E_1T_2(E_3) \Big) \Qkm^{[2]} \nonumber \\
							&\quad{} + (q-q^{-1})^2(q^2c_1c_2)T_{213}(E_2) \sum_{n \geq 0} \dfrac{(q-q^{-1})^n}{ \{ n \}! } (q^2c_1c_2)^n \{n+1 \} T_2(E_1)^nT_2(E_3)^n \nonumber.
\end{align} 
We want to write the last summand in terms of $\Qkm^{[2]}$. To do this, we use the fact that $\{n+1\} = 1 + q^2\{n\}$ for $n \geq 1$. This is a useful fact that will be used again in future calculations. Using this, we have
\begin{align*}
&\sum_{n \geq 0} \dfrac{(q-q^{-1})^n}{ \{ n \}! } (q^2c_1c_2)^n \{n+1 \} T_2(E_1)^nT_2(E_3)^n\\
	&\qquad{}=	1 + \sum_{n \geq 1} \dfrac{(q-q^{-1})^n}{ \{ n \}! } (q^2c_1c_2)^n ( 1 + q^2\{n\} ) T_2(E_1)^nT_2(E_3)^n \\
	&\qquad{}=	\Qkm^{[2]} + q^2\sum_{n \geq 1} \dfrac{(q-q^{-1})^n}{ \{ n-1 \}! } (q^2c_1c_2)^n T_2(E_1)^nT_2(E_3)^n \\
	&\qquad{}=	\Qkm^{[2]} + q^2\sum_{n \geq 0} \dfrac{(q-q^{-1})^{n+1}}{ \{ n \}! } (q^2c_1c_2)^{n+1} T_2(E_1)^{n+1}T_2(E_3)^{n+1}\\
	&\qquad{}=	\Qkm^{[2]} + (q-q^{-1})(q^4c_1c_2)T_2(E_1)T_2(E_3)\Qkm^{[2]}.
\end{align*}
Inserting this equation into the expression for $\ir{2}{\Qkm^{[2]}}$, we obtain
\begin{align}
	\ir{2}{\Qkm^{[2]}}	&=(q-q^{-1})^2(qc_1c_2) \Big( E_3T_2(E_1) + E_1T_2(E_3) \Big) \Qkm^{[2]} \label{eq:A3 ir2 Qkm2} \\
							&\quad{} + (q-q^{-1})^2(q^2c_1c_2)T_{213}(E_2)\Qkm^{[2]} \nonumber \\ 
							&\qquad{}+ (q-q^{-1})^3(q^6c_1^2c_2^2)T_{213}(E_2)T_2(E_1)T_2(E_3)\Qkm^{[2]} \nonumber. 
\end{align}
Finally, we have
\begin{equation} \label{eq:A3 ir2 Qkm1}
	\ir{2}{\Qkm^{[1]}}	\overset{\eqref{eq:QkmRankOneA1xA1}}{=}	(q-q^{-1})(q^2c_2)E_2\Qkm^{[1]}.
\end{equation}
Now, we use Lemma \ref{A3commE2} to rewrite the term~$\Qkm_{[4]}\Qkm_{[3]}\Qkm_{[2]}\ir{2}{\Qkm^{[1]}}$. By calculations similar to those leading to \eqref{eq:A3 ir2 Qkm3} and \eqref{eq:A3 ir2 Qkm2}, we obtain
\begin{align}
	\Qkm_{[2]}E_2	&=	E_2\Qkm^{[2]}, \label{eq:A3 Qkm2E2}\\ 
	\Qkm_{[3]}E_2	&=	E_2\Qkm^{[3]} - (q-q^{-1})^2(q^4c_1^2c_2)\Qkm_{[3]}T_{213}(E_2)T_2(E_1)T_2(E_3), \label{eq:A3 Qkm3E2}\\
	\Qkm_{[4]}E_2	&=	E_2\Qkm^{[4]} - q^{-1}(q-q^{-1})c_1\Qkm_{[4]} \big(E_1T_2(E_3) + E_3T_2(E_1) \big) \label{eq:A3 Qkm4E2} \\
		&\qquad{}- (q-q^{-1})c_1 \big( \Qkm_{[4]} + (q-q^{-1})c_1\Qkm_{[4]}E_1E_3 \big)T_{213}(E_2).\nonumber
\end{align}
We use these to obtain the following.
\begin{align*}
\Qkm_{[4]}\Qkm_{[3]}&\Qkm_{[2]}\ir{2}{\Qkm^{[1]}}	\\
	&\overset{\eqref{eq:A3 ir2 Qkm1}}{=}	(q-q^{-1})(q^2c_2)\Qkm_{[4]}\Qkm_{[3]}\Qkm_{[2]}E_2\Qkm^{[1]}\\
	&\overset{\eqref{eq:A3 Qkm2E2}}{=}	(q-q^{-1})(q^2c_2)\Qkm_{[4]}\Qkm_{[3]}E_2\Qkm^{[2]}\Qkm^{[1]}\\
	&\overset{\eqref{eq:A3 Qkm3E2}}{=}	(q-q^{-1})(q^2c_2)\Qkm_{[4]}E_2\Qkm^{[3]}\Qkm^{[2]}\Qkm^{[1]}\\ &\qquad{}- (q-q^{-1})^3(q^6c_1^2c_2^2)\Qkm_{[4]}\Qkm_{[3]}T_{213}(E_2)T_2(E_1)T_2(E_3)\Qkm^{[2]}\Qkm^{[1]}\\
	&\overset{\eqref{eq:A3 Qkm4E2}}{=}	(q-q^{-1})(q^2c_2)\Big(E_2\Qkm^{[4]} - q^{-1}(q-q^{-1})c_1\Qkm_{[4]} \big(E_1T_2(E_3) + E_3T_2(E_1) \big)\\
				&\qquad{}- (q-q^{-1})c_1 \big( \Qkm_{[4]} + (q-q^{-1})c_1\Qkm_{[4]}E_1E_3 \big) T_{213}(E_2) \Big) \Qkm^{[3]}\Qkm^{[2]}\Qkm^{[1]}\\ 
				&\qquad\qquad{}- (q-q^{-1})^3(q^6c_1^2c_2^2)\Qkm_{[4]}\Qkm_{[3]}T_{213}(E_2)T_2(E_1)T_2(E_3)\Qkm^{[2]}\Qkm^{[1]}.
\end{align*}
We gather terms now and get
\begin{align*}
		\Qkm_{[4]}&\Qkm_{[3]}\Qkm_{[2]}\ir{2}{\Qkm^{[1]}}\\	
				&= (q-q^{-1})(q^2c_2)E_2\Qkm_{\widet{w_0}{\prime}} \\
					&\quad{}-(q-q^{-1})^2(qc_1c_2)\Qkm_{[4]} \big(E_1T_2(E_3) + E_3T_2(E_1) \big)\Qkm^{[3]}\Qkm^{[2]}\Qkm^{[1]}\\
					&\quad{}\quad{} - (q-q^{-1})^2(q^2c_1c_2)\Qkm_{[4]}\Qkm_{[3]}T_{213}(E_2)\Qkm^{[2]}\Qkm^{[1]}\\ 
					&\quad\quad\quad{}+ (q-q^{-1})^3(q^6c_1^2c_2^2)\Qkm_{[4]}\Qkm_{[3]}T_{213}(E_2)T_2(E_1)T_2(E_3)\Qkm^{[2]}\Qkm^{[1]}\\
						&\quad\quad\quad\quad{} -(q-q^{-1})^3(q^2c_1^2c_2)\Qkm_{[4]}E_1E_3T_{213}(E_2)\Qkm^{[3]}\Qkm^{[2]}\Qkm^{[1]}\\
			&= (q-q^{-1})(q^2c_2)E_2\Qkm_{\widet{w_0}^{\prime}} - \Qkm_{[4]}\Qkm_{[3]} \ir{2}{\Qkm^{[2]}}\Qkm^{[1]} - \Qkm_{[4]} \ir{2}{\Qkm^{[3]}} \Qkm^{[2]}\Qkm^{[1]}
\end{align*}
where we use the fact that $E_1T_2(E_3)$ and $E_3T_2(E_1)$ both commute with $T_{213}(E_2)$, and hence with $\Qkm^{[3]}$. It follows that
\begin{align*}
	\ir{2}{\Qkm_{\widet{w_0}^{\prime}}}
		&=	\Qkm_{[4]} \ir{2}{\Qkm^{[3]}} \Qkm^{[2]}\Qkm^{[1]} + \Qkm_{[4]}\Qkm_{[3]} \ir{2}{\Qkm^{[2]}} \Qkm^{[1]} + \Qkm_{[4]} \Qkm_{[3]}\Qkm_{[2]} \ir{2}{\Qkm^{[1]}}\\
		&=	(q-q^{-1})(q^2c_2)E_2\Qkm_{\widet{w_0}^{\prime}}
\end{align*}
as required.
\end{proof}
%


\subsection{Type $AIII_n$ for $n \geq 4$}\label{app:AIIIge4}

Consider the Satake diagram of type $AIII_n$ for $n \geq 4$.

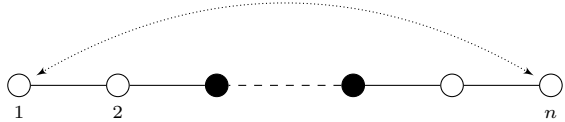
\begin{center}
	\begin{tikzpicture}
		[white/.style={circle,draw=black,inner sep = 0mm, minimum size = 3mm},
		black/.style={circle,draw=black,fill=black, inner sep= 0mm, minimum size = 3mm}]
	
		\node[white] (1)	 [label = below:{\scriptsize $1$}]	{};
		\node[white] (first) [right = of 1] [label = below:{\scriptsize $2$}]  {}
			edge (1);
		\node[black] (second) [right=of first] {}
			edge (first);
		\node[black] (last) [right=1.5cm of second] {}
			edge [dashed] (second);		
		\node[white] (fourth) [right=of last]  {}
			edge (last);
		\node[white] (2) [right = of fourth] [label = below:{\scriptsize $\phantom{1} n \phantom{1}$}]{}
			edge(fourth) 
			edge	 [latex'-latex' , shorten <=3pt, shorten >=3pt, bend right=30, densely dotted] node[auto,swap] {} (1);
	\end{tikzpicture}
\end{center}
In this case the restricted root system is of type $B_2$ with
\begin{equation*}
\widet{\alpha_1} =  \dfrac{\alpha_1 + \alpha_n}{2},	\qquad	\widet{\alpha_2} = \dfrac{\alpha_2 + \alpha_3 + \dotsm + \alpha_{n-1}}{2}.
\end{equation*}
The subgroup $\widet{W} \subset W^{\Theta}$ is generated by the elements
\begin{equation*}
\widet{s_1} = s_1s_n, \qquad
\widet{s_2} = w_Xw_{\{2,n-1\} \cup X} = s_2s_3 \dots s_{n-1} \dots s_3s_2.
\end{equation*}
The longest word of the restricted Weyl group has two reduced expressions given by
\begin{equation*}
\widet{w_0} = \widet{s_1}\widet{s_2}\widet{s_1}\widet{s_2}, \qquad
\widet{w_0^{\prime}} = \widet{s_2}\widet{s_1}\widet{s_2}\widet{s_1}.
\end{equation*}
The definition \eqref{eq:ParameterSetC} and condition \eqref{eq:ciCond} imply that $c_1=c_n=\overline{c_1}$. By Lemmas \ref{lem:QkmRankOneA1xA1} and \ref{lem:QkmRankOneAIV} we have
\begin{align*}
\Qkm_1	&=	\sum_{k \geq 0} \dfrac{ \q^k }{\{k\}!}c_1^k (E_1E_n)^k,\\
\Qkm_2
	&=	\bigg( \sum_{k \geq 0}\dfrac{(c_{2}s(n-1))^k}{\{k\}!} T_2T_{w_X}(E_{n-1})^k \bigg)
		\bigg( \sum_{k \geq 0}\dfrac{(c_{n-1}s(2))^k}{\{k\}!} T_{n-1}T_{w_X}(E_2)^k \bigg).
\end{align*}

\begin{proposition} \label{AppProp: AIII PartialQkm1}
The partial quasi $K$-matrix $\Qkm_{\widet{w_0}}$ coincides with the quasi $K$-matrix $\Qkm$.
\end{proposition}

We have the following relations needed in the proof of Proposition \ref{AppProp: AIII PartialQkm1}, proved by induction.

\begin{lemma} \label{AIIIRankTwoRels}
For any $k \in \mathbb{N}$ the relations
\begin{align}
&T_{1n(n-1)}T_{w_X}(E_2)^k E_n	=	q^{-k} E_n T_{1n(n-1)}T_{w_X}(E_2)^k, \nonumber \\
&T_{12n}T_{w_X}(E_{n-1})^kE_n		=	q^{-k} E_n T_{12n}T_{w_X}(E_{n-1})^k, \nonumber \\
&T_{n \cdots 3}(E_2)^k E_n		=	q^{-k} E_n T_{n \cdots 3}(E_2)^k, \nonumber \\
&T_{1 \cdots (n-2)}(E_{n-1})^k E_n
	=	q^k E_n T_{1 \cdots (n-2)}(E_{n-1})^k \nonumber\\ &\hspace{40pt}- q \{k \} T_{1 \cdots (n-2)}(E_{n-1})^{k-1} 								T_{12n}T_{w_X}(E_{n-1}), \nonumber \\
&T_{n-1}T_{w_X}(E_2)^k E_n
	=	q^k E_n T_{n-1}T_{w_X}(E_2)^k \nonumber \\
	&\hspace{40pt}- q \{k \} T_{n-1}T_{w_X}(E_2)^{k-1} T_{n \cdots 3}(E_2), \nonumber \\
&T_2T_{w_X}(E_{n-1})^k E_n
	=	q^k E_n T_2T_{w_X}(E_{n-1})^k  \nonumber \\
	&\hspace{40pt}- q \{k \} T_2T_{w_X}(E_{n-1})^{k-1} T_{2n}T_{w_X}(E_{n-1}), \nonumber \\
&T_{2n}T_{w_X}(E_{n-1}) T_{n-1}T_{w_X}(E_2)^k
	=	q^{-k}T_{n-1}T_{w_X}(E_2)^k T_{2n}T_{w_X}(E_{n-1}) \label{eq:AIIITrickComm1} \\
		&\hspace{40pt}+	q^{1-k} (q-q^{-1}) \{k \}T_{n-1}T_{w_X}(E_2)^{k-1}T_2T_{w_X}(E_{n-1})T_{n \cdots 3}(E_2),\nonumber \\
&T_{12n}T_{w_X}(E_{n-1})^{k} T_{n \cdots 3}(E_2)  \label{eq:AIIITrickComm2}\\
	&\hspace{40pt} =	\big( q^{-k} + (q-q^{-1})q^{1-k}\{ k \} \big) T_{n \cdots 3}(E_2) T_{12n}T_{w_X}(E_{n-1})^{k} \nonumber
\end{align}
hold in $U_q(\lie{sl}_{n+1})$.
\end{lemma}

\begin{lemma} \label{AIIIRankTwoLusDer}
For any $k \in \mathbb{N}$ the relations 
\begin{align}
&\ir{1}{T_{1 \cdots (n-2)}(E_{n-1})^k}  \label{eq: AIII ir1 T1...n-2(En-1)} \\
	&\hspace{40pt}=	q^{-1}(q-q^{-1}) \{k \} T_{2 \cdots (n-2)}(E_{n-1})T_{1 \cdots (n-2)}(E_{n-1})^{k-1}, \nonumber \\
&\ir{1}{T_{12n}T_{w_X}(E_{n-1})^k} \label{eq: AIII ir1 T12nTwX(En-1)} \\
	&\hspace{40pt}=	q^{-1}(q-q^{-1}) \{k \} T_{2n}T_{w_X}(E_{n-1})T_{12n}T_{w_X}(E_{n-1})^{k-1}, \nonumber \\
&\ir{1}{T_{1n(n-1)}T_{w_X}(E_2)^k} \label{eq: AIII ir1 T1n(n-1)TwX(E_2)} \\
	&\hspace{40pt}=	q^{-1}(q-q^{-1}) \{k \} T_{n(n-1)}T_{w_X}(E_2)T_{1n(n-1)}T_{w_X}(E_2)^{k-1} \nonumber
\end{align}
hold in $U_q(\lie{sl}_{n+1})$.
\end{lemma}

One can show that the above Lemmas still hold if we consider the case where we have no black dots. In this situation, the calculations differ slightly but the results still hold.

\begin{proof}[Proof of Proposition \ref{AppProp: AIII PartialQkm1}]
We write
\begin{equation}
	\Qkm_{\widet{w_0}} = \Qkm^{[4]}\Qkm^{[3]}\Qkm^{[2]}\Qkm^{[1]},
\end{equation}
where, recalling the notation $\widet{c_2}^2 = c_2c_{n-1}s(n-1)s(2)$, we have
\begin{align*}
	\Qkm^{[4]}	&\overset{\eqref{eq:LongestWordPartialQkmPart}}{=}	
						 \Qkm_2,\\
	\Qkm^{[3]}	&\overset{\phantom{\eqref{eq:LongestWordPartialQkmPart}}}{=}		\sum_{k \geq 0}\dfrac{(q-q^{-1})^k}{\{k\}!}(qc_1\widet{c_2}^2)^kT_{n \cdots 3}(E_2)^kT_{1 \cdots (n-2)}(E_{n-1})^k,\\
	\Qkm^{[2]}	&\overset{\phantom{\eqref{eq:LongestWordPartialQkmPart}}}{=}		\bigg(\sum_{k \geq 0} \dfrac{(qc_1c_2s(n-1))^{k}}{\{k\}!}T_{12n}T_{w_X}(E_{n-1})^{k} \bigg)\\ &\qquad\qquad\qquad\qquad\qquad \bigg( \sum_{k \geq 0} \dfrac{(qc_1c_{n-1}s(2))^{k}}{\{k\}!} T_{1n(n-1)}T_{w_X}(E_2)^{k} \bigg),\\
	\Qkm^{[1]}	&\overset{\phantom{\eqref{eq:LongestWordPartialQkmPart}}}{=}	\Qkm_1.
\end{align*}
When necessary, since $\Qkm^{[4]}$ and $\Qkm^{[2]}$ are a product of two infinite sums, we write $\Qkm^{[4]} = \Qkm^{[4;1]}\Qkm^{[4;2]}$ and $\Qkm^{[2]} = \Qkm^{[2;1]}\Qkm^{[2;2]}$.

For each $i = 1, 2, 3, 4$, let $\Qkm_{[i]} = K_1\Qkm^{[i]}K_1^{-1}$. We write $\Qkm_{[4]} = \Qkm_{[4;1]}\Qkm_{[4;2]}$ and $\Qkm_{[2]}= \Qkm_{[2;1]}\Qkm_{[2;2]}$. Now, by the rank one case for $\Qkm_2$ given in Lemma \ref{lem:QkmRankOneAIV} and \cite[8.26, (4)]{b-Jantzen96}, we obtain
\begin{equation}
	\ir{2}{\Qkm_{\widet{w_0}}} = (q-q^{-1})c_2s(n-1)T_{w_X}(E_{n-1})\Qkm_{\widet{w_0}}.
\end{equation}
The underlying symmetry implies that we only need to show that
\begin{equation*}
\ir{1}{\Qkm_{\widet{w_0}}} = \q c_1 E_n \Qkm_{\widet{w_0}}.
\end{equation*}
By Property \eqref{eq:ir} of the skew derivative ${}_1r$ and \cite[8.26, (4)]{b-Jantzen96} we have
\begin{equation} \label{AppEq: AIII ir1Qkmw}
\ir{1}{\Qkm_{\widet{w_0}}} = \Qkm_{[4]} \ir{1}{\Qkm^{[3]}} \Qkm^{[2]}\Qkm^{[1]} + \Qkm_{[4]}\Qkm_{[3]}\ir{1}\Qkm^{[2]}\Qkm^{[1]} + \Qkm_{[4]}\Qkm_{[3]}\Qkm_{[2]}\ir{1}{\Qkm^{[1]}}.
\end{equation}
Using Lemma \ref{AIIIRankTwoLusDer} and Lemma \ref{lem:QkmRankOneA1xA1} we find that
\begin{align}
	\ir{1}{\Qkm^{[3]}}	&\overset{\eqref{eq: AIII ir1 T1...n-2(En-1)}}{=}	(q-q^{-1})^2 c_1 \widet{c_2}^2T_{n \cdots 3}(E_2)T_{2 \cdots (n-2)}(E_{n-1})\Qkm^{[3]},\label{eq:AIIIr1X3}\\
	\ir{1}{\Qkm^{[1]}}	&\overset{\eqref{eq:QkmRankOneA1xA1}}{=}	(q-q^{-1})c_1E_n\Qkm^{[1]} \label{eq:AIIIr1X1}.
\end{align}
To write an expression for $\ir{1}{\Qkm^{[2]}}$, we use the splitting of $\Qkm^{[2]}$ into a product of two infinite sums.  By equations \eqref{eq: AIII ir1 T12nTwX(En-1)} and \eqref{eq: AIII ir1 T1n(n-1)TwX(E_2)} of Lemma \ref{AIIIRankTwoLusDer} we have
\begin{align}
	\ir{1}{\Qkm^{[2]}}	&=	\ir{1}{\Qkm^{[2;1]}}\Qkm^{[2;2]} + \Qkm_{[2;1]}\ir{1}{\Qkm^{[2;2]}} \label{eq: AIII ir1Qkm2}\\
						&=	(q-q^{-1})c_1c_2s(n-1)T_{2n}T_{w_X}(E_{n-1})\Qkm^{[2]} \nonumber \\
							&\qquad{}+	(q-q^{-1})c_1c_{n-1}s(2)\Qkm_{[2;1]} T_{n(n-1)}T_{w_X}(E_2)\Qkm^{[2;2]}. \nonumber
\end{align} 
We would like the last summand of this expression to be in terms of $\Qkm^{[2]}$. To this end, we use Equation \eqref{eq:AIIITrickComm2} in Lemma \ref{AIIIRankTwoRels} to obtain
\begin{align*}
&\Qkm_{[2;1]}T_{n(n-1)}T_{w_X}(E_2)\\	
	&\quad{}\overset{\phantom{\eqref{eq:AIIITrickComm2}}}{=}
		\bigg( 1 + \sum_{k \geq 1} \dfrac{(qc_1c_2s(n-1))^{k}}{\{k\}!}q^{k}T_{12n}T_{w_X}(E_{n-1})^{k} \bigg) T_{n(n-1)}T_{w_X}(E_2)\\
	&\quad{}\overset{\eqref{eq:AIIITrickComm2}}{=}	
		T_{n(n-1)}T_{w_X}(E_2)\Qkm^{[2;1]}\\ &\qquad\qquad{}+ q^2(q-q^{-1})c_1c_2s(n-1)T_{n(n-1)}T_{w_X}(E_2)T_{12n}T_{w_X}(E_{n-1})\Qkm^{[2;1]}.	 			
\end{align*}
Substituting this into \eqref{eq: AIII ir1Qkm2} it follows that
\begin{align}
	\ir{1}{\Qkm^{[2]}}	&=	(q-q^{-1})c_1c_2s(n-1)T_{2n}T_{w_X}(E_{n-1})\Qkm^{[2]} \label{eq:AIIIr1X2}\\ &\quad{}+ (q-q^{-1})c_1c_{n-1}s(2)T_{n(n-1)}T_{w_X}(E_2)\Qkm^{[2]} \nonumber \\ 
							&\quad\quad{}	+	q^2(q-q^{-1})^2 c_1^2\widet{c_2}^2T_{n(n-1)}T_{w_X}(E_2)T_{12n}T_{w_X}(E_{n-1})\Qkm^{[2]}.\nonumber
\end{align}
When we calculate $\ir{1}{\Qkm_{\widet{w_0}}}$, we obtain a component of the form $\Qkm_{[4]}\Qkm_{[3]}\Qkm_{[2]}\ir{1}{\Qkm^{[1]}}$. From Equation \eqref{eq:AIIIr1X1}, we see that we obtain an $E_n$ term that we want to pass to the front of this expression. We use Lemma \ref{AIIIRankTwoRels} to do this. We have
\begin{align}
	\Qkm_{[2]}E_n	&=	E_n\Qkm^{[2]}, \label{eq:AIIIX2En}\\
	\Qkm_{[3]}E_n	&=	E_n\Qkm_{[3]}
						- q^2(q-q^{-1})c_1 \widet{c_2}^2T_{n \cdots 3}(E_2)\Qkm^{[3]}T_{12n}T_{w_X}(E_{n-1}),\label{eq:AIIIX3En}\\
	\Qkm_{[4;2]}E_n	&=	E_n\Qkm^{[4;2]} - c_1c_{n-1}s(2)\Qkm_{[4;2]}T_{n \cdots 3}(E_2),\label{eq:AIIIX41En}\\
	\Qkm_{[4;1]}E_n	&=	E_n\Qkm^{[4;1]} - c_1c_2s(n-1)\Qkm_{[4;1]}T_{2n}T_{w_X}(E_{n-1})\label{eq:AIIIX42En}.
\end{align} 
Now, using Equation \eqref{eq:AIIITrickComm1}, we obtain the following expression for $\Qkm_{[4]}E_n$.
\begin{align}
\Qkm_{[4]}E_n	
	&\overset{\eqref{eq:AIIIX41En}}{=}	
		\Qkm_{[4;1]} \big( E_n\Qkm^{[4;2]} - c_1 c_{n-1}s(2)\Qkm_{[4;2]}T_{n \cdots 3}(E_2) \big) \nonumber \\
	&\overset{\eqref{eq:AIIIX42En}}{=}	
		\big( E_n\Qkm^{[4;1]} - c_1c_2s(n-1)\Qkm_{[4;1]}T_{2n}T_{w_X}(E_{n-1})\big)\Qkm^{[4;2]} \nonumber\\&\qquad{}-							c_1c_{n-1}s(2)\Qkm_{[4]}T_{n \cdots 3}(E_2)\nonumber \\
	&\overset{\eqref{eq:AIIITrickComm1}}{=}	
		E_n\Qkm^{[4]} - c_1c_{n-1}s(2)\Qkm_{[4]}T_{n \cdots 3}(E_2)\nonumber \\
	&\qquad{}	-	c_1c_2s(n-1)\Qkm_{[4;1]}\Qkm_{[4;2]}T_{2n}T_{w_X}(E_{n-1})\nonumber \\ &\qquad\quad{}+ (q-									q^{-1})c_1^2c_2c_{n-1}s(2)s(n-1)\Qkm_{[4;1]}\Qkm_{[4;2]}T_2T_{w_X}(E_{n-1})T_{n \cdots 3}(E_2) \nonumber \\
	&\overset{\phantom{\eqref{eq:AIIITrickComm1}}}{=} 	
		E_n\Qkm^{[4]} - c_1c_{n-1}s(2)\Qkm_{[4]}T_{n \cdots 3}(E_2) - c_1c_2s(n-1)\Qkm_{[4]}T_{2n}T_{w_X}(E_{n-1}) \label{eq:AIIIX4En}\\
	&\qquad{} - (q-q^{-1})c_1^2\widet{c_2}^2\Qkm_{[4]}T_2T_{w_X}(E_{n-1})T_{n \cdots 3}(E_2). \nonumber				
\end{align}
Using equations \eqref{eq:AIIIX2En}, \eqref{eq:AIIIX3En} and \eqref{eq:AIIIX4En}, and comparing with equations \eqref{eq:AIIIr1X3}, \eqref{eq:AIIIr1X1} and \eqref{eq:AIIIr1X2}, one finds that
\begin{align*}
	\Qkm_{[4]}\Qkm_{[3]}\Qkm_{[2]} \ir{1}{\Qkm^{[1]}}
		&=	(q-q^{-1})\Qkm_{[4]}\Qkm_{[3]}\Qkm_{[2]}E_n\Qkm^{[1]}\\
		&=  (q-q^{-1})E_n\Qkm_{\widet{w_0}} - \Qkm_{[4]} \ir{1}{\Qkm^{[3]}}\Qkm^{[2]}\Qkm^{[1]} - \Qkm_{[4]}\Qkm_{[3]}\ir{1}{\Qkm^{[2]}}\Qkm^{[1]},
\end{align*}
and hence it follows that
\begin{equation*}
	\ir{1}{\Qkm_{\widet{w_0}}}	=	(q-q^{-1})E_n\Qkm_{\widet{w_0}},
\end{equation*}
as required. This completes the proof.
\end{proof}

\begin{proposition} \label{Approp:AIII PartialQkm2}
The partial quasi $K$-matrix $\Qkm_{\widet{w_0}^{\prime}}$ coincides with the quasi $K$-matrix $\Qkm$.
\end{proposition}

We have the following relations needed in the proof of Proposition \ref{Approp:AIII PartialQkm2}, proved by induction.

\begin{lemma} \label{AIIIRankTwoRelsb}
For any $k \in \mathbb{N}$ the relations
\begin{align}
	&T_{2 \cdots (n-1)}(E_n)^kT_{w_X}(E_{n-1}) = T_{w_X}(E_{n-1})T_{2 \cdots (n-1)}(E_n)^k,\nonumber \\
	&T_{(n-1) \cdots 2}(E_1)^kT_{w_X}(E_{n-1}) = q^{-k}T_{w_X}(E_{n-1})T_{(n-1) \cdots 2}(E_1)^k, \nonumber\\
	&T_{2 \cdots (n-1)}T_{12n}T_{w_X}(E_{n-1})^kT_{w_X}(E_{n-1}) = T_{w_X}(E_{n-1})T_{2 \cdots (n-1)}T_{12n}T_{w_X}(E_{n-1})^k,\nonumber \\
	&T_{(n-1) \cdots 2}T_{1n(n-1)}T_{w_X}(E_2)^kT_{w_X}(E_{n-1})
		= T_{w_X}(E_{n-1})T_{(n-1) \cdots 2}T_{1n(n-1)}T_{w_X}(E_2)^k \nonumber\\ 
			&\qquad{}- (q-q^{-1})\{k\}T_{(n-1) \cdots 2}T_{1n(n-1)}T_{w_X}(E_2)^{k-1}T_{3 \cdots (n-1)}(E_n)T_{(n-1) \cdots 2}(E_1), \nonumber \\
	 &E_1^kT_{w_X}(E_{n-1}) = T_{w_X}(E_{n-1})E_1^k, \nonumber \\
	 &E_n^kT_{w_X}(E_{n-1}) = q^kT_{w_X}(E_{n-1})E_n^k - q\{k\}E_n^{k-1}T_{w_X}T_{n-1}(E_n) \nonumber 
\end{align}
hold in $U_q(\lie{sl}_{n+1})$.
\end{lemma}

\begin{lemma} \label{AIIIRankTwoLusDerb}
For any $k \in \mathbb{N}$ the relations
\begin{align}
	&\ir{2}{T_{2 \cdots (n-1)}(E_n)^k} = q^{-1}(q-q^{-1})\{k\}T_{3 \cdots (n-1)}(E_n)T_{2 \cdots (n-1)}(E_n)^{k-1}, \nonumber \\
	&\ir{2}{T_{2 \cdots (n-1)}T_{12n}T_{w_X}(E_{n-1})^k} \nonumber \\ &\qquad{}= q^{-2}(q-q^{-1})^2\{k\}E_1T_{3 \cdots (n-1)}(E_n)T_{2 \cdots (n-1)}T_{12n}T_{w_X}(E_{n-1})^{k-1} \nonumber 
\end{align}
hold in $U_q(\lie{sl}_{n+1})$.
\end{lemma}

\begin{proof}[Proof of Proposition \ref{Approp:AIII PartialQkm2}]
We write
\begin{equation*}
	\Qkm_{\widet{w_0}^{\prime}}	=	\Qkm^{[4]}\Qkm^{[3]}\Qkm^{[2]}\Qkm^{[1]},
\end{equation*}
where we have
\begin{align*}
\Qkm^{[4]}	&\overset{\eqref{eq:LongestWordPartialQkmPart}}{=}  \Qkm_1,	\\
\Qkm^{[3]}	&\overset{\phantom{\eqref{eq:LongestWordPartialQkmPart}}}{=}
		\bigg(	\sum_{k_1 \geq 0}\dfrac{ (qc_1c_2s(n-1))^{k_1} }{ \{ k_1 \}! } T_{2 \cdots (n-1)}T_{12n}T_{w_X}(E_{n-1})^{k_1} \bigg)\\ 
					&\qquad\qquad\qquad\qquad\qquad \bigg(	\sum_{k_2 \geq 0}\dfrac{ (qc_1c_{n-1}s(2))^{k_2} }{ \{ k_2 \}! } T_{(n-1) \cdots 2}T_{1n(n-1)}T_{w_X}(E_{2})^{k_2} \bigg),\\
	\Qkm^{[2]}	&\overset{\phantom{\eqref{eq:LongestWordPartialQkmPart}}}{=}		\sum_{k \geq 0} \dfrac{(q-q^{-1})^k}{\{k\}!}(qc_1\widet{c_2}^2)^k T_{(n-1) \cdots 2}(E_1)^k T_{2 \cdots (n-1)}(E_n)^k,\\
	\Qkm^{[1]}	&\overset{\phantom{\eqref{eq:LongestWordPartialQkmPart}}}{=}	\Qkm_2.
\end{align*}
For $i = 1, 2, 3, 4$ let $\Qkm_{[i]} = K_2\Qkm^{[i]}K_2^{-1}$.
By Lemma \ref{lem:QkmRankOneA1xA1} and \cite[8.26, (4)]{b-Jantzen96} it follows that
\begin{align*}
	\ir{1}{\Qkm_{\widet{w_0}^{\prime}}}	&=	\ir{1}{\Qkm^{[4]}} \Qkm^{[3]}\Qkm^{[2]}\Qkm^{[1]}\\
								&=	(q-q^{-1})E_n\Qkm_{\widet{w_0}^{\prime}}.
\end{align*}
Hence, we only need to check that
\begin{equation*}
	\ir{2}{\Qkm_{\widet{w_0}^{\prime}}}	=	q^{-1}(q-q^{-1})c_2s(n-1)T_{w_X}(E_{n-1})\Qkm_{\widet{w_0}^{\prime}}.
\end{equation*}
By \cite[8.26, (4)]{b-Jantzen96} and property \eqref{eq:ir} of the skew derivative ${}_2r$ we have
\begin{equation*}
\ir{2}{\Qkm_{\widet{w_0}^{\prime}}} = \Qkm_{[4]}\ir{2}{\Qkm^{[3]}}\Qkm^{[2]}\Qkm^{[1]} + \Qkm_{[4]}\Qkm_{[3]}\ir{2}{\Qkm^{[2]}}\Qkm^{[1]} + \Qkm_{[4]}\Qkm_{[3]}\Qkm_{[2]}\ir{2}{\Qkm^{[1]}}.
\end{equation*}
Using Lemmas \ref{lem:QkmRankOneAIV} and \ref{AIIIRankTwoLusDerb}, we have
\begin{align*}
	\ir{2}{\Qkm^{[3]}}	&=	q^{-1}(q-q^{-1})^2c_1c_2s(n-1)E_1T_{3 \cdots (n-1)}(E_n)\Qkm^{[3]},\\
	\ir{2}{\Qkm^{[2]}}	&=	(q-q^{-1})^2c_1\widet{c_2}^2T_{3 \cdots (n-1)}(E_n)T_{(n-1) \cdots 2}(E_1)\Qkm^{[2]},\\
	\ir{2}{\Qkm^{[1]}}	&=	q^{-1}(q-q^{-1})c_2s(n-1)T_{w_X}(E_{n-1})\Qkm^{[1]}.
\end{align*}
By Lemma \ref{AIIIRankTwoRelsb}, we have
\begin{align*}
	\Qkm_{[2]}T_{w_X}(E_{n-1})	&=	T_{w_X}(E_{n-1})\Qkm^{[2]},\\
	\Qkm_{[3]}T_{w_X}(E_{n-1})	&=	T_{w_X}(E_{n-1})\Qkm^{[3]}\\ &\quad{}- q(q-q^{-1})c_1c_{n-1}s(2)\Qkm_{[3]}T_{3 \cdots (n-1)}(E_n)T_{(n-1) \cdots 2}(E_1),\\
	\Qkm_{[4]}T_{w_X}(E_{n-1})	&=	T_{w_X}(E_{n-1})\Qkm^{[4]} - (q-q^{-1})c_1\Qkm_{[4]}E_1T_{3 \cdots (n-1)}(E_n).
\end{align*}
It follows that
\begin{align*}
	\Qkm_{[4]}\Qkm_{[3]}\Qkm_{[2]}\ir{2}{\Qkm^{[1]}} &= q^{-1}(q-q^{-1})c_2s(n-1)T_{w_X}(E_{n-1})\Qkm_{w^{\prime}} \\
		&\quad{}- \Qkm_{[4]}\ir{2}{\Qkm^{[3]}}\Qkm^{[2]}\Qkm^{[1]} - \Qkm_{[4]}\Qkm_{[3]}\ir{2}{\Qkm^{[2]}}\Qkm^{[1]}
\end{align*}
and therefore we obtain
\begin{equation*}
	\ir{2}{\Qkm_{\widet{w_0}^{\prime}}}	=	q^{-1}(q-q^{-1})c_2s(n-1)T_{w_X}(E_{n-1})\Qkm_{\widet{w_0}^{\prime}}
\end{equation*}
as required.
\end{proof}


\subsection{Type $CI_2$}
 Consider the Satake diagram of type $CI_2$.
 
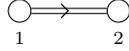
\begin{center}
	\begin{tikzpicture}
		[white/.style={circle,draw=black,inner sep = 0mm, minimum size = 3mm},
		black/.style={circle,draw=black,fill=black, inner sep= 0mm, minimum size = 3mm}]
		
		\node[white] (first) [label = below:{\scriptsize $1$}] {};		
		\node[white] (second) [right=of first] [label = below:{\scriptsize $2$}] {}			
			edge [double equal sign distance, -<-] (first);
	\end{tikzpicture}
\end{center}
Since $\Theta = -\textrm{id}$ the subgroup $\widet{W}$ coincides with $W$. 
The longest word of the Weyl group has two reduced expressions given by
\begin{equation*}
	w_0 = s_1s_2s_1s_2,\qquad
	w_0^{\prime} = s_2s_1s_2s_1.	
\end{equation*}
By Lemma \ref{lem:QkmRankOneAI} we have
\begin{align*}
\Qkm_1	&=	\sum_{n \geq 0} \dfrac{\Q{2}^n}{\{2n\}_1!!} (q^4c_1)^nE_1^{2n},\\
\Qkm_2	&=	\sum_{n \geq 0} \dfrac{\q^n}{\{2n\}_2!!} (q^2c_2)^nE_2^{2n}.
\end{align*}

\begin{proposition} \label{AppProp:BC1212}
The partial quasi $K$-matrix $\Qkm_{w_0}$ coincides with the quasi $K$-matrix $\Qkm$.
\end{proposition}

The following relations are necessary for the proof of Proposition \ref{AppProp:BC1212}, proved by induction.

\begin{lemma} \label{AppLemBC: E1comms}
For any $n \in \mathbb{N}$ the relations
\begin{align}
&E_2^nE_1 
	= 	q^{2n}E_1E_2^{n} - q^2\{n\}_2E_2^{n-1}T_1(E_2) - q^3\{n\}_2\{n-1\}_2E_2^{n-1}T_{12}(E_1), \label{AppEqBC: E2E1comm}\\
&T_1(E_2)^nE_1
	=	q^{-2n}E_1T_1(E_2)^n, \label{AppEqBC: T1E2E1comm}\\
&T_{12}(E_1)^nE_1
	=	E_1T_{12}(E_1)^n - \frac{(q^2-1)}{[2]_2}\{n\}_1T_{12}(E_1)^{n-1}T_1(E_2)^2 \label{AppEqBC: T12E1E1comm}	
\end{align}
hold in $U_q(\lie{so}_5)$.
\end{lemma}  

\begin{lemma} \label{AppLemBC: 1r}
For any $n \in \mathbb{N}$ the relations 
\begin{align}
&\ir{1}{T_1(E_2)^{n+1}}
	=	q^{-2}\Q{2} \{n+1\}_2 E_2T_1(E_2)^{n} \label{AppEqBC: ir1(T1E2)}\\ &\qquad\qquad{}+ q^{-1}\Q{2} \{n+1\}_2\{n\}_2T_{12}(E_1)T_1(E_2)^{n-1}, \nonumber\\
&\ir{1}{T_{12}(E_1)^n}
	=	q^{-3}\q^2\{n\}_1E_2^2T_{12}(E_1)^{n-1} \label{AppEqBC: ir1(T12E1)}
\end{align}
hold in $U_q(\lie{so}_5)$.
\end{lemma}

\begin{proof}[Proof of Proposition \ref{AppProp:BC1212}]
We have 
\begin{equation*}
	\Qkm_{w_0} = \Qkm^{[4]}\Qkm^{[3]}\Qkm^{[2]}\Qkm^{[1]}
\end{equation*}
where
\begin{align*}
\Qkm^{[4]}	&\overset{\eqref{eq:LongestWordPartialQkmPart}}{=}  \Qkm_2, \\
\Qkm^{[3]}	&\overset{\phantom{\eqref{eq:LongestWordPartialQkmPart}}}{=}	\Psi \circ T_{12} \circ \Psi^{-1}(\Qkm_1)
			=\sum_{n \geq 0} \dfrac{\Q{2}^n}{\{2n\}_1!!} (q^4c_1)^n(q^2c_2)^{2n}T_{12}(E_1)^{2n},\\
\Qkm^{[2]}	&\overset{\phantom{\eqref{eq:LongestWordPartialQkmPart}}}{=}	\Psi \circ T_1 \circ \Psi^{-1}(\Qkm_2)
			=	\sum_{n \geq 0} \dfrac{\q^n}{\{2n\}_2!!} (q^4c_1)^n(q^2c_2)^n T_1(E_2)^{2n},\\
\Qkm^{[1]}	&\overset{\phantom{\eqref{eq:LongestWordPartialQkmPart}}}{=}	\Qkm_1.
\end{align*}
By Lemma \ref{lem:QkmRankOneAI} and \cite[8.26, (4)]{b-Jantzen96} we have
\begin{align*}
\ir{2}{\Qkm_{w_0}}	&=	\ir{2}{\Qkm_2}\Qkm^{[3]}\Qkm^{[2]}\Qkm^{[1]}\\
					&=	\q (q^2c_2) E_2 \Qkm_2\Qkm^{[3]}\Qkm^{[2]}\Qkm^{[1]}\\
					&=	\q (q^2c_2) E_2 \Qkm_{w_0}.
\end{align*}
We want to show that
\begin{equation*}
\ir{1}{\Qkm_{w_0}}	=	\Q{2}(q^4c_1)E_1\Qkm_{w_0}.
\end{equation*}
For each $i = 1,2,3,4$ let $\Qkm_{[i]} = K_1\Qkm^{[i]}K_1^{-1}$.
Note that $\Qkm_{[3]} = \Qkm^{[3]}$.
By property \eqref{eq:ir} of the skew derivative ${}_1r$ we see that
\begin{equation} \label{AppeqBC: 1rXw0}
	\ir{1}{\Qkm_{w_0}} = \Qkm_{[4]}\ir{1}{\Qkm^{[3]}}\Qkm^{[2]}\Qkm^{[1]} + \Qkm_{[4]}\Qkm_{[3]}\ir{1}{\Qkm^{[2]}}\Qkm^{[1]} + \Qkm_{[4]}\Qkm_{[3]}\Qkm_{[2]}\ir{1}{\Qkm^{[1]}}.
\end{equation}
Lemma \ref{AppLemBC: 1r} gives
\begin{align}
\ir{1}{\Qkm^{[3]}}
	&=	q^{-3}\q^2 \Q{2} (q^4c_1)(q^2c_2)^2 E_2^2T_{12}(E_1)\Qkm^{[3]}, \label{AppEqBC: 1rX3}\\
\ir{1}{\Qkm^{[1]}} 
	&=	\Q{2}(q^4c_1)E_1\Qkm^{[1]}. \label{AppEqBC: 1rX1}
\end{align}
By Equation \eqref{AppEqBC: ir1(T1E2)} we have
\begin{equation} \label{AppEqBC: irX2}
\begin{split}
	\ir{1}{\Qkm^{[2]}}
		&= q^{-2}\q \Q{2}(q^4c_1)(q^2c_2)E_2T_1(E_2)\Qkm^{[2]}\\
		&\quad{}+ q^{-1}\q\Q{2}(q^6c_1c_2)T_{12}(E_1) \widehat{\Qkm}^{[2]}
\end{split}
\end{equation}
where 
\begin{equation*}
\widehat{\Qkm}^{[2]} = \sum_{n \geq 0} \dfrac{\q^n}{\{2n\}_2!!}(q^6c_1c_2)^n \{2n+1\}_2T_1(E_2)^{2n}.
\end{equation*}
For any $n \geq 1$ we have
\begin{equation} \label{AppEqBC: Trick}
\{2n+1\}_2 = 1 + q^2\{2n\}_2.
\end{equation}
It hence follows that
\begin{equation*}
\begin{split}
\sum_{n \geq 0} \dfrac{\q^n}{\{2n\}_2!!}(q^4c_1)^n(q^2c_2)^n& \{2n+1\}_2T_1(E_2)^{2n}\\
	&=	\Qkm^{[2]} + q^2\q (q^4c_1)(q^2c_2)T_1(E_2)^2\Qkm^{[2]}.
\end{split}
\end{equation*}
Substituting this into Equation \eqref{AppEqBC: irX2} we see that
\begin{equation} \label{AppEqBC: 1rX2}
\begin{split}
	\ir{1}{\Qkm^{[2]}}
		&= q^{-2}\q \Q{2}(q^4c_1)(q^2c_2)E_2T_1(E_2)\Qkm^{[2]}\\
		&\quad{}+ q^{-1}\q\Q{2}(q^4c_1)(q^2c_2)T_{12}(E_1)\Qkm^{[2]}\\
		&\quad{}\quad{}+ q \q^2 \Q{2} (q^4c_1)^2(q^2c_2)^2T_{12}(E_1)T_1(E_2)^2\Qkm^{[2]}
\end{split}	
\end{equation}
When we calculate $\ir{1}{\Qkm_{w_0}}$, we obtain a component of the form $\Qkm_{[4]}\Qkm_{[3]}\Qkm_{[2]}\ir{1}{\Qkm^{[1]}}$. From Equation \eqref{AppEqBC: 1rX1}, this contains an $E_1$ term that we pass to the front using Lemma \ref{AppLemBC: E1comms}. We have
\begin{align}
\Qkm_{[2]}E_1 &\overset{\eqref{AppEqBC: T1E2E1comm}}{=} E_1\Qkm^{[2]}, \label{AppEqBC: Qkm2E1} \\
\Qkm_{[3]}E_1 &\overset{\eqref{AppEqBC: T12E1E1comm}}{=}	E_1\Qkm^{[3]} - q\q^2(q^4c_1)(q^2c_2)^2\Qkm_{[3]}T_{12}(E_1)T_1(E_2)^2, \label{AppEqBC: Qkm3E1}\\
\Qkm_{[4]}E_1 &\overset{\eqref{AppEqBC: E2E1comm}}{=}	E_1\Qkm^{[4]} - q^{-2}\q (q^2c_2)E_2\Qkm_{[4]}T_1(E_2) \label{AppEqBC: Qkm4E1}\\ &\qquad- q^{-1}\q (q^2c_2) \Qkm_{[4]}T_{12}(E_1) \nonumber \\   &\qquad{}\quad{}- q^{-3}\q^2(q^2c_2)^2\Qkm_{[4]}E_2^2T_{12}(E_1) \nonumber 
\end{align}
where we also use \eqref{AppEqBC: Trick} to obtain \eqref{AppEqBC: Qkm4E1}.
Note that $\Qkm^{[3]}$ commutes with $E_2T_1(E_2)$. Using \eqref{AppEqBC: Qkm2E1}, \eqref{AppEqBC: Qkm3E1} and \eqref{AppEqBC: Qkm4E1}, and comparing with \eqref{AppEqBC: 1rX3} and \eqref{AppEqBC: 1rX2} we obtain
\begin{align*}
\Qkm_{[4]}\Qkm_{[3]}&\Qkm_{[2]}\ir{1}{\Qkm^{[1]}}\\ &= \Q{2}(q^4c_1)E_1\Qkm_{w_0} - \Qkm_{[4]}\Qkm_{[3]}\ir{1}{\Qkm^{[2]}}\Qkm^{[1]} - \Qkm_{[4]}\ir{1}{\Qkm^{[3]}}\Qkm^{[2]}\Qkm^{[1]}.
\end{align*}
Hence by \eqref{AppeqBC: 1rXw0} we have
\begin{equation*}
\ir{1}{\Qkm_{w_0}} = \Q{2}(q^4c_1)E_1\Qkm_{w_0}
\end{equation*} 
as required.
\end{proof}

We now consider the reduced expression $w_0^{\prime} = s_2s_1s_2s_1$.

\begin{proposition} \label{AppPropBC: 2121}
The partial quasi $K$-matrix $\Qkm_{w_0^{\prime}}$ coincides with the quasi $K$-matrix $\Qkm$.
\end{proposition}

The following relations are needed for the proof and are obtained by induction.

\begin{lemma} \label{AppLemBC: E2comms}
For any $n \in \mathbb{N}$ the relations
\begin{align}
E_1^nE_2
	&=	q^{2n}E_2E_1^n - q^2\{n\}_1 E_1^{n-1}T_{21}(E_2), \label{AppEqBC: E1E2}\\
T_2(E_1)^nE_2
	&=	q^{-2n}E_2T_2(E_1)^n, \label{AppEqBC: T2E1E2}\\
T_{21}(E_2)^nE_2
	&=	E_2T_{21}(E_2)^n - [2]_2\{n\}_2T_{21}(E_2)^{n-1}T_2(E_1) \label{AppEqBC: T21E2E2}
\end{align}
hold in $U_q(\lie{so}_5)$.
\end{lemma}

\begin{lemma} \label{AppLemBC: 2r}
For any $n \in \mathbb{N}$ the relations
\begin{align}
\ir{2}{T_2(E_1)^n}		&=	\q \{n\}_1T_{21}(E_2)T_2(E_1)^{n-1}, \label{AppEqBC: 2r(T2E1)}\\
\ir{2}{T_{21}(E_2)^n}	&=	q^{-2}\Q{2}\{n\}_2E_1T_{21}(E_2)^{n-1} \label{AppEqBC: 2r(T21E2)}
\end{align}
hold in $U_q(\lie{so}_5)$.
\end{lemma}

\begin{proof}[Proof of Proposition \ref{AppPropBC: 2121}]
We have 
\begin{equation*}
\Qkm_{w_0^{\prime}} = \Qkm^{[4]}\Qkm^{[3]}\Qkm^{[2]}\Qkm^{[1]}
\end{equation*}
where
\begin{align*}
\Qkm^{[4]}	&\overset{\eqref{eq:LongestWordPartialQkmPart}}{=} \Qkm_1,\\
\Qkm^{[3]}	&\overset{\phantom{\eqref{eq:LongestWordPartialQkmPart}}}{=} 
				\sum_{n \geq 0} \dfrac{\q^n}{\{2n\}_2!!} (q^4c_1)^n(q^2c_2)^nT_{21}(E_2)^{2n},\\
\Qkm^{[2]}	&\overset{\phantom{\eqref{eq:LongestWordPartialQkmPart}}}{=}	
				\sum_{n \geq 0} \dfrac{\Q{2}^n}{\{2n\}_1!!} (q^4c_1)^n(q^2c_2)^{2n}T_2(E_1)^{2n},\\
\Qkm^{[1]}	&\overset{\phantom{\eqref{eq:LongestWordPartialQkmPart}}}{=}	\Qkm_2.
\end{align*}
By Lemma \ref{lem:QkmRankOneAI} and \cite[8.26, (4)]{b-Jantzen96} we have $\ir{1}{\Qkm_{w_0^{\prime}}}=\Q{2}(q^4c_1)E_1\Qkm_{w_0^{\prime}}$.	
We want to show that
\begin{equation*}
\ir{2}{\Qkm_{w_0^{\prime}}}	=	\q (q^2c_2)E_2 \Qkm_{w_0^{\prime}}.
\end{equation*}
For each $i = 1,2,3,4$ let $\Qkm_{[i]} = K_2\Qkm^{[i]}K_2^{-1}$.
By property \eqref{eq:ir} of the skew derivative ${}_2r$ we have
\begin{equation} \label{AppEqBC: 2rXw0}
\ir{2}{\Qkm_{w_0^{\prime}}}	=	\Qkm_{[4]}\ir{2}{\Qkm^{[3]}}\Qkm^{[2]}\Qkm^{[1]} + \Qkm_{[4]}\Qkm_{[3]}\ir{2}{\Qkm^{[2]}}\Qkm^{[1]} + \Qkm_{[4]}\Qkm_{[3]}\Qkm_{[2]}\ir{2}{\Qkm^{[1]}}.
\end{equation}
Using Lemma \ref{AppLemBC: 2r} we obtain
\begin{align}
\ir{2}{\Qkm^{[3]}}	
	&\overset{\eqref{AppEqBC: 2r(T21E2)}}{=} q^{-2}\q\Q{2}(q^4c_1)(q^2c_2)E_1T_{21}(E_2)\Qkm^{[3]}, \label{AppEqBC:2rQkm3}\\
\ir{2}{\Qkm^{[2]}}
	&\overset{\eqref{AppEqBC: 2r(T2E1)}}{=} \q\Q{2}(q^4c_1)(q^2c_2)^2T_{21}(E_2)T_2(E_1)\Qkm^{[2]}, \label{AppEqBC:2rQkm2}\\
\ir{2}{\Qkm^{[1]}}
	&\overset{\eqref{eq:QkmRankOneAI}}{=} \q (q^2c_2) E_2\Qkm^{[1]}. \label{AppEqBC:2rQkm1}
\end{align}
By Lemma \ref{AppLemBC: E2comms} we have
\begin{align}
\Qkm_{[2]}E_2	&\overset{\eqref{AppEqBC: T2E1E2}}{=}	E_2\Qkm^{[2]}, \label{AppEqBC: Qkm2E2}\\
\Qkm_{[3]}E_2	&\overset{\eqref{AppEqBC: T21E2E2}}{=}	E_2\Qkm^{[3]}  
	- \Q{2}(q^4c_1)(q^2c_2)\Qkm_{[3]}T_{21}(E_2)T_2(E_1), \label{AppEqBC: Qkm3E2}\\
\Qkm_{[4]}E_2	&\overset{\eqref{AppEqBC: E1E2}}{=}		E_2\Qkm^{[4]}
	- q^{-2}\Q{2}(q^4c_1)\Qkm_{[4]}E_1T_{21}(E_2). \label{AppEqBC: Qkm4E2}
\end{align}
Using equations \eqref{AppEqBC: Qkm2E2}, \eqref{AppEqBC: Qkm3E2} and \eqref{AppEqBC: Qkm4E2}, and comparing with equations \eqref{AppEqBC:2rQkm3} and \eqref{AppEqBC:2rQkm2} we can rewrite the term $\Qkm_{[4]}\Qkm_{[3]}\Qkm_{[2]}\ir{2}{\Qkm^{[1]}}$ as
\begin{align*}
\Qkm_{[4]}\Qkm_{[3]}&\Qkm_{[2]}\ir{2}{\Qkm^{[1]}} \\ &\overset{\eqref{AppEqBC:2rQkm1}}{=} \q (q^2c_2) \Qkm_{[4]}\Qkm_{[3]}\Qkm_{[2]}E_2\Qkm^{[1]}\\	
&\overset{\phantom{\eqref{AppEqBC:2rQkm1}}}{=} \q (q^2c_2) E_2\Qkm_{w_0^{\prime}} - \Qkm_{[4]}\ir{2}{\Qkm^{[3]}}\Qkm^{[2]}\Qkm^{[1]} - \Qkm_{[4]}\Qkm_{[3]}\ir{2}{\Qkm^{[2]}}\Qkm^{[1]}.
\end{align*}
It hence follows from \eqref{AppEqBC: 2rXw0} that $\ir{2}{\Qkm_{w_0^{\prime}}} = \q (q^2c_2)E_2\Qkm_{w_0^{\prime}}$ as required.
\end{proof}
\endgroup

\providecommand{\bysame}{\leavevmode\hbox to3em{\hrulefill}\thinspace}
\providecommand{\MR}{\relax\ifhmode\unskip\space\fi MR }
\providecommand{\MRhref}[2]{%
  \href{http://www.ams.org/mathscinet-getitem?mr=#1}{#2}
}
\providecommand{\href}[2]{#2}


\end{document}